\documentclass[a4paper]{amsart}
\usepackage{amsmath,amsthm,amssymb,mathtools,mathrsfs}
\usepackage{bm}
\usepackage{tikz-cd}
\usepackage[all]{xy}
\usepackage{hyperref}
    \definecolor{urlcolor}{rgb}{0,0,0}
    \definecolor{linkcolor}{rgb}{.7,0.10,0.2}
    \definecolor{citecolor}{rgb}{.12,.54,.11}
\hypersetup{
      breaklinks=true, 
      colorlinks=true,
      urlcolor=urlcolor,
      linkcolor=linkcolor,
      citecolor=citecolor,
}

\usepackage{tikz-cd}
\usepackage{setspace}
\usepackage{xcolor}
\usepackage[shortlabels]{enumitem}
\usepackage[utf8]{inputenc}
\numberwithin{equation}{section}
\usepackage[top=2.3cm,
bottom=2.3cm,
left=2.3cm,
right=2.3cm]{geometry}

\usepackage[
style=alphabetic,
doi=false,
isbn=false,
url=false,
eprint=false,
block = space,
maxcitenames=100,
maxbibnames=100,
sorting=anyt]{biblatex} 

\newtheorem{theorem}{Theorem}[section]
\newtheorem{corollary}[theorem]{Corollary}
\newtheorem{proposition}[theorem]{Proposition}
\newtheorem{proposition-definition}[theorem]{Proposition-Definition}
\newtheorem{lemma}[theorem]{Lemma}

\newtheorem{conjecture}[theorem]{Conjecture}

\theoremstyle{definition} 
\newtheorem{definition}[theorem]{Definition}

\newtheorem{theorem-definition}[theorem]{Theorem-Definition}

\theoremstyle{remark} 
\newtheorem{remark}[theorem]{Remark}

\renewcommand{\emptyset}{\varnothing}
\renewcommand{\geq}{\geqslant}
\renewcommand{\leq}{\leqslant}
\renewcommand{\subset}{\subseteq}
\renewcommand{\setminus}{\smallsetminus}
\renewcommand{\tilde}{\widetilde}

\newcommand{\BD}{\mathbb{D}}

\newcommand{\BoA}{\mathbf{A}}
\newcommand{\BoC}{\mathbf{C}}

\newcommand{\BoG}{\mathbf{G}}

\newcommand{\BoN}{\mathbf{N}}
\newcommand{\BoP}{\mathbf{P}}
\newcommand{\BoQ}{\mathbf{Q}}

\newcommand{\BoZ}{\mathbf{Z}}

\newcommand{\CD}{\mathcal{D}}

\newcommand{\CF}{\mathcal{F}}
\newcommand{\CG}{\mathcal{G}}
\newcommand{\CH}{\mathcal{H}}

\newcommand{\CJ}{\mathcal{J}}

\newcommand{\CM}{\mathcal{M}}

\newcommand{\CO}{\mathcal{O}}
\newcommand{\CP}{\mathcal{P}}

\newcommand{\CT}{\mathcal{T}}

\newcommand{\CW}{\mathcal{W}}

\newcommand{\FM}{\mathfrak{M}}

\newcommand{\FP}{\mathfrak{P}}

\newcommand{\FX}{\mathfrak{X}}
\newcommand{\FY}{\mathfrak{Y}}

\newcommand{\SC}{\mathscr{C}}

\newcommand{\SF}{\mathscr{F}}
\newcommand{\SG}{\mathscr{G}}
\newcommand{\SH}{\mathscr{H}}

\newcommand{\SL}{\mathscr{L}}

\newcommand{\SP}{\mathscr{P}}
\newcommand{\SQ}{\mathscr{Q}}

\newcommand{\Aut}{\mathrm{Aut}}

\newcommand{\Bun}{\mathfrak{B}un}
\newcommand{\BPS}{\mathcal{BPS}}
\newcommand{\rmBPS}{\mathrm{BPS}}
\newcommand{\rmb}{\mathrm{b}}
\newcommand{\rmc}{\mathrm{c}}
\newcommand{\cl}{\mathrm{cl}}
\newcommand{\crit}{\mathrm{crit}}
\newcommand{\cms}{/\!\!/}
\newcommand{\codim}{\mathrm{codim}}
\newcommand{\Dol}{\mathrm{Dol}}
\newcommand{\Fg}{\mathfrak{g}}
\newcommand{\GL}{\mathrm{GL}}

\newcommand{\gr}{\mathrm{gr}}
\newcommand{\Gr}{\mathrm{Gr}}
\newcommand{\HO}{\mathrm{H}}
\newcommand{\Hom}{\mathrm{Hom}}
\newcommand{\IC}{\mathcal{IC}}
\newcommand{\id}{\mathrm{id}}
\newcommand{\IH}{\mathrm{IH}}

\newcommand{\Ind}{\mathrm{Ind}}
\newcommand{\Irr}{\mathrm{Irr}}
\newcommand{\Lie}{\mathrm{Lie}}
\newcommand{\Fl}{\mathfrak{l}}
\newcommand{\Mat}{\mathrm{Mat}}
\newcommand{\MHM}{\mathrm{MHM}}
\newcommand{\MMHM}{\mathrm{MMHM}}
\newcommand{\nil}{\mathrm{nil}}

\newcommand{\rmm}{\mathrm{m}}
\newcommand{\Perv}{\mathrm{Perv}}

\newcommand{\pt}{\mathrm{pt}}
\newcommand{\rank}{\mathrm{rank}}
\newcommand{\red}{\mathrm{red}}

\newcommand{\rmBM}{\mathrm{BM}}

\newcommand{\rat}{\mathbf{rat}}
\newcommand{\reg}{\mathrm{reg}}

\newcommand{\st}{\mathrm{st}}
\newcommand{\Spec}{\mathrm{Spec}}
\newcommand{\ssimp}{\mathrm{ss}}

\newcommand{\sst}{\mathrm{-sst}}

\newcommand{\rmSL}{\mathrm{SL}}

\newcommand{\supp}{\mathrm{supp}}
\newcommand{\Tan}{\mathrm{T}}

\newcommand{\Ft}{\mathfrak{t}}
\newcommand{\vir}{\mathrm{vir}}

\title{Cohomological integrality for symmetric quotient stacks}
\date{\today}

\author{Lucien Hennecart}
\address{Laboratoire Ami\'enois de Math\'ematique Fondamentale et Appliqu\'ee, CNRS UMR 7352, Universit\'e de Picardie Jules Verne, 33 rue Saint Leu, 80000 Amiens, France}\email{lucien.hennecart@u-picardie.fr}

\begin{document}

\begin{abstract}
In this paper, we establish the sheafified version of the cohomological integrality conjecture for stacks obtained as quotients of smooth affine symmetric algebraic varieties by reductive algebraic groups equipped with an invariant function. A central step is the definition of the BPS sheaf as a complex of monodromic mixed Hodge modules on the coarse moduli space. We prove the purity of the BPS sheaf when the situation arises from a smooth affine weakly symplectic algebraic variety with a weak moment map. This situation gives local models for $1$-Artin derived stacks with a self-dual cotangent complex. We then apply these results to prove a conjecture of Halpern-Leistner predicting the purity of the Borel--Moore homology of derived stacks with self-dual cotangent complex (in particular, $0$-shifted symplectic stacks) having a proper good moduli space. One striking application is a proof of the purity of the Borel--Moore homology of the moduli stack of  principal Higgs bundles over a smooth projective curve for a reductive group $G$, extending the case of $G=\GL_n$ to arbitrary reductive groups.
\end{abstract}

\maketitle

\setcounter{tocdepth}{1}
\tableofcontents

\section{Introduction}

We continue the study of the topology and enumerative geometry of moduli stacks initiated in \cite{hennecart2024cohomological}, where we established a \emph{cohomological integrality isomorphism} for quotient stacks of the form $V/G$ where $V$ is a finite-dimensional representation of  a reductive group $G$. The main result of \cite{hennecart2024cohomological} is recalled in \S\ref{section:reminderabsolute}, Theorem~\ref{theorem:abscohintreminder}. It provides a decomposition of the cohomology $\HO^*(V/G,\BoQ)$ in terms of parabolic inductions, reminiscent of Springer theory (see for example \cite{lusztig1984intersection}, \cite[Theorem~1]{gunningham2018generalized}), when $V$ is a symmetric representation of $G$. This study is not only of representation-theoretic relevance but also of geometric interest, with applications to understand smooth, $0$-shifted and $(-1)$-shifted symplectic stacks \cite{pantev2013shifted}, as demonstrated in this paper and more recently in \cite{bu2025cohomology}.

Integrality conjectures and results originate in the work of Kontsevich and Soibelman \cite[\S7]{kontsevich2008stability}, \cite[\S6]{kontsevich2011cohomological} and deal with some kinds of factorizations of generating series which give rise to a precise understanding of meaningful enumerative invariants: Donaldson--Thomas (DT) invariants (introduced in \cite{donaldson1998gauge} for stable sheaves and in \cite{joyce2012theory} for semistable sheaves), Bogomol'nyi--Prasad--Sommerfield (BPS) invariants (in relation to the algebra of BPS states \cite{harvey1998algebras}), and their cohomologically and motivically refined avatars \cite{morrison2012motivic,davison2020cohomological,davison2015motivic}. The work of Efimov \cite{efimov2012cohomological} (concerning symmetric quivers without potential) can be seen as a categorification of the cohomological integrality of Kontsevich and Soibelman. Indeed, a factorization of the generating series involving the BPS invariants is turned into an isomorphism between a vector space and the symmetric power of a graded vector space with finite-dimensional pieces whose dimensions give the refined BPS invariants \cite[Theorem~1.1]{efimov2012cohomological}.

Meinhardt and Reineke \cite{meinhardt2019donaldson} pushed further the study of cohomological Hall algebras (CoHAs) of symmetric quivers without potentials. They provide a \emph{motivic} refinement (using Grothendieck rings of mixed Hodge structures) of the cohomological integrality isomorphism of Efimov. By using framed representations, they determined precisely the BPS invariants of a symmetric quiver: they are given by the intersection cohomology of the good moduli space \cite[Theorem~1.1]{meinhardt2019donaldson}.

Davison and Meinhardt then considered and proved \emph{sheafified} versions of cohomological integrality isomorphisms for quivers with potential \cite{davison2020cohomological}. Instead of proving directly an isomorphism at the level of vector spaces, they define a \emph{sheafified} cohomological Hall algebra (an algebra object in a category of monodromic mixed Hodge modules) and prove a cohomological integrality result which resembles a PBW isomorphism in this category \cite[Theorem~A]{davison2020cohomological}. The idea is to take advantage of the good moduli space map for representations of the Jacobi algebra of a quiver with potential, which is a canonical map from the stack of representations to an algebraic scheme/space.

All formerly known cohomological integrality results involve cohomological Hall algebras, as defined and studied in the non-exhaustive list of references \cite{kontsevich2011cohomological,schiffmann2012hall,schiffmann2020cohomological,davison2020cohomological,yang2018cohomological,kapranov2022cohomological,davison2022bps,kinjo2024cohomological}. Cohomological Hall algebras are powerful objects to study the topology and enumerative geometry of ubiquitous moduli spaces in algebraic geometry, such as the moduli spaces of vector bundles over smooth projective curves, the moduli spaces of Higgs bundles and local systems in relation with the $\mathrm{P}=\mathrm{W}$ conjecture, and the moduli spaces of quiver representations in connection with representation theory (Yangians and quantum groups \cite{maulik2019quantum}). Recent works have shown how this technology can be used effectively to solve algebro-geometric problems: nonabelian Hodge isomorphisms for stacks, decomposition of the cohomology of Nakajima quiver varieties, positivity of cuspidal polynomials of quivers, $\mathrm{P}=\mathrm{W}$-conjecture \cite{maulikp,hausel2022p} and identification of the Maulik--Okounkov Yangian, see in particular \cite{davison2022bps,davison2023bps,hausel2022p,botta2023okounkov,schiffmann2023cohomological,mellit2023coherent} for a selection. The formalism we start to develop here is expected to help understand analogous questions for arbitrary reductive groups.

Classical cohomological integrality isomorphisms give equalities involving plethystic exponentials \cite[\S6.2, Corollary 3]{kontsevich2011cohomological}, while their categorification involve symmetric powers of a vector space \cite[Theorem 1.1]{efimov2012cohomological}. Both are type A phenomena, in the sense that they involve symmetric groups which are the Weyl groups of type A algebraic groups such as $\mathrm{GL}_n$, $\mathrm{SL}_n$ and $\mathrm{PGL}_n$. Cohomological integrality in situations involving more general reductive groups also involve their Weyl groups \cite[Theorem~1.1]{hennecart2024cohomological}. Understanding the mechanics of Weyl groups in this context and their behaviour with respect to \'etale slices is a technical but central part of this work, see \S\ref{subsection:Weylgroups}. When working with quivers, this discussion is well-understood and all compatibilities of Weyl groups follow from the monoidality of some restriction functors.

The direction of the categorification of cohomological Hall algebras initiated in \cite{porta2023two} seems very promising in order to define categorical enumerative invariants, as shown by the recent series of papers of P\u adurariu--Toda, e.g. \cite{puadurariu2023categorical}. The formalism of \cite{porta2023two} can be applied to the more general situation of Hamiltonian reductions of symplectic vector spaces or more generally $0$-shifted symplectic stacks. It is worth to study in the perspective of categorical enumerative invariants and categorical or noncommutative resolutions of symplectic singularities \cite{van2004non,lunts2010categorical,kuznetsov2015categorical}.

In this paper, rather than going up on the ladder of categorification, we extend the scope of cohomological Donaldson--Thomas theory to a larger class of stacks than that involved when studying sheaves on manifolds or representations of quivers. In particular, we develop the formalism to study stacks which do not arise as stacks of objects in some Abelian category. This study was in particular made necessary by the desire to understand better the stacks of principal bundles, including stacks of Higgs bundles, over smooth projective curves. In this more general setting, parabolic inductions (\S\ref{section:parabolicinduction}) replace the CoHA multiplication. Stacks of the form $V/G$ were studied from a different perspective in \cite{vspenko2021semi,vspenko2017non} where, in rough terms, the coherent derived category of $V/G$ together with semi-orthogonal decompositions were used to construct categorical resolutions of the GIT quotient $V\cms G$. This is related to the connection between the derived category of a stack and the GIT quotient \cite{halpern2015derived} and to the noncommutative minimal model programme \cite{halpern2023noncommutative}.

We briefly describe the contents of this paper. We first prove a sheaf-theoretic cohomological integrality isomorphism for a class of quotient stacks $X/G$ where $G$ is a reductive group and $X$ is a symmetric smooth affine $G$-variety (Theorem~\ref{theorem:sheafifiedcohintintro}). When $X=V$ is a finite-dimensional symmetric representation of $G$, we obtain a sheaf-theoretic version of \cite[Theorem 1.1]{hennecart2024cohomological}. The proof of Theorem~\ref{theorem:sheafifiedcohintintro} relies heavily on \cite{hennecart2024cohomological} via a local-global principle involving \'etale slices (\S\ref{section:localneighbourhoodsmoothstacks}). Applying vanishing cycle functors immediately gives cohomological integrality isomorphisms for smooth affine critical $G$-varieties $(X,f)$ (a $G$-equivariant Landau--Ginzburg model), see Theorem~\ref{theorem:critcohint}. Here, $f\colon X\rightarrow\BoC$ is a $G$-invariant function. Although this seems to be a very mild generalization of Theorem~\ref{theorem:sheafifiedcohintintro}, it has far-reaching consequences obtained via (cohomological) dimensional reduction (\cite{davison2017critical} for the version relevant for us, see also \cite{kinjo2022dimensional,davison2022deformed}, and \cite[Proposition~6]{kontsevich2011cohomological} for the first occurrence of cohomological dimensional reduction). In particular, it can be applied to the study of $0$-shifted symplectic stacks \S\ref{subsubsection:purityBPSweak} and more generally of derived stacks with self-dual cotangent bundle. Our main result there is that for such stacks, the pushforward of the dualizing complex of the stack to the good moduli space is a pure complex of mixed Hodge modules. We use this property to prove a conjecture of Halpern-Leistner \cite[Conjecture 4.4]{halpern2015theta}: the Borel--Moore homology of $0$-shifted symplectic stacks, or more generally, $1$-Artin derived stacks with self-dual cotangent bundle having a proper good moduli space is \emph{pure} in the sense of Hodge theory (Theorem~\ref{theorem:purityintroduction}). Note that by \cite{ahlqvist2023good}, we do not have to pass to the classical truncation to consider the good moduli spaces. This generalizes one of the main results of \cite{davison2021purity}, where a proof of Halpern-Leistner's conjecture for stacks arising from $2$-Calabi--Yau categories is given. Note that the only place where derived structures appear crucially in this paper are for local models of derived stacks with self-dual cotangent complexes, and such local models can then be dealt with classically.

This paper contains two appendices. The goal of Appendix~\ref{appendix:MHM} is to give a minimal account regarding monodromic mixed Hodge modules following \cite[\S2]{davison2020cohomological}. Appendix~\ref{appendix:examples} exemplifies Theorem~\ref{theorem:sheafifiedcohintintro} and echoes \cite[\S6 Examples]{hennecart2024cohomological}.

\subsection{Main results}
In the remainder of the introduction, we present succinctly our main results and some of their applications.
\subsubsection{Semisimplicity}

Let $X$ be a smooth affine algebraic variety with an action of a complex reductive group $G$. We let $\pi\colon X/G\rightarrow X\cms G=\Spec(\BoC[X]^G)$ be the good moduli space map.

\begin{theorem}[= Proposition~\ref{proposition:semisimplicitypistar}]
\label{theorem:puritysmoothquotient}
 The complex of mixed Hodge modules $\pi_*\underline{\BoQ}_{X/G}\in\CD^+(\MMHM(X\cms G))$ is pure of weight zero, and therefore semisimple.
\end{theorem}

The strategy of proof is to first establish the \emph{approximation by proper maps} property for $\pi'\colon X/(G\times\BoC^*)\rightarrow X\cms G$ where $\BoC^*$ acts trivially on $X$ (\S\ref{section:approximationpropermaps}, Proposition~\ref{proposition:Gprimeapproachable}). We then relate the pushforwards $\pi_*\underline{\BoQ}_{X/G}$ and $\pi'_*\underline{\BoQ}_{X/(G\times\BoC^*)}$. Theorem~\ref{theorem:puritysmoothquotient} implies that $\pi_*\underline{\BoQ}_{X/G}\in \MMHM^{\BoZ}(X\cms G)$ is a cohomologically graded mixed Hodge module, that is $\pi_*\underline{\BoQ}_{X/G}\cong \bigoplus_{i\in\BoZ}\underline{\CH}^i(\pi_*\underline{\BoQ}_{X/G})[-i]$ decomposes as a direct sum of its shifted perverse cohomologies (\S\ref{subsection:cohgradedmhm}).

Via \'etale slices \cite{alper2020luna}, one can globalize Theorem \ref{theorem:puritysmoothquotient} to more general stacks.
\begin{corollary}
\label{corollary:puritysmoothstacks}
 Let $\FM$ be a smooth Artin stack of finite type having a good moduli space $\pi\colon\FM\rightarrow\CM$ with affine diagonal. Then $\pi_*\underline{\BoQ}_{\FM}$ is a pure weight zero complex of mixed Hodge modules. In particular, if $\CM$ is proper, the mixed Hodge structure on $\HO^*(\FM,\BoQ)$ is pure of weight zero.
\end{corollary}
By specializing Corollary \ref{corollary:puritysmoothstacks} to the stack of semistable $G$-bundles over a smooth projective curve, when $G$ is a reductive group, we obtain a new proof of the following.

\begin{corollary}
\label{corollary:purityGbunintro}
 The moduli stack of semistable principal $G$-bundles over a smooth projective curve has pure weight zero cohomology.
\end{corollary}
\begin{proof}
 The moduli space of semistable $G$-bundles over a smooth projective curve is projective \cite[Theorem 5.9]{ramanathan1996moduli}. We can then apply Corollary~\ref{corollary:puritysmoothstacks}
\end{proof}

Our proof of Corollary~\ref{corollary:purityGbunintro}, using a version of approximation by proper maps, is similar in spirit to the proof appearing in \cite{kinjo2024decomposition}. Similarly, we can recover \cite[Corollary 5.2]{kinjo2024decomposition} regarding the purity of the cohomology of stacks of semistable sheaves on (local) del Pezzo surfaces. Our version of approximation by proper maps gives \emph{virtually small} morphisms in the sense of \cite{meinhardt2019donaldson}, see Definition~\ref{definition:appm}. The author will exploit this fact in a forthcoming work to give quantitative information on the perverse filtration on the cohomology of the stack arising from the good moduli space map.

Using the Harder--Narasimhan stratification, we are then able to recover a now classical and important result \cite[Proposition 4.4]{teleman1998borel}, see \cite[Corollary 3.3.2]{heinloth2010cohomology} for the version of this result over finite fields.

\begin{corollary}
 The moduli stack of $G$-bundles over a smooth projective curve has pure cohomology.
\end{corollary}
Our proof of this statement is valid in both characteristic zero (for which we use the theory of monodromic mixed Hodge modules, as in \cite{davison2020cohomological} and introduced by Saito, see e.g. \cite{saito1990mixed}) and in positive characteristic (for which we use $\ell$-adic sheaves as in \cite{kiehl2013weil}). They key point is that the approximation by proper maps for good moduli spaces $V/G\rightarrow V\cms G$ works in both settings (although we only explicit the characteristic zero situation). Of course, more is known about the cohomology of the stack $\Bun_G$, in particular it is known to be generated by tautological classes \cite{atiyah1983yang} (in characteristic zero) and \cite{heinloth2010cohomology} in positive characteristics, and the (absence of) relations between them is also known \cite{heinloth2010cohomology}, i.e. the ring $\HO^*(\Bun_G)$ is completely explicit.

\subsubsection{Sheafified cohomological integrality for symmetric quotient stacks}
\label{subsubsection:sheafcohintsym}

Let $G$ be a reductive group and $X$ a smooth algebraic variety with a $G$-action which is symmetric (for any $x\in X$ such that the orbit $G\cdot x$ is closed, $\Tan_xX$ and $\Tan^*_xX$ have the same sets of weights counted with multiplicities for a maximal torus $T_x\subset G_x$ inside the stabilizer $G_x$ of $x$, Definition~\ref{definition:symmetricrepresentation} -- this is equivalent to \cite[Definition 2.3]{hennecart2024cohomological} when $X=V$ is a representation of $G$). We let $T$ be a maximal torus of $G$. For any cocharacter $\lambda\in X_*(T)$, we have a Levi subgroup $G^{\lambda}\subset G$ (fixed point set under $\lambda$-conjugation) and a smooth symmetric $G^{\lambda}$-variety $X^{\lambda}=\bigsqcup_{\alpha\in \pi_0(X^{\lambda})}X^{\lambda}_{\alpha}$ (fixed points under the $\lambda$-action), which we may decompose into connected components. We have the good moduli space map $\pi_{(\lambda,\alpha)}\colon X^{\lambda}_{\alpha}/G^{\lambda}\rightarrow X^{\lambda}_{\alpha}\cms G^{\lambda}\coloneqq \Spec(\BoC[X^{\lambda}_{\alpha}])^{G^{\lambda}}$. We denote by $\SQ_{X,G}$ the set of pairs $(\lambda,\alpha)$ with $\lambda\in X_*(T)$ and $\alpha\in \pi_0(X^{\lambda})$.

We let $\SP_{X,G}$ be the set of equivalence classes of pairs $(\lambda,\alpha)\in\SQ_{X,G}$ of a cocharacter $\lambda$ of $T$ with a connected component $\alpha\in\pi_0(X^{\lambda})$ with respect to the relation
\[
 (\lambda,\alpha)\sim(\mu,\beta)\iff\left\{\begin{aligned}
                           &X_{\alpha}^{\lambda}=X_{\beta}^{\mu}\quad\text{ as subvarieties of $X$}\\
                           &G^{\lambda}=G^{\mu}\,.
                          \end{aligned}
\right.
\]
For a pair $(\lambda,\alpha)$, we let $\overline{(\lambda,\alpha)}$ be its class in $\SP_{X,G}$. The Weyl group $W$ of $G$ naturally acts on the set of pairs $(\lambda,\alpha)$, and this action descends to $\SP_{X,G}$. For any pair $(\lambda,\alpha)$, there is a subgroup $W_{(\lambda,\alpha)}\coloneqq \{w\in W\mid w\cdot\overline{(\lambda,\alpha)}=\overline{(\lambda,\alpha)}\}\subset W$ (\S\ref{subsection:Weylgroups}). The Weyl group $W^{\lambda}\coloneqq N_{G^{\lambda}}(T)/T$ of $G^{\lambda}$ is a normal subgroup of $W_{(\lambda,\alpha)}$ (Lemma~\ref{lemma:weylgroupsubgroup}). We denote by $\overline{W}_{(\lambda,\alpha)}\coloneqq W_{(\lambda,\alpha)}/W^{\lambda}$ the quotient, which is a subgroup of the relative Weyl group of $G^{\lambda}\subset G$ (\S\ref{subsubsection:relWeylsubvar}). We let $G_{(\lambda,\alpha)}=\ker(G^{\lambda}\rightarrow\Aut(X^{\lambda}_{\alpha}))\cap Z(G^{\lambda})$, the subgroup of the center of $G^{\lambda}$ acting trivially on $X^{\lambda}_{\alpha}$. Its neutral connected component $G_{(\lambda,\alpha)}^{\circ}$ is a torus. There is a character $\varepsilon_{X,(\lambda,\alpha)}\colon \overline{W}_{(\lambda,\alpha)}\rightarrow\{\pm1\}$ (Lemma~\ref{lemma:character}). We let $\underline{\BoQ}_{X/G}^{\vir}\coloneqq\underline{\BoQ}_{X/G}\otimes\SL^{-\frac{\dim X-\dim G}{2}}$ be the perversely twisted constant sheaf, where $\SL$ is the Tate twist. For any $(\lambda,\alpha)\in\SQ_{X,G}$, we have a $\overline{W}_{(\lambda,\alpha)}$-invariant finite map $\imath_{(\lambda,\alpha)}\colon X^{\lambda}_{\alpha}\cms G^{\lambda}\rightarrow X\cms G$ (Lemma~\ref{lemma:finitemap}).

\begin{theorem}[= Theorem~\ref{theorem:sheafcohint}]
\label{theorem:sheafifiedcohintintro}
 Let $X$ be a smooth symmetric affine $G$-variety. For any $(\lambda,\alpha)\in \SQ_{X,G}$, there exists a cohomologically graded mixed Hodge module (\S\ref{subsection:cohgradedmhm}) $\underline{\CP}_{(\lambda,\alpha)}\subset (\pi_{(\lambda,\alpha)})_*\underline{\BoQ}_{X^{\lambda}_{\alpha}/G^{\lambda}}^{\vir}\in\MMHM^{\BoZ}(X^{\lambda}_{\alpha}\cms G^{\lambda})$ such that the map
 \begin{equation}
 \label{equation:cohintiso}
\bigoplus_{\tilde{(\lambda,\alpha)}\in\SP_{X,G}/W}((\imath_{(\lambda,\alpha)})_*\underline{\CP}_{(\lambda,\alpha)}\otimes\HO^*(\pt/G_{(\lambda,\alpha)}))^{\varepsilon_{X,(\lambda,\alpha)}}\rightarrow \pi_*\underline{\BoQ}_{X/G}^{\vir}
 \end{equation}
in $\CD^{+}(\MMHM(X^{\lambda}_{\alpha}\cms G^{\lambda}))$ induced by the sheafified induction morphisms \eqref{equation:sheafifiedinduction} is an isomorphism. The quantity $((\imath_{(\lambda,\alpha)})_*\underline{\CP}_{(\lambda,\alpha)}\otimes\HO^*(\pt/G_{(\lambda,\alpha)}))^{\varepsilon_{X,(\lambda,\alpha)}}$ is the isotypic component of the character $\varepsilon_{X,(\lambda,\alpha)}$ for the $\overline{W}_{(\lambda,\alpha)}$-action on $((\imath_{(\lambda,\alpha)})_*\underline{\CP}_{(\lambda,\alpha)}\otimes\HO^*(\pt/G_{(\lambda,\alpha)}))\in\MMHM^{\BoZ}(X\cms G)$, an object of an Abelian category. Moreover, $\underline{\CP}_{(\lambda,\alpha)}$ only depends on the $G^{\lambda}$-variety $X^{\lambda}_{\alpha}$ up to isomorphism, and is uniquely determined by the existence of cohomological integrality isomorphisms \eqref{equation:cohintiso} for all pairs $(X^{\mu}_{\beta},G^{\mu})$, $(\mu,\beta)\in\SQ_{X,G}$.
\end{theorem}

The direct sum of Theorem~\ref{theorem:sheafifiedcohintintro} runs over $W$-orbits in $\SP_{X,G}$. For each orbit, we choose a representative $(\lambda,\alpha)$ with $\lambda\in X_*(T)$ and $\alpha\in\pi_0(X^{\lambda})$. The cohomological integrality isomorphism of Theorem~\ref{theorem:sheafifiedcohintintro} depends on such a choice but in a non-essential way.

We refer to \ref{subsection:bpsmhm} for the definition of the cohomologically graded mixed Hodge modules $\underline{\CP}_{(\lambda,\alpha)}$. In the introduction, we use the notation $\underline{\CP}_{(\lambda,\alpha)}$ for them to emphasize their connection with the vector spaces $\CP_{\lambda}$ defined by \cite[Theorem~1.1]{hennecart2024cohomological} when $X=V$ is a representation of $G$. Namely, in this case, $\CP_{\lambda}\cong \HO^*(X^{\lambda}_{\alpha}\cms G^{\lambda},\underline{\CP}_{\lambda})$. Later, we may denote them $\underline{\BPS}_{X,(\lambda,\alpha)}$ to insist on the similarity with the BPS sheaf for quivers introduced in \cite{meinhardt2019donaldson,davison2020cohomological}. However, a crucial difference is that the complexes $\underline{\BPS}_{X,(\lambda,\alpha)}$ may not be concentrated in a single perverse degree a priori.

\begin{proposition}
\label{proposition:boundedness}
 The cohomologically graded mixed Hodge modules $\underline{\CP}_{(\lambda,\alpha)}$ given by Theorem~\ref{theorem:sheafifiedcohintintro} are concentrated in finitely many cohomological degrees.
\end{proposition}

\subsubsection{Sheafified cohomological integrality for critical symmetric quotient stacks}
\label{subsubsection:sheafcohintcritsym}

Let $G$ be a reductive group and $X$ a smooth affine $G$-variety. We let $f\colon X\rightarrow\BoC$ be a $G$-invariant function. The pair $(X,f)$ is called a \emph{critical $G$-variety}. For any $(\lambda,\alpha)\in \SQ_{X,G}$, the restriction $f_{(\lambda,\alpha)}\colon X^{\lambda}_{\alpha}\rightarrow\BoC$ of $f$ to $X^{\lambda}_{\alpha}$ provides a critical $G^{\lambda}$-variety $(X^{\lambda}_{\alpha},f_{
(\lambda,\alpha)})$. We also denote by $f_{(\lambda,\alpha)}\colon X^{\lambda}_{\alpha}\cms G^{\lambda}\rightarrow\BoC$ the regular function induced on the good moduli space.

We define $\underline{\CH}_{X,G,f}\coloneqq \pi_*\varphi^p_{f}\underline{\BoQ}_{X/G}^{\vir}\in \MMHM^{\BoZ}(X\cms G)$ (we use \ref{theorem:puritysmoothquotient} and the fact that the perversely twisted vanishing cycle functor $\varphi_f^p$ preserves cohomologically graded mixed Hodge module as it is perverse $t$-exact) and $\CH_{X,G,f}\coloneqq \HO^*(X\cms G,\underline{\CH}_{X,G,f})$, a vector space.

\begin{theorem} (= Theorem~\ref{theorem:cohintcriticallocus})
\label{theorem:critcohint}
 We assume that $X$ is a smooth symmetric $G$-variety. For $(\lambda,\alpha)\in\SQ_{X,G}$, we let $\underline{\BPS}_{X^{\lambda}_{\alpha},f_{(\lambda,\alpha)}}\coloneqq\varphi_{f_{(\lambda,\alpha)}}^p\underline{\CP}_{(\lambda,\alpha)}\in\MMHM^{\BoZ}(X^{\lambda}_{\alpha}\cms G^{\lambda})$ be the perverse exact vanishing cycle functor applied to the cohomologically graded monodromic mixed Hodge modules $\underline{\CP}_{(\lambda,\alpha)}$ given by Theorem~\ref{theorem:sheafifiedcohintintro}. Then, we have an isomorphism
 \[
  \bigoplus_{\tilde{(\lambda,\alpha)}\in\SP_{X,G}/W}((\imath_{(\lambda,\alpha)})_*\underline{\BPS}_{X^{\lambda}_{\alpha},f_{(\lambda,\alpha)}}\otimes\HO^*(\pt/G_{(\lambda,\alpha)}))^{\varepsilon_{X,(\lambda,\alpha)}}\rightarrow\underline{\CH}_{X,G,f}\,.
 \]
\end{theorem}

\begin{corollary}
\label{corollary:abscohintcrit}
 We let $\rmBPS_{X^{\lambda}_{\alpha},f_{(\lambda,\alpha)}}\coloneqq\HO^*(X^{\lambda}_{\alpha}\cms G^{\lambda},\underline{\BPS}_{X^{\lambda}_{\alpha},f_{(\lambda,\alpha)}})$. Then, we have an isomorphism
\[
  \bigoplus_{\tilde{(\lambda,\alpha)}\in\SP_{X,G}/W}(\rmBPS_{X^{\lambda}_{\alpha},f_{(\lambda,\alpha)}}\otimes\HO^*(\pt/G_{(\lambda,\alpha)}))^{\varepsilon_{X,(\lambda,\alpha)}}\rightarrow\CH_{X,G,f}\,.
\]
\end{corollary}

The convention on the direct sums in Theorem~\ref{theorem:critcohint} and Corollary~\ref{corollary:abscohintcrit} is the same as for Theorem~\ref{theorem:sheafifiedcohintintro}.

For the critical cohomology, in contrast to the situation of $V/G$ presented in \S\ref{subsubsection:sheafcohintsym}, we first establish the sheafified version of the cohomological integrality and then, by taking derived global sections, the absolute version.

\subsubsection{Purity of the BPS sheaf for weak Hamiltonian reductions}
\label{subsubsection:purityBPSweak}
Let $X$ be a weakly Hamiltonian smooth affine $G$-variety (\cite[Definition~4.2.1]{halpern2020derived}, see also Definition~\ref{definition:weakmomentmap}). This means that there is a $G$-equivariant isomorphism of vector bundles $\psi\colon \Tan^*X\rightarrow\Tan X$ (we say that $X$ is \emph{weakly symplectic}), a $G$-equivariant map $\mu\colon X\rightarrow\Fg^*$ to $\Fg^*\coloneqq\Lie(G)^*$ considered with the coadjoint $G$-action and a $G$-equivariant isomorphism $\phi\colon X\times \Fg\rightarrow X\times \Fg$ such that the square
\[\begin{tikzcd}
	{X\times\Fg} & {\Tan^*X} \\
	{X\times\Fg} & {\Tan X}
	\arrow["{\mathrm{d}\mu}", from=1-1, to=1-2]
	\arrow["\phi"', from=1-1, to=2-1]
	\arrow["\psi", from=1-2, to=2-2]
	\arrow["a", from=2-1, to=2-2]
\end{tikzcd}\]
commutes after restriction over $\mu^{-1}(0)$, where $a\colon X\times \Fg\rightarrow\Tan X$, $(x,\xi)\mapsto (x,X_{\xi}(x))$ (where $X_{\xi}$ is the vector field on $X$ generated by the action of $\xi\in\Fg$) is the infinitesimal action of $G$ on $X$ and $\mathrm{d}\mu(x,\xi)\coloneqq \mathrm{d}\langle\mu(-),\xi\rangle(x)$ is the differential at $x$ of the function $f\colon x\mapsto \langle\mu(x),\xi\rangle$ on $X$ obtained by contraction of the map $\mu$. We call $\mu$ a \emph{weak moment map}. The function $f\colon \tilde{X}\coloneqq X\times\Fg\rightarrow\BoC$, $(x,\xi)\mapsto \langle\mu(x),\xi\rangle$ gives a critical affine $G$-variety $(\tilde{X},f)$. Moreover, by weak symplecticity, $X$ is a symmetric $G$-variety and therefore so is $\tilde{X}$.

The following lemma is a consequence of the \emph{factorization property} of the DT sheaf of the critical affine $G$-variety $(\tilde{X},f)$ (Lemma~\ref{lemma:factorizationproperty}). Recall that for $\alpha\in \pi_0(X)$, we define $G_{(0,\alpha)}=\ker(G\rightarrow\Aut(X_{\alpha}))\cap Z(G)$. We let $\Fg_{(0,\alpha)}\coloneqq\Lie(G_{(0,\alpha)})$ (in the sequel, we may assume, without loss of generality, that $X$ is connected, and therefore write $G_0$ and $\Fg_0$).

\begin{lemma}[Support lemma = Lemma~\ref{lemma:sectionsupport}]
\label{lemma:supportlemmaintroduction}
 The BPS sheaf $\underline{\BPS}_{\tilde{X},f}$ is supported on $X\cms G\times\Fg_{0}\subset \tilde{X}\cms G$. Moreover,
 \[
  \underline{\BPS}_{\tilde{X},f}\cong \underline{\BPS}_{X,f}^{\red}\boxtimes\underline{\BoQ}_{\Fg_{0}}
 \]
 where $\underline{\BPS}_{X,f}^{\red}\cong \imath^*\underline{\BPS}_{\tilde{X},f}$ is the dimensionally reduced BPS sheaf, where $\imath\colon X\cms G\rightarrow\tilde{X}\cms G$ is the natural closed immersion.
\end{lemma}

The support lemma can be used to prove a crucial purity result.
\begin{theorem}[= Theorem~\ref{theorem:purityBPSsheaf}]
The BPS sheaf $\underline{\BPS}_{\tilde{X},f}$ is a pure weight zero cohomologically graded mixed Hodge module on $\tilde{X}\cms G$.
\end{theorem}
The purity of $\underline{\BPS}_{\tilde{X},f}$ and $\underline{\BPS}_{X,f}^{\red}$ are equivalent by Lemma~\ref{lemma:supportlemmaintroduction}.

\begin{corollary}
Let $V$ be a representation of a reductive group $G$. Let $\mu\colon \Tan^*V\rightarrow\Fg^*$ be the moment map. Then, $\HO^{\rmBM}_{*,G}(\mu^{-1}(0),\BoQ)$ has pure weight $0$ mixed Hodge structure.
\end{corollary}
This corollary generalizes the purity of the Borel--Moore homology of the stack of representations of preprojective algebras \cite[Theorem A]{davison2016integrality}.

In particular, we may apply this corollary to the commuting stack $\SC(\Fg)=\{(x,y)\in\Fg^2\mid [x,y]=0\}$ of a reductive Lie algebra $\Fg$ and obtain a result that does not seem to appear in the literature (except in type A).

\begin{corollary}
 Let $\Fg$ be a reductive Lie algebra. Then, the $G$-equivariant Borel--Moore homology $\HO^{\rmBM}_{*,G}(\SC(\Fg),\BoQ)$ carries a pure mixed Hodge structure.
\end{corollary}

\subsubsection{A purity conjecture of Halpern-Leistner}

The following theorem is our answer to \cite[Conjecture 4.4]{halpern2015theta}.
\begin{theorem}[= Theorem~\ref{theorem:HLconjecture}]
\label{theorem:purityintroduction}
 Let $\FM$ be a $1$-Artin $0$-shifted symplectic stack. We assume that $\FM$ is of finite type and has a proper good moduli space with affine diagonal. Then, the Hodge structure on $\HO^{\rmBM}_*(\FM,\BoQ)$ is pure of weight zero. More generally, if $\FM$ is a derived algebraic stack with proper good moduli space having affine diagonal and such that $\mathbb{L}_{\FM}\cong\mathbb{L}_{\FM}^{\vee}$, then $\HO^{\rmBM}_*(\FM,\BoQ)$ carries a pure weight zero mixed Hodge structure.
\end{theorem}

This theorem is a generalization of \cite[Theorem 6.10]{davison2021purity}, which provides a proof of Halpern-Leistner's conjecture for stacks arising from $2$-Calabi--Yau categories. The purity was one of the ingredients to obtain precise structural results regarding the cohomological Hall algebras of $2$-Calabi--Yau categories, which involve in particular generalized Kac--Moody algebras, \cite{davison2022bps,davison2023bps}. The strategy of proof for stacks arising from $2$-Calabi--Yau categories is outlined in \cite[\S1.4]{davison2020bps}.

\subsection{Connection to other works}
After this paper has been made available on the arXiv, work of Bu, Davison, Ib{\'a}{\~n}ez N{\'u}{\~n}ez, Kinjo and P{\u{a}}durariu \cite{bu2025cohomology} dealing with cohomological integrality for orthogonal stacks appeared. These are the stacks which locally look like (via \'etale slices) quotient stacks $V/G$ where $G$ is an \emph{orthogonal} representation of $G$ (a representation of $G$ admitting an invariant non-degenerate symmetric bilinear form). The orthogonality condition is strictly stronger than symmetricity we use in the present paper. For example, the natural representation $\BoC^2$ of $\rmSL_2(\BoC)$ is symmetric but not orthogonal. Indeed, a nondegerate symmetric bilinear form on $\BoC^2$ induces a non-zero quadratic function, and there are no non-zero $\rmSL_2(\BoC)$-invariant functions on $\BoC^2$. See Lemma~\ref{lemma:mat2nnotorthogonal} for a larger class of non-orthogonal representations. For orthogonal stacks, the authors prove the generalization of Theorems~\ref{theorem:sheafifiedcohintintro}, \ref{theorem:critcohint} (Theorems~1.2.7 and 1.2.12 in \cite{bu2025cohomology} respectively) to general stacks having a good moduli space and satisfying the conditions (i)-(iv) in \cite[\S1.2.5]{bu2025cohomology}, in addition to being orthogonal. They also prove Conjecture~\ref{conjecture:BPSsheavesnofunction} in the orthogonal case (and its natural generalization to this class of stacks), which is an identification of the complexes of mixed Hodge modules $\underline{\BPS}_{X,(\lambda,\alpha)}$ with the intersection cohomology of the coarse moduli space $X^{\lambda}_{\alpha}\cms G^{\lambda}$ under a stability constraint on $(\lambda,\alpha)$. The example in \S\ref{subsection:examplemat2n} shows that Conjecture~\ref{conjecture:BPSsheavesnofunction} also holds for the non-orthogonal representations $\Mat_{2\times n}(\BoC)$ ($n$ odd) of $\rmSL_2(\BoC)$, suggesting it should hold for general symmetric representations of reductive groups. The strategy of \cite{bu2025cohomology} reduces Conjecture~\ref{conjecture:BPSsheavesnofunction} to the case of representations of reductive groups having no degree $2$ homogeneous invariant functions, which should lead to further developments.

Combining the results of \cite{bu2025cohomology} and \cite{hennecart2024cohomological}, it is possible to extract a very efficient algorithm to compute the intersection cohomology of good moduli space $V\cms G$ of the affine GIT quotient of an orthogonal representation of a reductive group $G$. While less precise, our approach gives a little more flexibility on the class of stacks considered, which are only assumed to be symmetric, allowing a proof of Halpern-Leistner's purity conjecture via local-global principles.

\subsection{Notations and conventions}
\label{subsection:conventions}
\begin{enumerate}[$\bullet$]
 \item The letter $G$ will be mostly used to denote a (connected) reductive algebraic group. Let $G$ be a reductive algebraic group. We call \emph{rank} of $G$, and write $\rank(G)$, the dimension of a maximal torus of $G$.
 \item If $V$ is a representation of a finite group $W$ and $\chi$ an irreducible character of $W$, we let $V^{\chi}$ be the $\chi$-isotypic component of $V$.
 \item If $X$ is a complex algebraic variety acted upon by an algebraic group $G$, we let $X/G$ be the quotient stack (denoted without brackets). The stabilizer of $x\in X$ is denoted by $G_x$.
 \item The multiplicative group is denoted by $\BoG_{\rmm}$.
 The Abelian group of characters of an algebraic torus $T$ is $\Hom_{\BoZ}(T,\BoG_{\rmm})$. The set of cocharacters of $T$ is the Abelian group $\Hom_{\BoZ}(\BoG_{\rmm},T)$.
 \item Let $H\subset G$ be algebraic groups and $X$ an $H$-variety. We let $X\times^HG\coloneqq (X\times G)/H$ where $H$ acts diagonally on $X\times G$ by $h\cdot (x,g)=(h\cdot x,gh^{-1})$. The formula $g\cdot (x,g')\coloneqq (x,gg')$ gives a $G$-action on $X\times^HG$. The quotient stacks $X/H$ and $X\times^HG/G$ are equivalent.
 \item The cohomology of stacks and varieties will be considered with $\BoQ$-coefficients when the sheaf of coefficients is hidden. The constructible derived categories are always taken with $\BoQ$-coefficients.
 \item Let $X$ be a complex algebraic variety. The constructible derived category $\CD_{\rmc}(X,\BoQ)$ has the natural $t$-structure with cohomology functors $\CH^i$ and truncation functors $\tau^{\leq i}, \tau^{\geq i}$. It also has the perverse $t$-structure with cohomology functors ${^p\CH^i}$ and truncation functors ${^p\tau^{\leq i}}, {^p\tau^{\geq i}}$ \cite{beilinson2018faisceaux}. The derived category of monodromic mixed Hodge modules $\CD(\MMHM(X))$ has the natural $t$-structure with cohomology functors $\underline{\CH}^i$ and truncation functors $\underline{\tau}^{\leq i}, \underline{\tau}^{\geq i}$. We refer to \S\ref{appendix:MHM} for a minimal account on mixed Hodge modules.
 \item Constructible sheaves on an algebraic variety $X$ are denoted for example $\SF, \SG, \BoQ_X$. Monodromic mixed Hodge modules on an algebraic variety are underlined, for example: $\underline{\SF}, \underline{\SG}$, $\underline{\BoQ}_X$.
 \item We have a triangulated functor $\rat\colon\CD(\MMHM(X))\rightarrow\CD_{\rmc}(X)$, which intertwines the functors ${^p\tau^{\leq i}}$ and $\underline{\tau}^{\leq i}$, ${^p\tau^{\geq i}}$ and $\underline{\tau}^{\geq i}$, and ${^p\CH^{i}}$ and $\underline{\CH}^{i}$.
\item If $\FX$ is a smooth algebraic variety or smooth stack, we define $\underline{\BoQ}_{\FX}^{\vir}\coloneqq\underline{\BoQ}_{\FX}\otimes\SL^{-\dim\FX/2}$ a pure monodromic mixed Hodge module sitting in cohomological degree $-\dim \FX$, where $\SL$ is the Tate twist (a pure cohomologically graded mixed Hodge structure concentrated in degree $2$). The theory of mixed Hodge modules is extended to algebraic stacks in \cite{tubach2024mixed}. The present paper mainly deals with quotient stacks.
\item For a function $f\colon X\rightarrow \BoC$, we denote by $\varphi_f^p$ the shift of the vanishing cycle functor that is perverse $t$-exact. It acts on (complexes of) monodromic mixed Hodge modules and constructible sheaves compatibly with the functor $\rat$.
\item Given an algebraic variety $X$, the intersection cohomology $\IH^*(X)$ is given the perverse shift. That is, $\IH^i(X)$ may be nonzero for $-\dim X\leq i\leq \dim X$ only. Cohomology is given the classical shift, so that $\HO^i(X)$ may be nonzero for $0\leq i\leq 2\dim X$ only.
\item Given an algebraic torus $T$, there is a pairing $\langle-,-\rangle\colon X_*(T)\times X^*(T)\rightarrow\BoZ$ between characters $X^*(T)$ and cocharacters $X_*(T)$.
\item If $\FX$ is finite type separated stack, we denote by $\HO^{\rmBM}(\FX,\BoQ)$ its Borel--Moore homology. By definition, we have $\HO^{\rmBM}_{-*}(\FX,\BoQ)=p_*\BD\BoQ_{\FX}$ for the unique map $p\colon \FX\rightarrow\pt$, where $\BD\BoQ_{\FX}$ is the dualizing sheaf of $\FX$.
\item We denote by $\FX^{\cl}$ the classical truncation of a derived stack $\FX$.
\item A $1$-Artin stack is a derived stack $\FX$ whose classical truncation $\FX^{\cl}$ is an Artin stack.
\end{enumerate}

\subsection{Acknowledgements}
At various stages of the preparation of this work, the author was supported by the Royal Society, by the National Science Foundation under Grant No. DMS-1928930 and by the Alfred P. Sloan Foundation under grant G-2021-16778, while the author was in residence at the Simons Laufer Mathematical Sciences Institute (SLMath, formerly MSRI) in Berkeley, California, during the Spring 2024 semester. The author thanks the SLMath Sciences Institute and the University of Edinburgh for the excellent working conditions. The author is grateful to Ben Davison for the support provided, in particular via postdoctoral fellowships and useful discussions regarding parts of this work. Many thanks to Andr\'es Ib\'a\~nez N\'u\~nez for pointing out a mistake in the proof of Corollary~\ref{corollary:purityHiggs}. The author also benefited from the excellent working conditions at the University of Utrecht, at the Lorentz Center in Leiden and  at the Simons Center for Geometry and Physics, Stony Brook University.

\section{Absolute cohomological integrality for symmetric representations of reductive groups}
\label{section:reminderabsolute}
In this section, we recall the absolute cohomological integrality isomorphism established in \cite{hennecart2024cohomological} for symmetric representations of reductive groups. Via local-global principles relying on \'etale slices \S\ref{section:localneighbourhoodsmoothstacks}, it is the building block of all sheafified cohomological integrality isomorphisms we prove in this paper, in particular Theorem~\ref{theorem:sheafifiedcohintintro}, which relies on, but also generalizes \cite[Theorem 1.1]{hennecart2024cohomological} in two crucial directions. Theorem~\ref{theorem:sheafifiedcohintintro} indeed applies to any \emph{symmetric smooth affine $G$-variety} ($G$ reductive) and not only $X=V$ a symmetric representation of $G$, and it provides an \emph{isomorphism of complexes of mixed Hodge modules} rather than just an isomorphism at the level of derived global sections (i.e. isomorphism of graded vector spaces or graded mixed Hodge structures).

Let $V$ be a representation of a reductive group $G$. Let $T\subset G$ be a maximal torus. Given a cocharacter $\lambda\in X_*(T)$, we define a Levi subgroup $G^{\lambda}\subset G$, the centralizer of $\lambda$, and $V^{\lambda}\subset V$ the subspace fixed by $\lambda$, a representation of $G^{\lambda}$. The attracting locus $P_{\lambda}\coloneqq G^{\lambda\geq 0}\subset G$ is a (standard) parabolic subgroup and $V^{\lambda\geq 0}\subset V$ is a representation of $P_{\lambda}$. In \cite[\S2]{hennecart2024cohomological}, we described the induction morphism $\Ind_{\lambda}\colon \CH_{\lambda}\coloneqq \HO^*(V^{\lambda}/G^{\lambda},\BoQ)\rightarrow\CH\coloneqq \CH_0=\HO^*(V/G,\BoQ)$ at both the sheafified and absolute levels (see also \S\ref{section:parabolicinduction}).

The set of cocharacters $X_*(T)$ has an equivalence relation $\sim$ given by $\lambda\sim\mu$ if and only if $V^{\lambda}=V^{\mu}$ and $G^{\lambda}=G^{\mu}$. We let $\SP_{V}\coloneqq X_*(T)/\sim$. It has an action of the Weyl group $W=N_G(T)/T$ of $G$.

For $\lambda\in X_*(T)$, we let $G_{\lambda}\coloneqq \ker(G^{\lambda}\rightarrow\GL(V^{\lambda}))\cap Z(G^{\lambda})$ be the subgroup of the center of $G^{\lambda}$ acting trivially on $V^{\lambda}$. The group $G_{\lambda}$ may be disconnected but it is a product of a torus and a finite Abelian group. The class of $\lambda$ in $\SP_{V}$ is denoted by $\overline{\lambda}$. We let $W_{\lambda}\coloneqq \{w\in W\mid w\cdot\overline{\lambda}=\overline{\lambda}\}$. We denote by $W^{\lambda}$ the Weyl group of $G^{\lambda}$. It is a normal subgroup of $W_{\lambda}$. We let $\overline{W}_{\lambda}\coloneqq W_{\lambda}/W^{\lambda}$ be the quotient. It can be interpreted as a \emph{relative Weyl group} associated to $X,G,\lambda$ (\S\ref{subsection:relativeWeylgroups}). It can be identified to a subgroup of the relative Weyl group $W_{G,G^{\lambda}}\coloneqq N_G(G^{\lambda})/G^{\lambda}$. We say that $V$ is symmetric if $V$ and $V^*$ have the same sets of weights counted with multiplicities \cite[Definition 2.3]{hennecart2024cohomological}. For a symmetric representation $V$ of $G$, there is a character $\varepsilon_{V,\lambda}\colon \overline{W}_{\lambda}\rightarrow\{\pm1\}$ \cite[Proposition~4.8]{hennecart2024cohomological}, see also \ref{subsection:character}. The group $\overline{W}_{\lambda}$ acts on $\HO^*(V^{\lambda}/G^{\lambda},\BoQ)$ \cite[\S4.1]{hennecart2024cohomological}.

The main theorem of \cite{hennecart2024cohomological} is the following.

\begin{theorem}[{\cite[Theorem 1.1]{hennecart2024cohomological}}]
\label{theorem:abscohintreminder}
 Let $V$ be a symmetric representation of a reductive group $G$. Then, there exists finite dimensional subspaces $\CP_{\lambda}\subset\HO^*(V^{\lambda}/G^{\lambda})$ ($\lambda\in X_*(T)$) stable under the $\overline{W}_{\lambda}$-action such that the induction morphisms induce an isomorphism
 \begin{equation}
 \label{equation:abscohintiso}
  \bigoplus_{\tilde{\lambda}\in \SP_V/W}(\CP_{\lambda}\otimes\HO^*(\pt/G_{\lambda}))^{\varepsilon_{V,\lambda}}\rightarrow \CH_V\coloneqq \HO^*(V/G)\,
 \end{equation}
where $(\CP_{\lambda}\otimes\HO^*(\pt/G_{\lambda}))^{\varepsilon_{V,\lambda}}$ is the $\varepsilon_{V,\lambda}$-isotypic component.
\end{theorem}

The main theorem of \cite{hennecart2024cohomological} holds for the larger class of \emph{weakly symmetric} representations. The definition of the subspaces $\CP_{\lambda}\subset \HO^*(V^{\lambda}/G^{\lambda})$ is constructive, although implicit. Their definition is also constrained by the existence of a cohomological integrality isomorphism \eqref{equation:abscohintiso}. We recall how $\CP_{0}$ is defined. To obtain $\CP_{\lambda}$ for a general cocharacter $\lambda\in X_*(T)$, we apply this definition to the $G^{\lambda}$-representation $V^{\lambda}$, taking the $\overline{W}_{\lambda}$-action into account. We let $\overline{V}$ be the representation of $\overline{G}\coloneqq G/G_0$ deduced from $V$, where, as above, $G_0=\ker(G\rightarrow \GL(V))\cap Z(G)$. Then, the definition of the induction maps (\S\ref{section:parabolicinduction} and \cite[\S2]{hennecart2024cohomological} for representations of reductive groups) gives, for each cocharacter $\lambda\in X_*(T/G_0)$ an induction morphism $\Ind_{\lambda}\colon \HO^*(\overline{V}^{\lambda}/\overline{G}^{\lambda})\rightarrow\HO^*(\overline{V}/\overline{G})$. The vector space $\CP_{0}$ is defined as \emph{a direct sum complement} of
\[
 \sum_{\substack{\lambda\in X_*(T/G_0)\\(\overline{V}^{\lambda},\overline{G}^{\lambda})\neq (\overline{V},\overline{G})}}\Ind_{\lambda}(\HO^*(\overline{V}^{\lambda}/\overline{G}^{\lambda}))\subset \HO^*(\overline{V}/\overline{G})\,
\]
where the condition on the sum is exactly the condition $\overline{\lambda}\prec \overline{0}$ for the order relation on classes of cocharacters \cite[\S1.1]{hennecart2024cohomological}, where $0$ denotes the trivial cocharacter, see also \S\ref{subsection:orderrelation}.

We note that any choice of $\overline{W}_{\lambda}$-stable direct sum complement $\CP_{\lambda}$ in $\HO^*(\overline{V^{\lambda}}/\overline{G^{\lambda}})$ ($\lambda\in X_*(T)$) gives a cohomological integrality isomorphism \eqref{equation:abscohintiso}.

For the reader familiar with cohomological integrality for quiver with or without potential \cite{meinhardt2019donaldson, davison2020cohomological}, we would like to emphasize here that we will define the sheafified version $\underline{\CP}_{\lambda}$ of $\CP_{\lambda}$ in \ref{subsection:bpsmhm} in a similar way, by considering a direct sum complement in the Abelian category of cohomologically graded monodromic mixed Hodge modules (i.e. direct sums of shifted monodromic mixed Hodge modules, \S\ref{subsection:cohgradedmhm}). That these objects admit a definition similar to that in \cite{davison2020cohomological} (where they correspond to the first perverse cohomology piece of some cohomologically graded mixed Hodge module, and also to the intersection complex of the good moduli space) is the subject of Conjecture~\ref{conjecture:BPSsheavesnofunction}, proven for orthogonal representations of $G$ in \cite{bu2025cohomology}, open in general.

\section{Approximation by proper maps for affine varieties with reductive group action}
\label{section:approximationpropermaps}

We give now the definition of \emph{approximation by proper maps} for morphisms from a quotient stack to a scheme, following \cite[\S4.1]{davison2020cohomological}. A notion of approximation by proper maps for good moduli spaces has recently been introduced in \cite{kinjo2024decomposition}, while this paper was in preparation. It differs slightly from our approach. We keep our notion as we use it in a forthcoming work to obtain explicit estimates in some cases, allowing us to prove \emph{virtuall smallness} of some morphisms, although the formalism developped in \cite{kinjo2024decomposition} is more general.

\begin{definition}[Approximation by proper maps]
\label{definition:appm}
 Let $G$ be a reductive group, $X$ a complex $G$-scheme and $Y$ a complex scheme. We say that a morphism $f\colon X/G\rightarrow Y$ is \emph{approachable by proper maps} if there exists $G$-representations $V_n$, $n\geq 1$, $G$-equivariant open subvarieties $U_n\subset V_n$
 \begin{enumerate}
  \item $G$ acts freely on $U_n$ and $U_n\rightarrow U_n/G$ is a principal $G$-bundle in the category of schemes,
  \item The codimension $\codim_{V_n}(V_n\setminus U_n)$ of $V_n\setminus U_n\subset V_n$ tends to $+\infty$ as $n\rightarrow+\infty$.
  \item There exists $G$-invariant open subvarieties $X_n\subset X\times V_n$ such that the action of $G$ on $X_n$ admits a good quotient $X_n\cms G$ in the category of schemes, $X'_n\coloneqq X\times U_n\subset X_n$, and the natural map $X'_n/G\rightarrow X_n\cms G$ is an open immersion.
  \item The morphisms $X_n\cms G\rightarrow Y$ obtained as the compositions $X_n\cms G\rightarrow X\cms G\rightarrow Y$, where the first map is induced by the projection $X\times V_n\rightarrow X$, are proper.
 \end{enumerate}
\end{definition}

We introduce the notations
\[
 \overline{U_n}\coloneqq U_n/G,\quad \overline{X_n}\coloneqq X_n\cms G,\quad \overline{X}\coloneqq X\cms G,\quad \overline{X'_n}\coloneqq X'_n/G\,.
\]

\begin{remark}
 We do not require that $G$ acts freely on $X_n$. Therefore, the quotient $X_n\cms G$ is not necessarily smooth. In particular, this is a weakening of the notion of approximation by proper maps considered in \cite{davison2020cohomological}. This fact will explain that we will later (e.g. in Corollary~\ref{corollary:pushforwardic}) consider the \emph{intersection cohomology monodromic mixed Hodge module}/perverse sheaf $\underline{\IC}(\overline{X_n})$ on $\overline{X_n}$ rather than the constant sheaf as in \cite{davison2020cohomological} or \cite{kinjo2024decomposition}.
\end{remark}

\begin{proposition}
\label{proposition:codimensionestimate}
In the situation of Definition~\ref{definition:appm}, the codimension $\codim_{\overline{X_n}}(\overline{X_n}\setminus\overline{X'_n})$ of the complement of $\overline{X'_n}$ inside $\overline{X_n}$ tends to $+\infty$.
\end{proposition}
\begin{proof}
 We have 
 \[
 \begin{aligned}
  \codim_{\overline{X_n}}(\overline{X_n}\setminus \overline{X'_n})&=(\dim X_n-\dim G)-\dim(X_n\cms G\setminus X'_n/G)\\
  &\geq \dim X_n-\dim G-\dim (X_n\setminus X'_n)\quad\text{as $\dim (X_n\setminus X'_n)\geq \dim(X_n\cms G\setminus X'_n/G)$}\\
  &\geq \dim X_n-\dim G-\dim ((X\times V_n)\setminus X'_n)\quad\text{as $X_n\subset X\times V_n$ is open}\\
  &=\dim (X\times V_n)-\dim G-\dim ((X\times V_n)\setminus X'_n)\\
  &=\codim_{V_n}(V_n\setminus U_n)-\dim G
 \end{aligned}
 \]
 and by assumption, $\codim_{V_n}(V_n\setminus U_n)\rightarrow+\infty$ as $n\rightarrow+\infty$. This concludes.
\end{proof}

\begin{lemma}
\label{lemma:truncation}
 Let $\pi\colon X\rightarrow Y$ be a morphism of complex algebraic varieties. Then, 
 \begin{enumerate}
  \item For any $k\in\BoZ$, $\tau^{\leq k}\pi_*\tau^{\leq k}\cong\tau^{\leq k}\pi_*$.
  \item  For any $k\leq l$, $\tau^{\leq k}\pi_*\tau^{\leq l}\cong\tau^{\leq k}\pi_*$.
 \end{enumerate}
\end{lemma}
\begin{proof}
 We recall that (see \S\ref{subsection:conventions}) the truncation functors involved here correspond to the natural $t$-structure on the constructible derived category. The second isomorphism follows from the first, since for $k\leq l$, using the isomorphism (1), one has $\tau^{\leq k}\pi_*\tau^{\leq l}\cong\tau^{\leq k}\tau^{\leq l}\pi_*\cong\tau^{\leq k}\pi_*$.
 
 The first isomorphism comes from the fact that $\pi_*$ is left $t$-exact. More precisely, we consider the octahedral axiom diagram built on the pair of morphisms
\[
 \tau^{\leq k}\pi_*\tau^{\leq k}\CF\rightarrow\pi_*\tau^{\leq k}\CF\rightarrow\pi_*\CF
\]
as in \cite[(1.1.7.1)]{beilinson2018faisceaux}. The first morphism is the adjunction morphism $\tau^{\leq k}\CG\rightarrow\CG$ for $\CG=\pi_*\tau^{\leq k}\CF$. The second morphism is $\pi_*$ applied to this adjunction morphism for $\CG=\CF$. It reads
\[\begin{tikzcd}[column sep=tiny, row sep=tiny]
	&&&&&& {} \\
	\\
	&&&& {\tau^{\geq k+1}\pi_*\tau^{\leq k}\CF} \\
	&&&&&&&& {} \\
	&& {\pi_*\tau^{\leq k}\CF} &&&& A \\
	&&&& {\pi_*\CF} \\
	{\tau^{\leq k}\pi_*\tau^{\leq k}\CF} &&&&&&&& {\pi_*\tau^{\geq k+1}\CF} \\
	&&&&&&&&&& {} & {} & {} \\
	&&&&&&&&&& {}
	\arrow[from=3-5, to=1-7]
	\arrow[from=3-5, to=5-7]
	\arrow[from=5-3, to=3-5]
	\arrow[from=5-3, to=6-5]
	\arrow[from=5-7, to=4-9]
	\arrow[from=5-7, to=7-9]
	\arrow[from=6-5, to=5-7]
	\arrow[from=6-5, to=7-9]
	\arrow[from=7-1, to=5-3]
	\arrow[from=7-1, to=6-5]
	\arrow[from=7-9, to=8-12]
	\arrow[from=7-9, to=9-11]
\end{tikzcd}\]
where all four triples of three objects aligned give a distinguished triangle. Since $\pi_*$ is left $t$-exact, $\pi_*\tau^{\geq k+1}\CF\in\CD_{\rmc}^{\geq k+1}(Y)$. By the distinguished triangle
\[
\tau^{\geq k+1}\pi_*\tau^{\leq k}\CF\rightarrow A\rightarrow\pi_*\tau^{\geq k+1}\CF\rightarrow,
\]
we have $A\in\CD_{\rmc}^{\geq k+1}(Y)$. Since $\tau^{\leq k}\pi_*\tau^{\leq k}\CF\in\CD^{\leq k}(Y)$, then by unicity of the distinguished triangle
\[
 U\rightarrow\pi_*\CF\rightarrow V\rightarrow
\]
with $U\in\CD^{\leq k}(Y)$ and $V\in\CD^{\geq k+1}(Y)$, we have $\tau^{\leq k}\pi_*\tau^{\leq k}\CF\cong\tau^{\leq k}\pi_*\CF$. This proves (1).
\end{proof}

\begin{proposition}
\label{proposition:comparisonintextensionstar}
 Let $X$ be an irreducible algebraic variety. Let $\jmath\colon U\hookrightarrow X$ be a smooth open subset. We let $\codim_X(X\setminus U)$ be the codimension of $X\setminus U$ in $X$. Then, the natural map
 \[
  \tau^{\leq\codim_{X}(X\setminus U)-1}(\jmath_{!*}\SL\rightarrow\jmath_*\SL)
 \]
is an isomorphism.
\end{proposition}
\begin{proof}
 We use Deligne's construction of the intermediate extension (e.g. \cite[Exercise 3.10.1]{achar2021perverse} or \cite[Proposition~2.1.11]{beilinson2018faisceaux}). Namely, we choose a stratification $X=\bigsqcup_{i=1}^nX_i$ of $X$ such that $\jmath_{!*}\SL$ is constructible with respect to this stratification, $X_1=U$ and $U_i\coloneqq \bigcup_{j=1}^iX_j$ is open in $X$ for $1\leq i\leq n$. This implies in particular
 \[
  \dim X_1\geq \dim X_2\geq\hdots\geq \dim X_n\,.
 \]
From this sequence of inequalities and $\codim_{U_{i}}(U_{i}\setminus U_{i-1})=\dim X-\dim X_i$, we obtain
\begin{equation}
 \label{equation:seriesinequalities}
  \codim_{U_n}(U_n\setminus U_{n-1})\geq\codim_{U_{n-1}}(U_{n-1}\setminus U_{n-2})\geq\hdots\geq\codim_{U_2}(U_2\setminus U_1)= \codim_X(X\setminus U)\,.
 \end{equation}
 We consider the open immersions
 \[
  U=U_1\xrightarrow{\jmath_{1}}U_2\xrightarrow{\jmath_2}\hdots\xrightarrow{\jmath_{n-1}}U_n=X
 \]
 and $\jmath=\jmath_{n-1}\circ\hdots\circ\jmath_1\colon U\rightarrow X$.
 Then, we have
 \[
  \jmath_{!*}\SL\cong \tau^{\leq\codim_{U_n}(U_n\setminus U_{n-1})-1}(\jmath_{n-1})_*\tau^{\codim_{U_{n-1}}(U_{n-1}\setminus U_{n-2})-1}(\jmath_{n-2})_*\hdots\tau^{\leq \codim_{U_2}(U_2\setminus U_1)-1}(\jmath_1)_*\SL\,.
 \]
The Proposition follows from the fact that, by \eqref{equation:seriesinequalities}, $\codim_X(X\setminus U)$ is smaller than or equal to all the truncations defining $\jmath_{!*}\SL$, by using Lemma~\ref{lemma:truncation} repeatedly.
\end{proof}

\begin{lemma}
\label{lemma:perverseclassicaltruncations}
 Let $\SF,\SG$ be complexes of constructible sheaves over an algebraic variety $X$. Then, if $\tau^{\leq k}\SF\cong\tau^{\leq k}\SG$, then ${^p\tau}^{\leq k}\SF\cong {^p\tau^{\leq k}}\SG$.
\end{lemma}
\begin{proof}
 If $\SH\in\Perv(X)$, then $\CH^i(\SH)=0$ (where $\CH$ denotes the cohomology sheaves for the natural $t$-structure on the derived category of constructible sheaves) for $i>0$ by the support condition (e.g. \cite[\S2.3]{de2009decomposition}). Therefore, we have an inclusion ${^p\CD_{\rmc}^{\leq k}(X)}\subset\CD_{\rmc}^{\leq k}(X)$. The right-adjoint to the inclusion ${^p\CD}_{\rmc}^{\leq k}(X)\subset\CD_{\rmc}(X)$ is ${^p\tau}^{\leq k}$. The composition of functors ${^p\tau}^{\leq k}\tau^{\leq k}$ is right-adjoint to the composition of inclusions ${^p\CD_{\rmc}^{\leq k}(X)}\subset\CD_{\rmc}^{\leq k}(X)\subset\CD_{\rmc}(X)$.  Therefore, we have a canonical isomorphism ${^p\tau}^{\leq k}\tau^{\leq k}\cong {^p\tau^{\leq k}}$. Applying ${^p\tau}^{\leq k}$ to $\tau^{\leq k}\SF\cong\tau^{\leq k}\SG$ gives the lemma.
\end{proof}

We consider the diagrams
\[\begin{tikzcd}
	{\overline{X'_n}} & {} & {\overline{X_n}} \\
	& Y
	\arrow["{\jmath_n}", from=1-1, to=1-3]
	\arrow["{\pi_n}"', from=1-1, to=2-2]
	\arrow["{p_n}", from=1-3, to=2-2]
\end{tikzcd}\]

\begin{corollary}
\label{corollary:pushforwardic}
 In the situation of Definition~\ref{definition:appm}, the natural morphism of complexes of monodromic mixed Hodge modules
 \[
  \underline{\tau}^{\leq \codim_{\overline{X_n}}(\overline{X_n}\setminus \overline{X'_n})-1}\left((\jmath_n)_{!*}\underline{\BoQ}_{\overline{X'_n}}=\underline{\IC}(\overline{X_n})\otimes \SL^{(\dim X_n-\dim G)/2}\rightarrow(\jmath_n)_*\underline{\BoQ}_{\overline{X'_n}}\right)
 \]
 is an isomorphism.
\end{corollary}
\begin{proof}
 We may prove the statement after applying the (conservative) functor $\rat$, i.e. in the categories of constructible sheaves. We have to prove that
  \[
  {^p\tau}^{\leq \codim_{\overline{X_n}}(\overline{X_n}\setminus \overline{X'_n})-1}\left((\jmath_n)_{!*}{\BoQ}_{\overline{X'_n}}={\IC}(\overline{X_n})[-(\dim X_n-\dim G)]\rightarrow(\jmath_n)_*{\BoQ}_{\overline{X'_n}}\right)
 \]
 is an isomorphism. This is the combination of Proposition~\ref{proposition:comparisonintextensionstar} and Lemma~\ref{lemma:perverseclassicaltruncations}.
\end{proof}

\begin{lemma}
\label{lemma:pftruncatediso}
 Let $\pi\colon X\rightarrow Y$ be a morphism between finite type separated complex schemes, $k\in\BoZ$ and $\underline{\CF}\xrightarrow{f}\underline{\CG}\in\CD^+(\MMHM(X))$ be a morphism of monodromic mixed Hodge modules such that ${\tau}^{\leq k}(f)\colon{\tau}^{\leq k}\CF\rightarrow{\tau}^{\leq k}\CG$ is an isomorphism of constructible complexes (where $\CF=\rat(\underline{\CF})$ and $\CG=\rat(\underline{\CG})$). Then, $\underline{\tau}^{\leq k}\pi_*f\colon\underline{\tau}^{\leq k}\pi_*\underline{\CF}\rightarrow\underline{\tau}^{\leq k}\pi_*\underline{\CG}$ is an isomorphism.
\end{lemma}
\begin{proof}
By conservativity of the functor $\rat$ to the derived category of constructible sheaves, we may prove the theorem for constructible sheaves. By Lemma~\ref{lemma:perverseclassicaltruncations}, we may replace perverse truncation functors by classical truncation functors. Then, this is essentially Lemma~\ref{lemma:truncation}.

 More precisely, we have the functorial triangles and the morphism between them:
\[\begin{tikzcd}
	{\tau^{\leq k}\CF} & \CF & {\tau^{\geq k+1}\CF} & {} \\
	{\tau^{\leq k}\CG} & \CG & {\tau^{\geq k+1}\CG} & {}
	\arrow["a",from=1-1, to=1-2]
	\arrow["{\tau^{\leq k}f}"', from=1-1, to=2-1]
	\arrow[from=1-2, to=1-3]
	\arrow["f"', from=1-2, to=2-2]
	\arrow[from=1-3, to=1-4]
	\arrow["{\tau^{\geq k+1}f}", from=1-3, to=2-3]
	\arrow["b",from=2-1, to=2-2]
	\arrow[from=2-2, to=2-3]
	\arrow[from=2-3, to=2-4]
\end{tikzcd}\]
Since $\tau^{\leq k}f$ is an isomorphism, so is $\tau^{\leq k}\pi_*\tau^{\leq k}f$. Therefore, to prove that $\tau^{\leq k}\pi_* f$ is an isomorphism, it suffices to prove that $\tau^{\leq k}\pi_*a$ and $\tau^{\leq k}\pi_*b$ are isomorphisms. The proofs for $a$ and $b$ are the same. This is a consequence of the octahedral axiom in triangulated categories. Namely, by applying $\pi_*$ to the distinguished triangle for $\CF$, we obtain the distinguished triangle
\[
 \pi_*\tau^{\leq k}\CF\xrightarrow{\pi_*a}\pi_*\CF\rightarrow\pi_*\tau^{\geq k+1}\CF\rightarrow\,.
\]
By Lemma~\ref{lemma:truncation}, we have an isomorphism $h\colon \tau^{\leq k}\pi_*\tau^{\leq k}\CF\xrightarrow{\sim}\tau^{\leq k}\CF$. This isomorphism is $\tau^{\leq k}\pi_*a$. Indeed, we have a commutative diagram
\[\begin{tikzcd}
	{\tau^{\leq k}\pi_*\tau^{\leq k}\CF} & {\pi_*\tau^{\leq k}\CF} & {\pi_*\CF} \\
	{\tau^{\leq k}\pi_*\CF} & {\pi_*\CF} & {\pi_*\CF}
	\arrow[from=1-1, to=1-2]
	\arrow["{\tau^{\leq k}\pi_*a}"', from=1-1, to=2-1]
	\arrow["{\pi_*a}", from=1-2, to=1-3]
	\arrow["{\pi_*a}"', from=1-2, to=2-2]
	\arrow["\id", from=1-3, to=2-3]
	\arrow[from=2-1, to=2-2]
	\arrow["\id"', from=2-2, to=2-3]
\end{tikzcd}\]
where left-most horizontal maps are adjunction maps. Therefore, the outer square
\[\begin{tikzcd}
	{\tau^{\leq k}\pi_*\tau^{\leq k}\CF} & {\pi_*\CF} & {} \\
	{\tau^{\leq k}\pi_*\CF} & {\pi_*\CF}
	\arrow[from=1-1, to=1-2]
	\arrow["{\tau^{\leq k}\pi_*a}"', from=1-1, to=2-1]
	\arrow["\id", from=1-2, to=2-2]
	\arrow[from=2-1, to=2-2]
\end{tikzcd}\]
commutes. We have an isomorphism of triangles obtained in the proof of Lemma~\ref{lemma:truncation}
\[\begin{tikzcd}
	{\tau^{\leq k}\pi_*\tau^{\leq k}\CF} & {\pi_*\CF} & {\tau^{\geq k+1}\pi_*\tau^{\leq k}\CF} & {} \\
	{\tau^{\leq k}\pi_*\CF} & {\pi_*\CF} & A & {}
	\arrow[from=1-1, to=1-2]
	\arrow["h"', from=1-1, to=2-1]
	\arrow[from=1-2, to=1-3]
	\arrow["\id", from=1-2, to=2-2]
	\arrow[from=1-3, to=1-4]
	\arrow[from=1-3, to=2-3]
	\arrow[from=2-1, to=2-2]
	\arrow[from=2-2, to=2-3]
	\arrow[from=2-3, to=2-4]
\end{tikzcd}\,.\]
Since $\Hom(\tau^{\leq k}\pi_*\tau^{\leq k}\CF,\tau^{\leq k}\pi_*\CF)\cong\Hom(\tau^{\leq k}\pi_*\tau^{\leq k}\CF,\pi_*\CF)$ by adjunction, then the morphisms $\tau^{\leq k}\pi_*a$ and $h$ coincide. Since the latter is an isomorphism, so is $\tau^{\leq k}\pi_*a$.
\end{proof}

\begin{corollary}
\label{corollary:isopushforward}
 In the situation of Definition~\ref{definition:appm}, we have an isomorphism
 \[
  \underline{\tau}^{\leq \codim_{\overline{X_n}}(\overline{X_n}\setminus \overline{X'_n})-1}(p_n)_*(\underline{\IC}(\overline{X_n})\otimes\SL^{\dim \overline{X_n}/2})\cong\underline{\tau}^{\leq \codim_{\overline{X_n}}(\overline{X_n}\setminus \overline{X'_n})-1}(\pi_n)_*\underline{\BoQ}_{\overline{X'_n}}
 \]
of complexes of mixed Hodge modules.
\end{corollary}
\begin{proof}
 By applying $\underline{\tau}^{\leq k}(p_n)_*$ with $k=\codim_{\overline{X_n}}(\overline{X_n}\setminus \overline{X'_n})-1$ to the morphism
 \[
  (\jmath_n)_{!*}\underline{\BoQ}_{\overline{X'_n}}=\underline{\IC}(\overline{X_n})\otimes \SL^{\dim \overline{X_n}/2}\rightarrow(\jmath_n)_*\underline{\BoQ}_{\overline{X'_n}}\,,
 \]
 Lemma \ref{lemma:pftruncatediso} together with Proposition~\ref{proposition:comparisonintextensionstar} provide us with an isomorphism
 \[
  \underline{\tau}^{\leq k}(p_n)_*(\jmath_n)_{!*}\underline{\BoQ}_{\overline{X'_n}}=\underline{\tau}^{\leq k}(p_n)_*\underline{\IC}(\overline{X_n})\otimes \SL^{\dim\overline{X_n}/2}\rightarrow\underline{\tau}^{\leq k}(p_n)_*(\jmath_n)_*\underline{\BoQ}_{\overline{X'_n}}\cong \underline{\tau}^{\leq k}(\pi_n)_*\underline{\BoQ}_{\overline{X'_n}}\,.
 \]
\end{proof}

\begin{corollary}
\label{corollary:purityappm}
 In the situation of Definition~\ref{definition:appm}, the complex of monodromic mixed Hodge modules $\underline{\tau}^{\leq \codim_{\overline{X_n}}(\overline{X_n}\setminus\overline{X'_n})-1}(\pi_n)_*\underline{\BoQ}_{\overline{X'_n}}\in\CD^+(\MMHM(Y))$ is pure, and hence semisimple.
\end{corollary}
\begin{proof}
 This follows from Corollary~\ref{corollary:isopushforward}, since $\underline{\IC}(\overline{X_n})$ is pure and $p_n$ is proper, which implies that $(p_n)_*\underline{\IC}(\overline{X_n})$ is pure, and consequently, $\underline{\tau}^{\leq \codim_{\overline{X_n}}(\overline{X_n}\setminus \overline{X'_n})-1}(p_n)_*\underline{\IC}(\overline{X_n})$ is pure. Then, purity classically implies semisimplicity.
\end{proof}

\begin{lemma}
 \label{lemma:codimensionsurjective}
 Let $k,n\in\BoN$. Let let $V_n\coloneqq \Hom(\BoC^n,\BoC^{k+1})$. We denote by $U_n$ the subset of surjective linear maps. Then, $\lim_{n\rightarrow+\infty}\codim_{V_n}(V_n\setminus U_n)=+\infty$.
\end{lemma}
\begin{proof}
 We may assume that $n\geq k+1$. For $1\leq m\leq n$, We denote by $F_m\subset\Hom(\BoC^n,\BoC^{k+1})$ the closed subset defined by the vanishing of the $m$ first $(k+1)\times(k+1)$ minors. Then, we have strict inclusions
 \[
  V_n\setminus U_n=F_{n-k}\subsetneq F_{n-k-1}\subsetneq\hdots\subsetneq F_{0}=\Hom(\BoC^n,\BoC^{k+1})\,.
 \]
The length of this chain grows linearly in $n$ and so tends to $+\infty$ as $n\rightarrow+\infty$. This concludes.
\end{proof}

\begin{proposition}
\label{proposition:Gprimeapproachable}
Let $G$ be a reductive group and $X$ an affine $G$-variety. We extend the $G$-action to a $G'\coloneqq G\times\BoG_{\rmm}$-action, where $\BoG_{\rmm}$ acts trivially on $X$. Then, the map $\pi'\colon X/G'\rightarrow X\cms G'\cong X\cms G$ is approachable by proper maps (Definition~\ref{definition:appm}).
\end{proposition}
\begin{proof}
 We choose an embedding $G\rightarrow \GL_k$ for some $k\geq 1$. We define an embedding $G'\rightarrow\GL_{k+1}$ such that $G$ is sent to $\GL_k\subset\GL_{k+1}$ and $\BoC^*$ to scalar matrices inside $\GL_{k+1}$. We let $V_n\coloneqq \Hom(\BoC^{n},\BoC^{k+1})$. The group $G'$ acts on $V_n$ via $G'\rightarrow\GL_{k+1}$. We let
 \[
 \begin{matrix}
  \theta&\colon&\GL_{k+1}&\rightarrow&\BoG_{\rmm}\\
  &&g&\mapsto&\det(g)\,.
 \end{matrix}
 \]
It induces a character $\theta'$ of $G'$ by restriction. We let $U_n\subset V_n$ be the open subset of surjective maps $\BoC^n\rightarrow\BoC^{k+1}$. Then, $U_n\subset V_n^{\theta'\sst}$ and so, $X\times U_n\subset (X\times V_n)^{\theta'\sst}$. Moreover, the $G'$-action on $U_n$ is free, and so the $G'$-action on $X\times U_n$ is also free. We let $X_n\coloneqq (X\times V_n)^{\theta'\sst}$. It admits a good quotient by $G$, the GIT quotient $(X\times V_n)\cms_{\theta'}G$. Since the copy of $\BoC^*\subset G'$ rescales $V_n$, $(X\times V_n)\cms G\cong X\cms G$. Therefore, the map $X_n\cms_{\theta'}G\rightarrow X\cms G$ is proper. In addition, the codimension of $V_n\setminus U_n$ in $V_n$ tends to $+\infty$ (Lemma~\ref{lemma:codimensionsurjective}) and therefore, $\codim_{X_n\cms G}(X_n\cms G\setminus U_n/G)\rightarrow+\infty$ (Proposition~\ref{proposition:codimensionestimate}). This gives an approximation by proper maps of $X/G'\rightarrow X\cms G$.
\end{proof}

\begin{corollary}
\label{corollary:purityadditionalGm}
 In the situation of Proposition~\ref{proposition:Gprimeapproachable}, the complex of monodromic mixed Hodge modules $\pi'_*\underline{\BoQ}_{X/G'}\in\CD^+(\MMHM(X\cms G))$ is pure.
\end{corollary}
\begin{proof}
 This follows from Corollary~\ref{corollary:purityappm} since the morphism $\pi'$ is approachable by proper maps.
\end{proof}

\subsection{Semisimplicity}

\begin{proposition}
\label{proposition:semisimplicitypistar}
 Let $G$ be a reductive algebraic group and $X$ a smooth affine $G$-variety. We consider the natural GIT quotient map $\pi\colon X/G\rightarrow X\cms G$. Then, $\pi_*\underline{\BoQ}_{X/G}$ is a pure complex of mixed Hodge modules. In particular, it is semisimple.
\end{proposition}
\begin{proof}
 Let $G'\coloneqq G\times\BoG_{\rmm}$ act on $X$, where $\BoG_{\rmm}$ acts trivially. We let $\pi'\colon X/G'\cong X/G\times\pt/\BoG_{\rmm}\rightarrow X\cms G'\cong X\cms G$. It suffices to prove that $\pi'_*\underline{\BoQ}_{X/G'}\cong \pi_*\underline{\BoQ}_{X/G}\otimes\HO^*(\pt/\BoG_{\rmm})$ is pure. This is Corollary~\ref{corollary:purityadditionalGm}.
\end{proof}
In particular, there is a (non-canonical) isomorphism $\pi_*\underline{\BoQ}_{X/G}\cong \bigoplus_{i\in\BoZ}\underline{\CH}^i(\pi_*\underline{\BoQ}_{X/G})[-i]$. We may therefore see $\pi_*\underline{\BoQ}_{X/G}$ as a cohomologically graded monodromic mixed Hodge module \S\ref{subsection:cohgradedmhm}.

\section{The set of partitions of a $G$-variety}
As explained in \cite{hennecart2024cohomological} and recalled in \S\ref{section:reminderabsolute}, a crucial role in the cohomological integrality isomorphism is played by a certain set of equivalence classes of cocharacters $\SP_{V}$ for a representation $V$ of a reductive group $G$. The set $\SP_V$ plays for an arbitrary representation $V$ of a reductive group $G$ the role of decomposition of dimension vectors for quivers in \cite{meinhardt2019donaldson,davison2020cohomological}.

In this section, we generalise the definition of $\SP_V$ to any smooth affine $G$-variety $X$ and we give an alternative, better suited description, involving tori instead of cocharacters (although cocharacters themselves need to be considered to define the induction \S\ref{section:parabolicinduction}).

\subsection{The set of partitions}
\label{subsection:setpartitions}
Let $X$ be a smooth $G$-variety. Let $\lambda\in X_*(T)$ be a cocharacter. We decompose the fixed point locus into connected components $X^{\lambda}=\bigsqcup_{\alpha\in\pi_0(X^{\lambda})}X^{\lambda}_{\alpha}$. We define $\SQ_{X,G}$ as the set of pairs $(\lambda,\alpha)$ of a cocharacter $\lambda\in X_*(T)$ together with a connected component $\alpha\in\pi_0(X^{\lambda})$. We define an equivalence relation on $\SQ_{X,G}$ as follows:
\[(\lambda,\alpha)\sim(\mu,\beta)\iff\left\{
 \begin{aligned}
  &X^{\lambda}_{\alpha}=X^{\mu}_{\beta}\quad\text{ as subvarieties of $X$}\\
  &G^{\lambda}=G^{\mu}\,.
 \end{aligned}
 \right.
\]
The quotient of $\SQ_{X,G}$ by this equivalence relation is denoted by $\SP_{X,G}\coloneqq \SQ_{X,G}/\sim$. The class of $(\lambda,\alpha)\in\SQ_{X,G}$ in $\SP_{X,G}$ is denoted by $\overline{(\lambda,\alpha)}$. The Weyl group $W$ of $G$ acts naturally on $\SQ_{X,G}$: if $w\in W$, we choose an arbitrary lift $\dot{w}\in N_G(T)$. Then, $\dot{w}\cdot X^{\lambda}_{\alpha}$ is a connected component $X^{w\cdot{\lambda}}_{w\cdot\alpha}$ of $X^{w\cdot\lambda}$, which does not depend on the lift $\dot{w}$. We may therefore define the action by $w\cdot(\lambda,\alpha)=(w\cdot\lambda,w\cdot\alpha)$. This action descends to $\SP_{X,G}$. We let $\tilde{\SP}_{X,G}\coloneqq \SP_{X,G}/W$ be the set-theoretic quotient. We call $\SP_{X,G}$ the \emph{set of partitions for $X,G$}.

\subsection{Order on the set of partitions}
\label{subsection:orderrelation}
We define an order relation on the set of partitions $\SP_{X,G}$ of $X,G$ and also on $\SQ_{X,G}$ as follows:
\[
 (\lambda,\alpha)\preceq (\mu,\beta)\iff\left\{\begin{aligned}
                                                X^{\lambda}_{\alpha}&\subset X^{\mu}_{\beta}\\
                                                G^{\lambda}&\subset G^{\mu}\,.
                                               \end{aligned}
\right.
\]

\subsection{Alternative descriptions}

\subsubsection{One-parameter subgroups}

Let $\Hom_{\BoZ}(\BoG_{\rmm},G)$ denote the set of one-parameter subgroups of $G$. Given $\lambda\colon \BoG_{\rmm}\rightarrow G$ such a one parameter subgroup, one may consider the fixed loci $X^{\lambda}\subset X$ and $G^{\lambda}\subset G$ and define the set $\SQ'_{X,G}$ as the set of pairs $(\lambda,\alpha)$ of a one-parameter subgroup of $G$ together with a connected component $\alpha$ of $X^{\lambda}$. The equivalence relation $\sim$ defined in \S\ref{subsection:setpartitions} (identification of the connected components of the fixed loci and of the Levi subgroups) can be extended to $\SQ'_{X,G}$. We let $\SP'_{X,G}\coloneqq\SQ'_{X,G}/\sim$ be the quotient. The class of $(\lambda,\alpha)\in\SQ'_{X,G}$ in $\SP'_{X,G}$ is denoted by $\overline{(\lambda,\alpha)}$. The action of $G$ by conjugation on $\Hom_{\BoZ}(\BoG_{\rmm},G)$ induces an action of $G$ on $\SQ'_{X,G}$ and on $\SP'_{X,G}$. We denote by $\tilde{\SP}_{X,G}\coloneqq\SP'_{X,G}/G$ the set-theoretic quotient.

\begin{lemma}
 We have a natural map $\SP_{X,G}\rightarrow\SP'_{X,G}$ which induces an isomorphism $\SP_{X,G}/W\rightarrow \SP'_{X,G}/G$.
\end{lemma}
\begin{proof}
The map $\SP_{X,G}\rightarrow\SP'_{X,G}$ is obtained by seeing a cocharacter of $T$ as a one-parameter subgroup of $G$ via the inclusion $T\subset G$. Since any one-parameter subgroup of $G$ is conjugated under $G$ to a cocharacter of $T$ (since any maximal torus of $G$ is conjugated to $T$), then this map is surjective. To prove injectivity, it suffices to prove that if $\lambda,\mu\in X_*(T)$ are two cocharacters conjugated by some $g\in G$, i.e. $\lambda=g\cdot \mu$, then they are conjugated under the normalizer $N_G(T)$ of $T\subset G$. We let $G^{\lambda}$ (resp. $G^{\mu}$) be the centralizer of $\lambda$ (resp. $\mu$). Then, $T\subset G^{\lambda},G^{\mu}$ and since $\lambda=g\cdot\mu$, then $T,gTg^{-1}\subset G^{\lambda}$ are two maximal tori. Therefore, there exists $h\in G^{\lambda}$ such that $hgT(hg)^{-1}=T$. We have $hg\in N_G(T)$ and moreover, $hg\cdot \mu=h\cdot \lambda=\lambda$: $\lambda$ and $\mu$ are conjugated under $N_G(T)$. This proves the injectivity.
\end{proof}

\subsubsection{Tori}
\label{subsubsection:tori}
We let $\CT(T)$ be the set of subtori of $T$. It is acted upon by the Weyl group $W$ of $G$. We let $\SQ^{''}_{X,G}$ be the set of pairs $(T',\alpha)$ of a torus $T'\in\CT(T)$ and a connected component $\alpha\in \pi_0(X^{T'})$ of the fixed locus $X^{T'}$. We may define the equivalence relation $\sim$ on $\SQ^{''}_{X,G}$ as in \S\ref{subsection:setpartitions}: $(T',\alpha)\sim(T'',\beta)$ if and only if $X^{T'}_{\alpha}=X^{T''}_{\beta}$ and $G^{T'}=G^{T''}$. We let $\SP''_{X,G}\coloneqq\SQ''_{X,G}/\sim$ be the quotient. The class of $(T',\alpha)\in\SQ_{X,G}''$ in $\SP''_{X,G}$ is denoted by $\overline{(T',\alpha)}$. The Weyl group $W$ on $G$ acts on $\SQ_{X,G}''$ and $\SP_{X,G}''$.

\begin{lemma}
 We have a natural $W$-equivariant bijection $\SP_{X,G}\rightarrow\SP''_{X,G}$. In particular, the set-theoretic quotients $\SP_{X,G}/W$ and $\SP''_{X,G}/W$ are in natural bijection.
\end{lemma}
\begin{proof}
 The map $\SP_{X,G}\rightarrow\SP''_{X,G}$ sends $(\lambda,\alpha)$ to $(\lambda(\BoG_{\rmm}),\alpha)$. Conversely, if $T'\in\CT(T)$ is a torus and $\alpha\in \pi_0(X^{T'})$, then for a generic cocharacter $\lambda\colon\BoG_{\rmm}\rightarrow T'\subset T$, $X^{\lambda}=X^{T'}$ and $G^{\lambda}=G^{T'}$. Then, there is a unique $\beta\in\pi_0(X^{\lambda})$ such that $X^{\lambda}_{\beta}=X^{T'}_{\alpha}$. The map sending $\overline{(T',\alpha)}\mapsto \overline{(\lambda,\beta)}$ is the inverse. The $W$-equivariance is straightforward.
\end{proof}

Let $(T',\alpha)\in \SQ''_{X,G}$. We let $\tilde{T'}$ be the neutral component of $\ker(G^{T'}\rightarrow \Aut(X^{T'}_{\alpha}))\cap Z(G^{T'})$. This is a torus in $T$, that is an element of $\CT(T)$.

\begin{lemma}
\label{lemma:connectedcomponent}
 The subvariety $X^{T'}_{\alpha}$ is a connected component of $X^{\tilde{T'}}$. Moreover, $G^{T'}=G^{\tilde{T'}}$.
\end{lemma}
\begin{proof}
 We have clearly $T'\subset \tilde{T'}$ and by definition, $\tilde{T'}$ acts trivially on $X^{T'}_{\alpha}$. Therefore, we have inclusions $X_{\alpha}^{T'}\subset X^{\tilde{T'}}\subset X^{T'}$. Since $X_{\alpha}^{T'}$ is a connected component of $X^{T'}$, it is also a connected component of $X^{\tilde{T'}}$. For the second statement, $T'\subset \tilde{T'}$ implies $G^{\tilde{T'}}\subset G^{T'}$ and since $\tilde{T'}\subset Z(G^{T'})$, this inclusion must be an equality.
\end{proof}

By Lemma~\ref{lemma:connectedcomponent}, there exists $\beta\in\pi_0(X^{\tilde{T'}})$ such that $X^{T'}_{\alpha}=X^{\tilde{T'}}_{\beta}$. Then, $(\tilde{T'},\beta)\in \SQ''_{X,G}$ is a \emph{preferred representative} of $\overline{(T',\alpha)}\in\SP''_{X,G}$. By construction, the preferred representative is uniquely defined. This gives the following description of $\SP_{X,G}$.

\begin{lemma}
\label{lemma:preferredrepresentative}
 We have
 \[
  \SP_{X,G}=\{(\tilde{T'},\alpha)\mid \tilde{T'}\subset T \text{ torus, }\alpha\in\pi_0(X^{\tilde{T'}}), \tilde{T'}=(\ker(G^{\tilde{T'}}\rightarrow\Aut(X^{\tilde{T'}}_{\alpha}))\cap Z(G^{\tilde{T'}}))^{\circ}\}\,.
 \]\qed
\end{lemma}

\begin{definition}
\label{definition:preferredtorus}
Given $(T',\alpha)\in \SQ'_{X,G}$, we let $G_{(T',\alpha)}^{\circ}\coloneqq (\ker(G^{T'}\rightarrow\Aut(X^{T'}_{\alpha}))\cap Z(G^{T'}))^{\circ}$. For $(\lambda,\alpha)\in \SQ_{X,G}$, we define $G_{(\lambda,\alpha)}^{\circ}$ similarly.

If $(T',\alpha)\in \SQ_{X,G}$ is such that $G_{(T',\alpha)}^{\circ}=T'$, we will call $T'$ a \emph{preferred torus for $X,G$}.
\end{definition}

\begin{remark}
 If $X$ is connected and has a $G$ fixed point (for example, if $X=V$ is a linear representation of $G$), then $X^{T'}$ is connected for any torus $T'\subset T$. Therefore, $\SP_{X,G}$ can be seen as the set of equivalence classes of cocharacters of $T$ or of subtori of $T$ under the equivalence relation described at the beginning of \S\ref{subsubsection:tori}. In particular, $\SQ''_{X,G}=\CT(T)$. In this case, we will drop the mention of the connected component $\alpha$ from the notation.
\end{remark}

\subsection{Weyl groups}
\label{subsection:Weylgroups}
Let $(\lambda,\alpha)\in \SQ_{X,G}$. We let $W_{(\lambda,\alpha)}\coloneqq \{w\in W\mid w\cdot\overline{(\lambda,\alpha)}=\overline{(\lambda,\alpha)}\}$. Similarly, if $(T',\alpha)\in\SQ''_{X,G}$, we let $W_{(T',\alpha)}\coloneqq \{w\in W\mid w\cdot \overline{(T',\alpha)}=\overline{(T',\alpha)}\}$.

\begin{lemma}
\label{lemma:relWeylpreferredtorus}
 Let $(T',\alpha)\in \SQ_{X,G}$ be such that $G_{(T',\alpha)}^{\circ}=T'$ (i.e. $T'$ is a preferred torus -- Definition~\ref{definition:preferredtorus}). Then, $W_{(T',\alpha)}=\{w\in W\mid w\cdot T'=T', w\cdot \alpha=\alpha\}$. In particular, if $X$ has a $G$-fixed point, $W_{T'}=\{w\in W\mid w\cdot T'=T'\}$ is the subgroup of elements preserving the preferred torus $T'$.
\end{lemma}
\begin{proof}
 Let $w\in W$. Then, $G_{(w\cdot T',w\cdot\alpha)}^{\circ}=w\cdot{G_{(T',\alpha)}^{\circ}}=\dot{w}G_{(T',\alpha)}^{\circ}\dot{w}^{-1}$ for some lift $\dot{w}\in N_G(T)$ of $w$. Moreover, $G_{(T',\alpha)}^{\circ}$ only depends on $X^{T'}_{\alpha}$ and $G^{T'}$ and therefore, if $w\in W_{(T',\alpha)}$, $G_{w\cdot (T',\alpha)}^{\circ}=G_{(T',\alpha)}^{\circ}$. We conclude that $w\cdot T'=T'$ and $w\cdot \alpha=\alpha$. Conversely, we assume that $w\cdot T'=T'$ and $w\cdot \alpha=\alpha$. Then, $G^{T'}=G^{w\cdot T'}$ and by definition, $w\cdot\alpha=\alpha$. This concludes.
\end{proof}

\begin{lemma}
\label{lemma:weylgroupsubgroup}
\begin{enumerate}
 \item Let $(\lambda,\alpha)\in \SQ_{X,G}$. We have an inclusion $W^{\lambda}\subset W_{(\lambda,\alpha)}$ where $W^{\lambda}=N_{G^{\lambda}}(T)/T$ is the Weyl group of $G^{\lambda}$. Moreover, $W^{\lambda}$ is a normal subgroup of $W_{(\lambda,\alpha)}$.
 \item Let $(T',\alpha)\in \SQ''_{X,G}$. Then, we have an inclusion $W^{T'}\subset W_{(T',\alpha)}$ where $W^{T'}$ is the Weyl group of $G^{T'}$. Moreover, $W^{\lambda}$ is a normal subgroup of $W_{(T',\alpha)}$.
\end{enumerate}
\end{lemma}
\begin{proof}
 We only prove the first statement. The proof of the second one is analogous. Let $\dot{w}\in N_{G^{\lambda}}(T)$. Then, $\dot{w}\in G^{\lambda}$ and therefore, $\dot{w}\cdot \lambda=\lambda$. Moreover, $G^{\lambda}$ is connected and so preserves each connected component of $X^{\lambda}$. Therefore, $w\cdot\overline{(\lambda,\alpha)}=\overline{(\lambda,\alpha)}$. To prove the normality, we let $\dot{w}\in N_{G^{\lambda}}(T)$ and $\dot{w}'\in N_G(T)$ such that $\dot{w}$ preserves $(\lambda,\alpha)$. In particular, $\dot{w}'G^{\lambda}\dot{w}'^{-1}=G^{\dot{w}'\cdot\lambda}=G^{\lambda}$. Then, $\dot{w}'\dot{w}\dot{w}'^{-1}\in \dot{w}'G^{\lambda}\dot{w}'^{-1}=G^{\lambda}$. This concludes.
\end{proof}

\begin{definition}
\label{definition:Weylgrouprep}
For $(\lambda,\alpha)\in\SP_{X,G}$, we define $\overline{W}_{(\lambda,\alpha)}\coloneqq W_{(\lambda,\alpha)}/W^{\lambda}$. Similarly, we define $\overline{W}_{(T',\alpha)}\coloneqq W_{(T',\alpha)}/W^{T'}$.
\end{definition}

For $(\lambda,\alpha)\in\SQ_{X,G}$, we let $N_{G}^{(\lambda,\alpha)}(T)=\{w\in N_G(T)\mid w\cdot{\overline{(\lambda,\alpha)}}=\overline{(\lambda,\alpha)}\}$. It only depends on $\overline{(\lambda,\alpha)}$ and so we may sometimes write $N_G^{\overline{(\lambda,\alpha)}}(T)$.

\begin{lemma}
\label{lemma:relWeylquotient}
 We have $\overline{W}_{(\lambda,\alpha)}\cong N_{G}^{(\lambda,\alpha)}(T)/N_{G^{\lambda}}(T)$.
\end{lemma}
\begin{proof}
 This is immediate from the definitions.
\end{proof}

\subsection{Equivariant sheaves}
We recall some (very) basic facts regarding equivariant constructible sheaves or (monodromic) mixed Hodge modules on an algebraic variety with a $W$-action.

\begin{proposition}
\label{proposition:groupactionPerversesheaf}
 Let $f\colon X\rightarrow S$ be a morphism between algebraic varieties. We assume that $f$ is equivariant for the action of an algebraic group $W$ where $W$ acts trivially on $S$. If $\underline{\CF}\in\CD^+_W(\MMHM(X))$, we have a morphism $W\rightarrow\Aut_{\CD^+(\MMHM(S))}(f_*\underline{\CF})$. In particular, if $\underline{\CF}\in\MMHM_W^{\BoZ}(X)$ is a pure cohomologically graded mixed Hodge module (\S\ref{subsection:cohgradedmhm}) and $f$ is proper, then we have a morphism of groups $W\rightarrow\Aut_{\MMHM^{\BoZ}(S)}(f_*\underline{\CF})$. 
\end{proposition}
\begin{proof}
 Let $\underline{\CF}\in\CD^+_W(\MMHM(X))$ be $W$-equivariant. The complex $f_*\underline{\CF}$ is $W$-equivariant on $S$ and the action on $S$ is trivial. This gives the morphism $W\rightarrow \Aut_{\CD^+(\MMHM(S))}(f_*\underline{\CF})$. The proof is the same for $\underline{\CF}\in\MMHM_W^{\BoZ}(X)$, given that proper maps preserve pure cohomologically graded mixed Hodge modules.
\end{proof}

\begin{corollary}
\label{corollary:decompositionisotypiccomponents}
 In the situation of Proposition~\ref{proposition:groupactionPerversesheaf}, assuming that the group $W$ is finite, $f$ is proper and $\underline{\CF}\in\MMHM_W^{\BoZ}(X)$, we have a decomposition of $f_*\underline{\CF}$ into $W$-isotypic components:
 \[
  f_*\underline{\CF}\cong \bigoplus_{\chi\in\Irr(W)}\underline{\CF}_{\chi}\otimes M_{\chi}
 \]
where $\Irr(\chi)$ is the set of irreducible characters of $W$, $M_{\chi}$ is the corresponding irreducible representation of $W$ and $\underline{\CF}_{\chi}\in\MMHM^{\BoZ}(S)$.
\end{corollary}

\subsection{Weyl group equivariance}

\begin{lemma}
 Let $(\lambda,\alpha)\in \SQ_{X,G}$. Then, the complex of mixed Hodge modules $(\pi_{(\lambda,\alpha)})_*\underline{\BoQ}_{X^{\lambda}_{\alpha}/G^{\lambda}}$ is $\overline{W}_{(\lambda,\alpha)}$-equivariant.
\end{lemma}
\begin{proof}
 We let $N_G^{\overline{(\lambda,\alpha)}}(T)$ act on $X^{\lambda}_{\alpha}$ via the inclusion in $X$ and on $G^{\lambda}$ by conjugation. This gives a $N_G^{\overline{(\lambda,\alpha)}}(T)$-action on the quotient stack $X^{\lambda}_{\alpha}/G^{\lambda}$ and on its good moduli space $X^{\lambda}_{\alpha}\cms G^{\lambda}$, so that $\pi_{(\lambda,\alpha)}$ is equivariant. Moreover, the action of the subgroup $N_{G^{\lambda}}(T)$ is trivial on the stack and its good moduli space, giving (using Lemma~\ref{lemma:relWeylquotient}) the looked after $\overline{W}_{(\lambda,\alpha)}$-action on $\pi_{(\lambda,\alpha)}\colon X^{\lambda}_{\alpha}/G^{\lambda}\rightarrow X^{\lambda}\cms G^{\lambda}$. Then, since $\underline{\BoQ}_{X^{\lambda}_{\alpha}/G^{\lambda}}$ is $\overline{W}_{(\lambda,\alpha)}$-equivariant, $(\pi_{(\lambda,\alpha)})_*\underline{\BoQ}_{X^{\lambda}_{\alpha}/G^{\lambda}}$ is $\overline{W}_{(\lambda,\alpha)}$-equivariant.
\end{proof}

\begin{lemma}
\label{lemma:Wactionimath}
 There is a $\overline{W}_{(\lambda,\alpha)}$-action on $(\imath_{(\lambda,\alpha)})_*(\pi_{(\lambda,\alpha)})_*\underline{\BoQ}_{X^{\lambda}_{\alpha}/G^{\lambda}}$.
\end{lemma}
\begin{proof}
 The map $\imath_{(\lambda,\alpha)}\colon X^{\lambda}_{\alpha}\cms G^{\lambda}\rightarrow X\cms G$ is $\overline{W}_{(\lambda,\alpha)}$-equivariant for the trivial action on $X\cms G$.
\end{proof}

\begin{corollary}
 The cohomologically graded mixed Hodge module $(\imath_{(\lambda,\alpha)})_*(\pi_{(\lambda,\alpha)})_*\underline{\BoQ}_{X^{\lambda}_{\alpha}/G^{\lambda}}$ decomposes into $\overline{W}_{(\lambda,\alpha)}$-isotypic components.
\end{corollary}
\begin{proof}
 This is an immediate combination of Corollary~\ref{corollary:decompositionisotypiccomponents} with Lemma~\ref{lemma:Wactionimath}. Proposition~\ref{proposition:semisimplicitypistar} implies that $(\pi_{(\lambda,\alpha)})_*\underline{\BoQ}_{X^{\lambda}_{\alpha}/G^{\lambda}}$ is a cohomologically graded mixed Hodge module.
\end{proof}

\subsection{\'Etale slices and preferred representatives}

Let $x\in X$ be a closed point such that the orbit $G\cdot x$ is closed. We let $G_x$ be the stabilizer group of $x$, a reductive group. We let $N_x\subset X$ be a normal slice to $G\cdot x$ through $x$. In particular, we have a $G_x$-equivariant isomorphism $\Tan_xX\cong\Tan_x(G\cdot x)\oplus \Tan_xN_x$. Moreover, as $G_x$-representations, we have $\Tan_x(G\cdot x)\cong \Fg/\Fg_x$ and $\Fg\cong \Fg_x\oplus (\Fg/\Fg_x)$. Up to replacing $x$ by a $G$-conjugate, we may assume that $T_x\coloneqq T\cap G_x$ is a maximal torus of $G_x$.

We note that $N_x$ has a $G_x$-fixed point $x$. Moreover, it is linear in the sense that, by Luna \'etale slice theorem \cite{luna1973slices}, there is a $G_x$-equivariant \'etale map between a Zariski open neighbourhood of $x$ in $N_x$ and a Zariski open neighbourhood of the origin inside $\Tan_xN_x$. \'Etale slices will be used to linearize the study of the topology of $X/G$ in \S\ref{section:localneighbourhoodsmoothstacks}.

\begin{lemma}
 There is a natural map $\SQ_{N_x,G_x}\rightarrow\SQ_{X,G}$.
\end{lemma}
\begin{proof}
 Let $\lambda\in X_*(T_x)$. Via the inclusion $T_x\subset T$, we may see $\lambda\in X_*(T)$. Moreover, $N_x^{\lambda}$ has a unique connected component. Therefore, the map of the lemma can be defined as the map $\lambda\mapsto (\lambda,\alpha)$ where $\alpha\in\pi_0(X^{\lambda})$ is the connected component containing $x$.
\end{proof}

\begin{lemma}
\label{lemma:fixedpointstorus}
 Let $T'\subset G_x$ be a torus. Then, $(\Fg/\Fg_x)^{T'}\cong \Fg^{T'}/\Fg_x^{T'}$.
\end{lemma}
\begin{proof}
 As representations of $G_x$ (and therefore of $T'\subset G_x$), we have $\Fg\cong\Fg_x\oplus (\Fg/\Fg_x)$. Taking $T'$-invariants on both sides gives $\Fg^{T'}\cong \Fg_x^{T'}\oplus(\Fg/\Fg_x)^{T'}$. We get the result by taking the quotient by $\Fg_x^{T'}$.
\end{proof}

We let $T'\in\SQ''_{N_x,G_x}$ be such that $T'=G_{x,T'}^{\circ}$ (Definition~\ref{definition:preferredtorus}).

\begin{lemma}
\label{lemma:preferredtori}
 We let $\alpha\in \pi_0(X^{T'})$ be the connected component containing $x$. Seeing $(T',\alpha)\in\SQ'_{X,G}$, we have $G_{(T',\alpha)}^{\circ}=T'$. In words, if $T'$ is a preferred torus for $N_x,G_x$, then it is a preferred torus for $X,G$.
\end{lemma}
\begin{proof}
We have to check that $T'$ is the neutral component of $\ker(G^{T'}\rightarrow\Aut(X^{T'}_{\alpha}))\cap Z(G^{T'})$, which by definition is $G_{(T',\alpha)}^{\circ}$. We have clearly $T'\subset Z(G^{T'})$. Moreover, it acts trivially on an open subset of $X^{T'}_{\alpha}$ by \'etale slices since it acts trivially on $\Tan_x X^{T'}\cong (\Fg/\Fg_x)^{T'}\oplus (\Tan_xN_x)^{T'}$. Therefore, $T'\subset G_{(T',\alpha)}^{\circ}$. Conversely, if $g\in Z(G^{T'})$ acts trivially on $X^{T'}_{\alpha}$, it stabilizes $x$ and so $g\in G_x$. It also act trivially on $\Tan_x(X^{T'})\cong (\Fg/\Fg_x)^{T'}\oplus(\Tan_x N_x)^{T'}$ (Lemma~\ref{lemma:fixedpointstorus}). Therefore, it acts trivially on $N_x^{T'}$. Therefore, $g\in G_{x,T'}$. This concludes.
\end{proof}

Using Lemma~\ref{lemma:preferredtori}, we obtain a map $\gamma_x\colon\SP_{N_x,G_x}\rightarrow\SP_{X,G}$ sending $\overline{T'}$ for $T'\in \CT(T_x)$ such that $T'=G_{x,T'}^{\circ}$ to the pair $\overline{(T',\alpha)}$ where we see $T'\in\CT(T)$ and we let $\alpha$ be the connected component of $X^{T'}$ containing $x$.

Let $T'\in\CT(T_x)$ be such that $T'=G_{T'}^{\circ}$ and $\alpha\in\pi_0(X^{T'})$ be the connected component containing $x$. We let $W_x$ be the Weyl group of $x$, $W_{x,T'}=\{w\in W\mid w\cdot T'=T'\}$ (according to \S\ref{subsection:Weylgroups}), $\overline{W}_{x,T'}=W_{x,T'}/W_x^{T'}$ (Definition~\ref{definition:Weylgrouprep}), etc.

\begin{lemma}
\label{lemma:injectmorphgroups}
 There exists a natural injective morphism of groups $\overline{W}_{x,T'}\rightarrow \overline{W}_{(T',\alpha)}$.
\end{lemma}
\begin{proof}
 Recall that $\overline{W}_{x,T'}=N^{T'}_{G_x}(T_x)/N_{G_x^{T'}}(T_x)$ where
 \[
  N_{G_x}^{T'}(T_x)=\{w\in N_{G_x}(T_x)\mid wT'w^{-1}=T'\}
 \]
 and similarly, $\overline{W}_{(T',\alpha)}=N_G^{(T',\alpha)}(T)/N_{G^{T'}}(T)$ (Lemma~\ref{lemma:relWeylquotient}, Lemma~\ref{lemma:relWeylpreferredtorus} and Lemma~\ref{lemma:preferredtori}).

 Let $w\in N_{G_x}^{T'}(T_x)$. Then, $T, w^{-1}Tw\subset G^{T'}$ are maximal tori. Therefore, there exists $h\in G^{T'}$ such that $hTh^{-1}=w^{-1}Tw$. Then, $wh\in N_G(T)$. Moreover, $whT'h^{-1}w^{-1}=wT'w^{-1}=T'$. In addition, since $G^{T'}$ is connected, it preserves connected components of $X^{T'}$ and so $h\cdot\alpha=\alpha$. Since $w$ fixes $x$, it must also preserve $\alpha$ and so $wh\cdot\alpha=\alpha$. The morphism of groups $\overline{W}_{x,T'}\rightarrow \overline{W}_{T',\alpha}$ sends the class of $w$ to the class of $wh$. We prove it is well-defined. Let $w,w'\in N_{G_x}^{T'}(T_x)$ with the same class in the quotient $\overline{W}_{x,T'}$ (that is, $ww'^{-1}\in N_{G_x^{T'}}(T_x)$) and $h,h'\in G^{T'}$ be such that $wh,w'h'\in N_G(T)$. Then, $whh'^{-1}w'^{-1}\in N_G(T)$ and for $t'\in T'$, 
 \[
 \begin{aligned}
  t'whh'^{-1}w'^{-1}t'^{-1}&=w\underbrace{w^{-1}t'w}_{\in T'}hh'^{-1}w'^{-1}t'^{-1}\\
						  &=whh'^{-1}\underbrace{w^{-1}t'w}w'^{-1}t'^{-1}\quad\text{since $h,h'\in G^{T'}$}\\
						  &=whh'^{-1}w'^{-1} \quad\text{since $ww'^{-1}\in N_{G_x^{T'}}(T_x)\subset G^{T'}$}
\end{aligned}
 \]
 and so the classes of $wh$ and $w'h'$ in $\overline{W}_{(T',\alpha)}$ coincide.
 
 We prove that the map constructed is a morphism of groups. We prove the compatibility with products. Let $w,w'\in N_{G_x}^{T'}(T_x)$ and $h,h',h''\in G^{T'}$ such that $wh, w'h', ww'h''\in N_G(T)$. For $t'\in T'$, we have
\[
\begin{aligned}
 t'ww'h''h'^{-1}w'^{-1}h^{-1}w^{-1}t'^{-1}&=ww'\underbrace{w'^{-1}w^{-1}t'ww'}_{\in T'}h''h'^{-1}w'^{-1}h^{-1}w^{-1}t'^{-1}\\
 &=ww'h''h'^{-1}\underbrace{w'^{-1}w^{-1}t'ww'}w'^{-1}h^{-1}w^{-1}t'^{-1}\quad\text{ since $h',h''\in G^{T'}$}\\
 &=ww'h''h'^{-1}w'^{-1}\underbrace{w^{-1}t'w}_{\in T'}h^{-1}w^{-1}t'^{-1}\\
 &=ww'h''h'^{-1}w'^{-1}h^{-1}\underbrace{w^{-1}t'w}w^{-1}t'^{-1}\quad\text{ since $h\in G^{T'}$}\\
 &=ww'h''h'^{-1}w'^{-1}h^{-1}w^{-1}\,.
\end{aligned}
\]
This means that $ww'h''h'^{-1}w'^{-1}h^{-1}w^{-1}\in N_{G^{T'}}(T)$, and so the classes of $whw'h'$ and $ww'h''$ coincide.

It only remains to prove the injectivity. Let $w\in N_{G_x}^{T'}(T_x)$ and $h\in G^{T'}$ be such that $wh\in N_{G^{T'}}(T)$. Then, for any $t'\in T'$, \[
\begin{aligned}
t'wt'^{-1}&=t'whh^{-1}t'^{-1}\\
          &=t'wht'^{-1}h^{-1}\quad\text{since $h\in G^{T'}$}\\
          &=w\quad\text{since $wh\in G^{T'}$}
\end{aligned}\]
Therefore, $w\in G_x^{T'}$. Since $w$ normalizes $T_x$, $w\in N_{G_x^{T'}}(T_x)$. This proves the injectivity.
\end{proof}

\subsection{Weyl group actions and \'etale slices}

Let $x\in X$ be such that $G\cdot x$ is closed. Up to $G$-conjugacy, we may assume that $T_x\coloneqq T\cap G_x$ is a maximal torus of $G_x$. We let $T'\subset T_x$ be a torus such that $T'=G_{x,T'}^{\circ}$ (Definition~\ref{definition:preferredtorus}). By Lemma~\ref{lemma:preferredtori}, we also have $T'=G_{(T',\alpha)}^{\circ}$ where $\alpha\in\pi_0(X^{T'})$ is the connected component containing $x$.

\begin{lemma}
\label{lemma:normalizestorusfixedpoints}
The group $N_G^{T'}(T)$ normalizes $G^{T'}$.
\end{lemma}
\begin{proof}
 Let $\dot{w}\in N_G^{T'}(T)$. In particular, $\dot{w}T'\dot{w}^{-1}=T'$. Let $g\in G^{T'}$ and $t'\in T'$. We have
\[
\begin{aligned}
 t'\dot{w}g\dot{w}^{-1}t'^{-1}&=\dot{w}\underbrace{\dot{w}^{-1}t'\dot{w}}_{\in T'}g\dot{w}^{-1}t'^{-1}\\
 &=\dot{w}g\dot{w}^{-1}\quad\text{ since $g\in G^{T'}$}\,.
\end{aligned}
\]
This concludes.
\end{proof}

\begin{lemma}
 We have an isomorphism
 \[
  \overline{W}_{x,T'}\cong\{w\in \overline{W}_{(T',\alpha)}\mid w\cdot (G^{T'}\cdot x)=(G^{T'}\cdot x)\}\,,
 \]
 therefore identifying the image of the injective morphism of Lemma~\ref{lemma:injectmorphgroups}. In other words, $\overline{W}_{x,T'}$ is the stabilizer of $G^{T'}\cdot x\in X^{T'}\cms G^{T'}$ inside $\overline{W}_{(T',\alpha)}$.
\end{lemma}
\begin{proof}
 It is clear that elements of $\overline{W}_{x,T'}$ stabilize $G^{T'}\cdot x$. Let $w\in \overline{W}_{(T',\alpha)}$ and $\dot{w}\in N_G(T)$ with $\dot{w}T'\dot{w}^{-1}=T'$ be a lift of $w$. We have $w\cdot(G^{T'}\cdot x)=G^{T'}\cdot(\dot{w}\cdot x)$ by Lemma~\ref{lemma:normalizestorusfixedpoints}. We assume that $w\cdot (G^{T'}\cdot x)=G^{T'}\cdot x$. By the definition of the injective morphism in Lemma~\ref{lemma:injectmorphgroups}, it suffices to find $k\in G^{T'}$ such that $\dot{w}k\in N_{G_x}^{T'}(T_x)$. There exists $g\in G^{T'}$ such that $\dot{w}g\in G_x$. Since $T_x$ and $(\dot{w}g)^{-1}T_x(\dot{w}g)$ are maximal tori of $G_x^{T'}$, there exists $h\in G_x^{T'}$ such that $hg^{-1}\dot{w}^{-1}\in N_{G_x}^{T'}(T_x)$. Setting $k=gh^{-1}$ concludes.
\end{proof}

\subsection{Inverse image}
 Let $x\in X$ be such that the orbit $G\cdot x$ is closed. We let $T'\subset T_x$ be a torus, i.e. $T'\in\CT(T_x)$. The map $\imath_{T'}\colon X^{T'}\cms G^{T'}\rightarrow X\cms G$ is finite (Lemma~\ref{lemma:finitemap}). We write $\imath_{T'}^{-1}(G\cdot x)=\{G^{T'}\cdot x_i\colon 1\leq i\leq N\}$ where $x_i\in X^{T'}$. We may write $x_i=g_i x$ for some $g_i\in G$.

 \begin{lemma}
  There is an inclusion $N_G(T')\subset N_G(G^{T'})$.
 \end{lemma}
 \begin{proof}
  Let $g\in G^{T'}$ and $h\in N_G(T')$. Then, for $t'\in T'$, we have
  \[
  \begin{aligned}
   t'hgh^{-1}t'^{-1}&=h\underbrace{h^{-1}t'h}_{\in T'}gh^{-1}t'^{-1}\\
					&=hg\underbrace{h^{-1}t'h}h^{-1}t'^{-1}\quad\text{ since $g\in G^{T'}$}\\
					&=hgh^{-1}\,.
  \end{aligned}
  \]
 \end{proof}

Let $T'\in\CT(T_x)$. We let $\SP_{N_x,G_x}^{T'}\subset \SP_{N_x,G_x}$ be the subset of classes of tori $\overline{T''}$ where $T''\subset T_x$ is a tori $N_G(T)$-conjugated (equivalently, $W$-conjugated) to $T'$.

 \begin{lemma}
 \label{lemma:bijection}
  The $W_x$-action on $\SP_{N_x,G_x}$ leaves $\SP_{N_x,G_x}^{T'}$ stable and there is  a natural map $\imath_{T'}^{-1}(G\cdot x)\rightarrow \SP_{N_x,G_x}^{T'}/W_{x}$. It induces a bijection.
 \[
  \imath_{T'}^{-1}(G\cdot x)/\overline{W}_{(T',\alpha)}\rightarrow \SP_{N_x,G_x}^{T'}/W_{x}\,.
 \]
 where $\overline{W}_{(T',\alpha)}= N_G^{T'}(T)/N_G(T)$.
  \end{lemma}
\begin{proof}

We show that $W_{x}$ acts on $\SP_{N_x,G_x}^{T'}$. Let $T''\in \SP_{N_x,G_x}^{T'}$ and $\dot{w}\in N_{G_x}(T_x)$, we have $T''\subset T_x\subset \dot{w}^{-1}T\dot{w}$, and so $\dot{w}^{-1}T\dot{w}$ and $T$ are maximal tori of $G^{T''}$. Let $h\in G^{T''}$ be such that $h^{-1}\dot{w}^{-1}T\dot{w}h=T$. Then, $\dot{w}h\in N_G(T)$ and $\dot{w}hT''
(\dot{w}h)^{-1}=\dot{w}T''\dot{w}^{-1}$. Consequently, $\dot{w}T''\dot{w}^{-1}\in \SP_{N_x,G_x}^{T'}$

Let $G^{T'}\cdot x_i\in \imath_{T'}^{-1}(G\cdot x)$. The torus $g_i^{-1}T'g_i$ is contained in $G_x$ since $x_i=g_ix\in X^{T'}$. The map $\imath_{T'}^{-1}(G\cdot x)\rightarrow \SP_{N_x,G_x}^{T'}/W_{x}$ sends $G^{T'}\cdot x_i$ to the conjugacy class of $g_i^{-1}T'g_i$. The map is well-defined, as if $g'\in G^{T'}$ is such that $x_i=g'\cdot x$, then $g'^{-1}g_i\in G_x$ and therefore, the tori $g_i^{-1}T'g_i$ and $g'^{-1}T'g'$ are conjugated in $G_x$.

We now construct a map $\SP_{N_x,G_x}^{T'}\rightarrow\imath_{T'}^{-1}(G\cdot x)/\overline{W}_{T'}$. Let $T''\subset T_x$ be a torus and $g\in N_G(T)$ such that $gT''g^{-1}=T'$. The map sends the conjugacy class of $T''$ to the orbit $G^{T'}\cdot gx$.

This map is well-defined. Indeed, if $g'\in N_G(T)$ is such that $g'T''g'^{-1}=T'$, then $g'g^{-1}\in N_G^{T'}(T)$ and so the $G^{T'}$-orbits $G^{T'}\cdot gx$ and $G^{T'}\cdot g'x$ are conjugated under $\overline{W}_{T'}$. 

It remains to check that the map is $W_{x}$-invariant. Let $T''\in\CT(T_x)$ be $N_G(T)$ conjugated to $T'$, $\dot{w}\in N_{G_x}^{T'}(T_x)$, $h\in G^{T'}$ such that $\dot{w}h\in N_G(T)$ (Proof of Lemma~\ref{lemma:injectmorphgroups}) and $g,g'\in N_G(T)$ be such that $gT''g^{-1}=T'$ and $g'\dot{w}hT''h^{-1}\dot{w}^{-1}g'^{-1}=T'$. Then, $g'\dot{w}hg^{-1}\in N_G(T)$ preserves $T'$ and sends $G^{T'}\cdot gx$ to $G^{T'}\cdot g'\dot{w}hx$. This concludes.
\end{proof}

\subsection{Relative Weyl group}
\label{subsection:relativeWeylgroups}
In \S\ref{subsection:Weylgroups}, we defined groups $\overline{W}_{(\lambda,\alpha)}$ associated to pairs $(\lambda,\alpha)\in\SQ_{X,G}$. In this section, we explain how these groups relate to the more familiar relative Weyl groups associated to a pair of a Levi subgroup of a reductive group.

\subsubsection{Relative Weyl group of a Levi subgroup}
\label{subsubsection:relWeyllevi}
Let $G$ be a reductive group and $L\subset G$ a Levi subgroup. Then, the relative Weyl group $W_{G,L}$ of $G$ relative to $L$ is defined as $W_{G,L}\coloneqq N_G(L)/L$. In particular, if $T\subset G$ is a maximal torus, we recover the definition of the Weyl group $W_G\coloneqq W_{G,T}$ of $G$.

 Let $T\subset L\subset G$ be a Levi subgroup of $G$ and a maximal torus of $L$ (and hence also maximal in $G$). We define $W_{G}^L\coloneqq \{w\in W\mid wLw^{-1}=L\}$. The group $W_{L}=N_L(T)/T$ is naturally a normal subgroup of $W_{G}^L$. We let $\overline{W}_G^L\coloneqq W_G^L/W_L$.
 
\begin{lemma}
\label{lemma:relWeylgroups}
The natural map $\overline{W}_G^L\rightarrow W_{G,L}$ is an isomorphism of groups.
\end{lemma}
\begin{proof}
 We construct an inverse to this map. Let $w\in W_{G,L}$ and $\dot{w}\in N_G(L)$ a lift of $w$. Then, $T$ and $\dot{w}T\dot{w}^{-1}$ are maximal tori of $L$. Therefore, there exists $l\in L$ such that $l\dot{w}\in N_G(T)$. The inverse sends $w$ to the class of $l\dot{w}$ in $\overline{W}_G^L$. It is well-defined, as if $l'\in L$ is such that $l'\dot{w}\in N_G(T)$, then $l\dot{w}(l'\dot{w})^{-1}=ll'^{-1}\in L\cap N_G(T)=N_L(T)$ and the class of $l\dot{w}(l'\dot{w})^{-1}$ is in $W_L$.
\end{proof}

\subsubsection{Relative Weyl group of sub-varieties}
\label{subsubsection:relWeylsubvar}
\begin{corollary}
\label{corollary:Weylgroup}
 Let $G$ be a reductive group and $X=\Fg\coloneqq \Lie(G)$ with the adjoint action. Then, $\SQ_{X,G}=X_*(T)$ and for $\lambda\in X_*(T)$, $\overline{W}_{\lambda}=W_{G,G^{\lambda}}$.
\end{corollary}
\begin{proof}
 We have $W_{\lambda}=\{w\in W\mid w\cdot\overline{\lambda}=\overline{\lambda}\}$. The rest follows from the definitions at the beginning of \S\ref{subsubsection:relWeyllevi} and Definition~\ref{definition:Weylgrouprep}.
\end{proof}

\begin{proposition}
Let $X$ be a smooth affine $G$-variety and $(\lambda,\alpha)\in\SQ_{X,G}$. Then, $\overline{W}_{(\lambda,\alpha)}\cong \{w\in W_{G,G^{\lambda}}\mid w\cdot X^{\lambda}_{\alpha} =X^{\lambda}_{\alpha}\}$.
\end{proposition}
\begin{proof}
 By definition, we have $\overline{W}_{(\lambda,\alpha)}=W_{(\lambda,\alpha)}/W_{G^{\lambda}}$. The result comes then from Lemma~\ref{lemma:relWeylgroups} and the fact that $W_{(\lambda,\alpha)}=\{w\in W\mid wG^{\lambda}w^{-1}=G^{\lambda}, w\cdot X^{\lambda}_{\alpha}=X^{\lambda}_{\alpha}\}$.
\end{proof}

\section{Parabolic induction for affine varieties with reductive group action}
\label{section:parabolicinduction}
We recall how to construct the sheafified parabolic induction, see for example \cite{kontsevich2011cohomological} or \cite[\S2]{hennecart2024cohomological}.
\subsection{The induction diagram}
\label{subsection:inductiondiagram}
Let $G$ be a reductive group and $X$ a smooth affine $G$-variety. To a cocharacter $\lambda\colon\BoG_{\rmm}\rightarrow G$, we associate 
\begin{enumerate}
 \item $G^{\lambda}\coloneqq \{g\in G\mid \forall t\in T,  \lambda(t)g\lambda(t)^{-1}=g\}$, the centraliser of $\lambda$, a Levi subgroup of $G$,
 \item $X^{\lambda}\coloneqq \{x\in X\mid\forall t\in T,  \lambda(t)\cdot x=x\}$, the fixed locus of $\lambda$, a smooth $G^{\lambda}$-variety, which we decompose as a disjoint union of its connected components $X^{\lambda}=\bigsqcup_{\alpha\in\pi_0(X^{\lambda})}X^{\lambda}_{\alpha}$,
 \item $P_{\lambda}=G^{\lambda\geq0}\coloneqq \{g\in G\mid \lim_{t\rightarrow 0}\limits \lambda(t)g\lambda(t)^{-1}\text{ exists}\}$, a parabolic subgroup of $G$,
 \item $X_{\alpha}^{\lambda\geq 0}\coloneqq \{x\in X\mid\lim_{t\rightarrow 0}\limits \lambda(t)\cdot x\in X^{\lambda}_{\alpha}\}$ for any $\alpha\in \pi_0(X^{\lambda})$, a smooth $P_{\lambda}$-variety.
\end{enumerate}
We denote by $\SQ_{X,G}$ the set of pairs $(\lambda,\alpha)$ where $\lambda\in X_*(T)$ is a cocharacter of $T$ and $\alpha\in \pi_0(X^{\lambda})$. For any $(\lambda,\alpha)\in \SQ_{X,G}$, we have the commutative induction diagram
\[\begin{tikzcd}
	& {X^{\lambda\geq 0}_{\alpha}/P_{\lambda}} \\
	{X^{\lambda}_{\alpha}/G^{\lambda}} && {X/G} \\
	{X_{\alpha}^{\lambda}\cms G^{\lambda}} && {X\cms G}
	\arrow["{q_{(\lambda,\alpha)}}"', from=1-2, to=2-1]
	\arrow["{p_{(\lambda,\alpha)}}", from=1-2, to=2-3]
	\arrow["{\pi_{(\lambda,\alpha)}}"', from=2-1, to=3-1]
	\arrow["\pi", from=2-3, to=3-3]
	\arrow["\imath_{(\lambda,\alpha)}",from=3-1, to=3-3]
\end{tikzcd}\]
where $V\cms G\coloneqq \Spec(\BoC[V]^G)$ is the affine GIT quotient of $V$ by $G$, $\imath_{(\lambda,\alpha)}$ is the map induced by the equivariant closed immersion $X^{\lambda}_{\alpha}\rightarrow X$, $p_{\lambda}$ is induced by the equivariant closed immersion $X^{\lambda\geq 0}_{\alpha}\subset X$, and $q_{(\lambda,\alpha)}$ is induced by the limit map $\tilde{q}_{(\lambda,\alpha)}\colon x\mapsto \lim_{t\rightarrow0}\limits \lambda(t)\cdot x$, a vector bundle by the Bia\l inicky-Birula formalism \cite{bialynicki1973some}.

\begin{lemma}
\label{lemma:attractclosed}
For any $(\lambda,\alpha)\in\SQ_{X,G}$, the varieties $X^{\lambda}$ and $X^{\lambda\geq 0}_{\alpha}$ are smooth closed subvarieties of $X$.
\end{lemma}
\begin{proof}
 This follows from the Bia\l inicky-Birula formalism on affine varieties, see for example \cite[\S2]{hausel2022very}, in particular the proof of Lemma 2.1, Lemma 2.2, Lemma 2.3 and Remark 2.4 of \emph{op. cit.} for a modern treatment.
\end{proof}

Let $(\lambda,\alpha)\in\SQ_{X,G}$. We can decompose the restriction to $X_{\alpha}^{\lambda}$ of the tangent bundle of $X$:
\begin{equation}
\label{equation:restrictiontangentbundlefixedlocus}
 \Tan_{|X^{\lambda}_{\alpha}} X\cong \Tan^-_{|X^{\lambda}_{\alpha}} X\oplus \Tan^0_{|X^{\lambda}_{\alpha}} X\oplus \Tan^+_{|X^{\lambda}_{\alpha}} X\,,
\end{equation}
according to the $\lambda$-weights. We have $\Tan X^{\lambda}_{\alpha}\cong \Tan_{|X^{\lambda}_{\alpha}}^0 X$. We have a $P_{\lambda}$-equivariant isomorphism $X^{\lambda\geq 0}_{\alpha}\cong \Tan^+_{|X^{\lambda}_{\alpha}}X$ and a commutative diagram
\begin{equation}
\label{equation:vbstack}
\begin{tikzcd}
	{\Tan^+_{|X^{\lambda}_{\alpha}}X} & {} & {X^{\lambda\geq 0}_{\alpha}} \\
	& {X^{\lambda}_{\alpha}}
	\arrow["\sim", from=1-1, to=1-3]
	\arrow[from=1-1, to=2-2]
	\arrow["\tilde{q}_{(\lambda,\alpha)}",from=1-3, to=2-2]
\end{tikzcd}
\end{equation}

If $V$ is a representation of $G$ and $T$ a maximal torus of $G$, we denote by $\CW(V)$ the set of weights of $V$ counted with multiplicities. If $\lambda\in X_*(T)$ is a cocharacter, we let $\CW^{\lambda>0}(V)$ be the set of weights $\alpha$ of $V$ counted with multiplicities such that $\langle\lambda,\alpha\rangle>0$ (where we use the pairing between characters and cocharacters, see \S\ref{subsection:conventions}).

\begin{lemma}
\label{lemma:smoothnessqproperp}
 The map $q_{(\lambda,\alpha)}$ is a vector bundle stack of rank $r_{(\lambda,\alpha)}\coloneqq\rank (\Tan^+_{|X^{\lambda}_{\alpha}}X)-\#\CW^{\lambda>0}(\Fg)$ and the map $p_{(\lambda,\alpha)}$ is a proper and representable morphism of stacks.
\end{lemma}
\begin{proof}
When $X=V$ is a representation of $G$, this is {\cite[Lemma 2.1]{hennecart2024cohomological}}.

 The statement regarding $q_{(\lambda,\alpha)}$ follows from the commutative diagram \eqref{equation:vbstack} and the fact that the kernel of $P_{\lambda}\rightarrow G^{\lambda}$ is unipotent of dimension $\#\CW^{\lambda>0}(\Fg)$.
 
 We can write $X^{\lambda}_{\alpha}/P_{\lambda}\simeq \tilde{X}^{\lambda}_{\alpha}/G$ where $\tilde{X}^{\lambda}_{\alpha}\coloneqq X^{\lambda\geq 0}_{\alpha}\times^{P_{\lambda}}G$ (see \S\ref{subsection:conventions}). Then, the map $p_{(\lambda,\alpha)}$ can be identified with the map between quotient stacks induced by the $G$-equivariant morphism $X^{\lambda\geq 0}_{\alpha}\times^{P_{\lambda}}G\rightarrow X$, $(x,g)\mapsto g\cdot x$ which is proper by Lemma~\ref{lemma:attractclosed}. Indeed, it factors as
 \[
  X^{\lambda\geq 0}_{\alpha}\times^{P_{\lambda}}G\rightarrow X\times (G/P_{\lambda})\rightarrow X
 \]
 where the first map is $(x,g)\mapsto (g\cdot x,gP_{\lambda})$ and the second map is the projection on the first factor. The first map is a closed immersion while the second is clearly projective.
 \end{proof}

\begin{lemma}
\label{lemma:finitemap}
 The map $\imath_{(\lambda,\alpha)}$ is a finite map. If $G=T$ is a torus, $\imath_{\lambda}$ is a closed immersion.
\end{lemma}
\begin{proof}
The proof of the finiteness of $\imath_{\lambda}$ is given in \cite[Lemma 2.2]{hennecart2024cohomological} when $X=V$ is a representation of $G$. If $G=T$ is a torus, the closed immersion $X^{\lambda}_{\alpha}\rightarrow X$ is given at the level of function by the surjective morphism of algebras $\BoC[X]\rightarrow\BoC[X^{\lambda}_{\alpha}]$. Taking $T$-invariants on both sides preserves surjectivity and hence, $\imath_{(\lambda,\alpha)}\colon X^{\lambda}\cms T\rightarrow X\cms T$ is a closed immersion.
 
 We prove the general statement. The arguments are parallel to that of the proof \cite[Lemma 2.1]{meinhardt2019donaldson}. Namely, we consider the commutative diagram
\[\begin{tikzcd}
	&& {X^{\lambda\geq 0}_{\alpha}/P_{\lambda}} \\
	{X^{\lambda}_{\alpha}/G^{\lambda}} && {\Spec(\BoC[\tilde{X}^{\lambda\geq 0}_{\alpha}]^{G})} && {X/G} \\
	{X^{\lambda}_{\alpha}\cms G^{\lambda}} &&&& {X\cms G}
	\arrow["q_{(\lambda,\alpha)}"', from=1-3, to=2-1]
	\arrow["\jmath_{(\lambda,\alpha)}", bend left=30, from=2-1, to= 1-3]
	\arrow["{\pi_{\lambda\geq 0,\alpha}}", from=1-3, to=2-3]
	\arrow["p_{(\lambda,\alpha)}", from=1-3, to=2-5]
	\arrow["{\pi_{(\lambda,\alpha)}}"', from=2-1, to=3-1]
	\arrow["{\overline{q}_{(\lambda,\alpha)}}"', from=2-3, to=3-1]
	\arrow["{\overline{p}_{(\lambda,\alpha)}}", from=2-3, to=3-5]
	\arrow["\pi", from=2-5, to=3-5]
	\arrow["{\overline{\jmath}_{(\lambda,\alpha)}}"', bend right=20, from=3-1, to=2-3]
    \arrow["\imath_{(\lambda,\alpha)}"', bend right=20, from=3-1, to=3-5]
    \end{tikzcd}\]
where $\tilde{X}^{\lambda\geq 0}_{\alpha}\coloneqq X^{\lambda\geq 0}_{\alpha}\times^{P_{\lambda}}G$, so that we have an equivalence of stacks $X^{\lambda\geq 0}_{\alpha}/P_{\lambda}\simeq\tilde{X}^{\lambda\geq 0}_{\alpha}/G$. We now explain the definitions of the maps in this diagram.

The map $q_{(\lambda,\alpha)}$ is given by the equivariant map $(X^{\lambda\geq 0}_{\alpha},P_{\lambda})\rightarrow (X^{\lambda}_{\alpha},G^{\lambda})$

Therefore, the pullback of functions by $q_{(\lambda,\alpha)}$ induces a morphism
\[
 \overline{q}_{(\lambda,\alpha)}^*\colon\BoC[X^{\lambda}_{\alpha}]^{G^{\lambda}}\rightarrow \BoC[\tilde{X}^{\lambda\geq 0}_{\alpha}]^{G}
\]
whose dual is the morphism between affine schemes
\[
\overline{q}_{(\lambda,\alpha)}\colon \Spec(\BoC[\tilde{X}^{\lambda\geq 0}_{\alpha}]^{G^{\lambda}})\rightarrow\Spec(\BoC[X^{\lambda}_{\alpha}])=X^{\lambda}_{\alpha}\cms G^{\lambda}.
\]
The map $\jmath_{(\lambda,\alpha)}$ is given by the equivariant morphism $(X^{\lambda}_{\alpha},G^{\lambda})\rightarrow(X^{\lambda\geq 0}_{\alpha},P_{\lambda})$ which induces the morphism
\[
 \overline{\jmath}_{(\lambda,\alpha)}^*\colon \BoC[\tilde{X}^{\lambda\geq 0}_{\alpha}]^{G}\rightarrow\BoC[X^{\lambda}_{\alpha}]^{G^{\lambda}}
\]
and dually the morphism of schemes
\[
 \overline{\jmath}_{(\lambda,\alpha)}\colon\Spec(\BoC[X^{\lambda}_{\alpha}]^{G^{\lambda}})\rightarrow\Spec(\BoC[\tilde{X}^{\lambda\geq 0}_{\alpha}]^{G}).
\]

Since $q_{(\lambda,\alpha)}\circ \jmath_{(\lambda,\alpha)}=\id$, we have $\overline{q}_{(\lambda,\alpha)}\circ\overline{\jmath}_{(\lambda,\alpha)}=\id$. On the ring of functions, we have $\overline{\jmath}_{(\lambda,\alpha)}^*\circ \overline{q}_{(\lambda,\alpha)}^*=\id$ and so $\jmath_{(\lambda,\alpha)}^*$ is surjective. This tells us that $\overline{\jmath}_{(\lambda,\alpha)}$ is a closed immersion. Moreover, $\imath_{(\lambda,\alpha)}=\overline{p}_{(\lambda,\alpha)}\circ\overline{\jmath}_{(\lambda,\alpha)}$ and so it suffices to prove that $\overline{p}_{(\lambda,\alpha)}$ is a finite map. Since it is a morphism of finite type complex schemes, it suffices to prove that it is an integral morphism \cite[\href{https://stacks.math.columbia.edu/tag/01WJ}{Lemma 01WJ}]{stacks-project}.

Since $\tilde{p}_{(\lambda,\alpha)}\colon \tilde{X}^{\lambda\geq 0}_{\alpha}\rightarrow X$ (the map lifting $p_{(\lambda,\alpha)}$) is a projective morphism (Proof of Lemma~\ref{lemma:smoothnessqproperp}), $(\tilde{p}_{(\lambda,\alpha)})_*\CO_{\tilde{X}^{\lambda\geq 0}_{\alpha}}$ is a coherent $\CO_{X}$-module. Therefore, the map $\tilde{p}_{(\lambda,\alpha)}^*\colon\BoC[X]\rightarrow \BoC[\tilde{X}^{\lambda\geq 0}_{\alpha}]$ is finite and hence integral  \cite[\href{https://stacks.math.columbia.edu/tag/01WJ}{Lemma 01WJ}]{stacks-project}. We deduce that the map
\[
 \overline{p}_{(\lambda,\alpha)}^*\colon\BoC[X]^G\rightarrow \BoC[\tilde{X}^{\lambda\geq 0}_{\alpha}]^G
\]
is integral, and hence finite, as follows. Let $a\in\BoC[\tilde{X}^{\lambda\geq 0}_{\alpha}]^G$. It is solution of a monic polynomial with coefficients in $\BoC[X]$. By applying the Reynolds operator to the equation, we obtain a monic equation with coefficients in $\BoC[X]^G$ of which $a$ is a zero. This concludes.

\end{proof}

For $(\lambda,\alpha)\in \SQ_{X,G}$, we define $d_{(\lambda,\alpha)}\coloneqq\dim X^{\lambda}_{\alpha}/G^{\lambda}=\dim X^{\lambda}_{\alpha}-\dim G^{\lambda}$. The quantity $r_{(\lambda,\alpha)}$ is defined in Lemma~\ref{lemma:smoothnessqproperp} as the rank of the vector bundle stack $q_{(\lambda,\alpha)}$.

\begin{definition}
\label{definition:symmetricrepresentation}
\begin{enumerate}
 \item  We say that a representation $V$ of a reductive group $G$ is \emph{symmetric} if $V$ and its dual $V^*$ have the same sets of weights counted with multiplicities: $\CW(V)=\CW(V^*)$.
 \item We say that a smooth $G$-variety $X$ is \emph{symmetric} if for any $x\in X$ such that $G\cdot x$ is closed, $\Tan_xX$ is a symmetric representation of $G_x$, the stabilizer of $x$ in $G$.
\end{enumerate}
\end{definition}
It is clear that when $X=V$, the notions of symmetricity of (1) and (2) coincide.

For later use, we recall the following lemma.

\begin{lemma}
\label{lemma:numberweights}
 Let $X$ be a symmetric $G$-variety and $d=\dim X/G$. Then,
 \[
  d_{(\lambda,\alpha)}+2r_{(\lambda,\alpha)}=d.
 \]
\end{lemma}
\begin{proof}
 When $X=V$ is a finite-dimensional representation of $G$, this is \cite[Lemma 2.4]{hennecart2024cohomological}. In this context, we write
 \[
  \Tan_{|X^{\lambda}_{\alpha}}X\cong \Tan^-_{|X^{\lambda}_{\alpha}}X\oplus \Tan^0_{|X^{\lambda}_{\alpha}}X\oplus \Tan^+_{|X^{\lambda}_{\alpha}}X
 \]
as in \eqref{equation:restrictiontangentbundlefixedlocus}.
By symmetry of $X$, $\rank (\Tan^-_{|X^{\lambda}_{\alpha}}X)=\rank (\Tan^+_{|X^{\lambda}_{\alpha}}X)$. Therefore, $d=\dim \Tan_xX-\dim \Fg^{\lambda}-2\dim \Fg^{\lambda>0}=2r_{(\lambda,\alpha)}+d_{(\lambda,\alpha)}$.
\end{proof}

\subsection{Parabolic induction}
\label{subsection:parabolicinduction}
For $(\lambda,\alpha)\in \SQ_{X,G}$, we let $\underline{\BoQ}_{X^{\lambda}_{\alpha}/G^{\lambda}}^{\vir}\coloneqq \underline{\BoQ}_{X^{\lambda}_{\alpha}/G^{\lambda}}\otimes\SL^{-d_{(\lambda,\alpha)}/2}=\underline{\IC}(X^{\lambda}_{\alpha}/G^{\lambda})$ be the intersection complex monodromic mixed Hodge module of the smooth stack $X^{\lambda}_{\alpha}/G^{\lambda}$. We refer to \cite{davison2020cohomological} for the necessary background regarding monodromic mixed Hodge modules. Alternatively, for the purposes of this paper, the reader may prefer to consider instead constructible sheaves. The Tate twist $-\otimes\SL$ is then replaced by the shift $[-2]$. Given that the stack $X^{\lambda}_{\alpha}/G^{\lambda}$ is smooth and the Tate twist encoded by the superscript $\vir$, we have $\BD\underline{\BoQ}_{X^{\lambda}_{\alpha}/G^{\lambda}}^{\vir}\cong\underline{\BoQ}_{X^{\lambda}_{\alpha}/G^{\lambda}}^{\vir}$.

Since $q_{(\lambda,\alpha)}$ is smooth of relative dimension $r_{(\lambda,\alpha)}$ (Lemma~\ref{lemma:smoothnessqproperp}), we have $q_{(\lambda,\alpha)}^*\cong q_{(\lambda,\alpha)}^!\otimes\SL^{r_{(\lambda,\alpha)}}$ and thus we obtain
\[
 q_{(\lambda,\alpha)}^*\BD\underline{\BoQ}_{X^{\lambda}_{\alpha}/G^{\lambda}}^{\vir}\cong\BD\underline{\BoQ}_{X^{\lambda\geq 0}_{\alpha}/P_{\lambda}}\otimes\SL^{r_{(\lambda,\alpha)}+d_{(\lambda,\alpha)}/2}
\]
which by adjunction $(q_{(\lambda,\alpha)}^*,(q_{(\lambda,\alpha)})_*)$ provides us with the map
\begin{equation}
\label{equation:smoothpbq}
 \BD\underline{\BoQ}_{X^{\lambda}_{\alpha}/G^{\lambda}}^{\vir}\rightarrow(q_{(\lambda,\alpha)})_*\BD\underline{\BoQ}_{X^{\lambda\geq 0}_{\alpha}/P_{\lambda}}\otimes\SL^{r_{(\lambda,\alpha)}+d_{(\lambda,\alpha)}/2}.
\end{equation}
Moreover, the map $p_{(\lambda,\alpha)}$ is proper (Lemma~\ref{lemma:smoothnessqproperp}), and so by dualizing the adjunction map
\[
 \underline{\BoQ}_{X/G}\rightarrow (p_{(\lambda,\alpha)})_*\underline{\BoQ}_{X^{\lambda\geq 0}_{\alpha}/P_{\lambda}},
\]
and using $\BD (p_{(\lambda,\alpha)})_*=(p_{(\lambda,\alpha)})_!\BD$ and $(p_{(\lambda,\alpha)})_*=(p_{(\lambda,\alpha)})_!$, we obtain
\begin{equation}
\label{equation:properpfp}
 (p_{(\lambda,\alpha)})_*\BD\underline{\BoQ}_{X^{\lambda\geq 0}_{\alpha}/P_{\lambda}}\rightarrow\BD\underline{\BoQ}_{X/G}.
\end{equation}
We eventually obtain the map
\begin{equation}
\label{equation:sheafifiedinduction}
 \Ind_{(\lambda,\alpha)}\colon (\imath_{(\lambda,\alpha)})_*(\pi_{(\lambda,\alpha)})_*\BD\underline{\BoQ}_{X^{\lambda}_{\alpha}/G^{\lambda}}^{\vir}\rightarrow \pi_*\BD\underline{\BoQ}_{X/G}\otimes\SL^{r_{(\lambda,\alpha)}+d_{(\lambda,\alpha)}/2}
\end{equation}
by composing $(\imath_{(\lambda,\alpha)})_*(\pi_{(\lambda,\alpha)})_*$ applied to \eqref{equation:smoothpbq} with $\pi_*$ applied to \eqref{equation:properpfp} twisted by $\SL^{r_{(\lambda,\alpha)}+d_{(\lambda,\alpha)}/2}$. The map \eqref{equation:sheafifiedinduction} will be called \emph{sheafified induction morphism}. When $X$ is symmetric (Definition~\ref{definition:symmetricrepresentation}), Lemma~\ref{lemma:numberweights} implies that the induction morphism is a morphism of complexes
\begin{equation}
\label{equation:sheafindsymm}
 \Ind_{(\lambda,\alpha)}\colon (\imath_{(\lambda,\alpha)})_*(\pi_{(\lambda,\alpha)})_*\BD\underline{\BoQ}_{X^{\lambda}_{\alpha}/G^{\lambda}}^{\vir}\rightarrow \pi_*\BD\underline{\BoQ}_{X/G}^{\vir}\,.
\end{equation}

\begin{lemma}[{\cite[Lemma 2.5]{hennecart2024cohomological}}]
\label{lemma:preservescohdegrees}
 Let $X$ be a symmetric $G$-variety. The induction morphism 
 \[\Ind_{(\lambda,\alpha)}\colon \HO^{*+d_{(\lambda,\alpha)}}(X^{\lambda}_{\alpha}/G^{\lambda})\rightarrow \HO^{*+d}(X/G)\] preserves cohomological degrees.
\end{lemma}
\begin{proof}
 The proof of \cite[Lemma 2.5]{hennecart2024cohomological}, which is the case $X=V$ for a representation $V$ of $G$, applies also for the $G$-variety $X$: the induction is given by taking the derived global sections of the morphism of complexes of mixed Hodge modules \eqref{equation:sheafindsymm}.
\end{proof}

\subsection{Augmentation}
\label{subsection:augmentation}
Let $(\lambda,\alpha)\in \SQ_{X,G}$. We defined a $G^{\lambda}$-variety $X^{\lambda}_{\alpha}$ in \S\ref{subsection:inductiondiagram}. We can see $T$ as a maximal torus of $G^{\lambda}$. If $\mu\in X_*(T)$, we have $(X^{\lambda})^{\mu}=X^{\lambda}\cap X^{\mu}$, $(G^{\lambda})^{\mu}=G^{\lambda}\cap G^{\mu}$, $(X^{\lambda})^{\mu\geq 0}$ and $P_{\mu,\lambda}\coloneqq P_{\mu}\cap G^{\lambda}$. 

For $(\mu,\beta)\in\SQ_{X^{\lambda}_{\alpha},G^{\lambda}}$, the induction diagram is
\[\begin{tikzcd}
	& {(X^{\lambda}_{\alpha})^{\mu\geq 0}_{\beta}/P_{\mu,\lambda}} \\
	{(X^{\lambda}_{\alpha})^{\mu}_{\beta}/(G^{\lambda})^{\mu}} && {X^{\lambda}_\alpha{}/G^{\lambda}} \\
	{(X^{\lambda}_{\alpha})^{\mu}_{\beta}\cms (G^{\lambda})^{\mu}} && {X^{\lambda}_{\alpha}\cms G^{\lambda}}
	\arrow["{q_{(\lambda,\alpha),(\mu,\beta)}}"', from=1-2, to=2-1]
	\arrow["{p_{(\lambda,\alpha),(\mu,\beta)}}", from=1-2, to=2-3]
	\arrow["{\pi_{(\lambda,\alpha),(\mu,\beta)}}"', from=2-1, to=3-1]
	\arrow["\pi_{(\lambda,\alpha)}", from=2-3, to=3-3]
	\arrow["\imath_{(\lambda,\alpha),(\mu,\beta)}",from=3-1, to=3-3]
\end{tikzcd}\]

The definition of the induction map associated with this datum is
\[
 \Ind_{(\mu,\beta),(\lambda,\alpha)}\colon \HO^*((X^{\lambda}_{\alpha})^{\mu}_{\beta}/(G^{\lambda})^{\mu})\rightarrow \HO^*(X^{\lambda}_{\alpha}/G^{\lambda})\,.
\]
 To unclutter the notation, we may write $(X^{\lambda}_{\alpha})^{\mu}_{\beta}=X^{\lambda,\mu}_{\alpha,\beta}$.
\begin{lemma}[{\cite[Lemma 2.6]{hennecart2024cohomological}}]
\label{lemma:gcdcocharacters}
  For any $\lambda,\mu\in X_*(T)$, one can find $\nu\in X_*(T)$ such that $X^{\nu}=X^{\mu}\cap X^{\lambda}$, $G^{\nu}=G^{\mu}\cap G^{\lambda}$, $(X^{\lambda})^{\nu\geq0}=(X^{\lambda})^{\mu\geq0}$ and $P_{\mu,\lambda}=P_{\nu,\lambda}$.
\end{lemma}
\begin{proof}
 The proof of \cite[Lemma 2.6]{hennecart2024cohomological} is written in the case of $X=V$ a representation of reductive group, but generalises to any $X$ in a straightforward way.
\end{proof}
In particular, for $(\lambda,\alpha)\in \SQ_{X,G}$, $(\mu,\beta)\in\SQ_{X^{\lambda}_{\alpha},G^{\lambda}}$ and $\nu$ given by Lemma~\ref{lemma:gcdcocharacters} there exists $\gamma\in \pi_{0}(X^{\nu})$ such that $(X^{\lambda}_{\alpha})^{\mu}_{\beta}=X^{\nu}_{\gamma}$.

\subsection{Associativity}
\begin{proposition}[{\cite[Proposition 2.7]{hennecart2024cohomological}}]
\label{proposition:associativity}
 Let $(\nu,\gamma)\in\SQ_{X,G}, (\mu,\beta)\in\SQ_{X^{\nu}_{\gamma},G^{\nu}}, (\lambda,\alpha)\in\SQ_{X^{\nu,\mu}_{\gamma,\beta},G^{\nu,\mu}}$ be such that $(X^{\nu}_{\gamma})^{\lambda\geq 0}\supset (X^{\nu}_{\gamma})^{\mu\geq 0}$. Then,
 \[
  \Ind_{(\lambda,\alpha),(\nu,\gamma)}=\Ind_{(\mu,\beta),(\nu,\gamma)}\circ\Ind_{(\lambda,\alpha),(\mu,\beta)}.
 \]
\end{proposition}
\begin{proof}
 The proof of \cite[Proposition 2.7]{hennecart2024cohomological}, written in the case when $X=V$ is a representation of $G$, adapts to the case of general smooth affine varieties in a straighforward way.
\end{proof}

\section{Perverse filtration}
\label{section:perversefiltration}
\subsection{Perverse filtration}

Let $X$ be a smooth affine $G$-variety, where $G$ is a reductive group. We let $\CH=\CH_{X,G}\coloneqq \HO^*(X/G,\BoQ^{\vir})$. Then, the quotient map $\pi\colon X/G\rightarrow X\cms G$ gives a split perverse filtration $\FP^{\leq i}\CH_{X,G}\coloneqq \HO^*(X\cms G,{^p\tau^{\leq i}\pi_*\BoQ_{X/G}^{\vir}})$ on $\CH$ (by Proposition~\ref{proposition:semisimplicitypistar}). For $(\lambda,\alpha)\in\SQ_{X,G}$, we write $\CH_{(\lambda,\alpha)}\coloneqq\CH_{X^{\lambda}_{\alpha},G^{\lambda}}$.

\subsection{Induction and perverse filtration}
\begin{proposition}
\label{proposition:inductionperversefiltration}
Let $X$ be a smooth affine symmetric $G$-variety. Then, the induction morphisms $\Ind_{(\lambda,\alpha)}\colon \CH_{(\lambda,\alpha)}\rightarrow\CH$ (\S\ref{section:parabolicinduction}) respect the perverse filtration.
\end{proposition}
\begin{proof}
 The induction morphisms $\Ind_{(\lambda,\alpha)}$ are given by taking the derived global sections of the morphisms of complexes $\Ind_{(\lambda,\alpha)}\colon (\imath_{(\lambda,\alpha)})_*(\pi_{(\lambda,\alpha)})_*\underline{\BoQ}_{X^{\lambda}_{\alpha}/G^{\lambda}}^{\vir}\rightarrow\pi_*\underline{\BoQ}_{X/G}^{\vir}$ \eqref{equation:sheafifiedinduction}. The map $\imath_{(\lambda,\alpha)}$ is finite (Lemma~\ref{lemma:finitemap}) and is therefore perverse $t$-exact. By standard properties of $t$-structures, the composition ${\underline{\tau}}^{\leq k}(\imath_{(\lambda,\alpha)})_*(\pi_{(\lambda,\alpha)})_*\underline{\BoQ}_{X^{\lambda}_{\alpha}/G^{\lambda}}^{\vir}\rightarrow (\imath_{(\lambda,\alpha)})_*(\pi_{(\lambda,\alpha)})_*\underline{\BoQ}_{X^{\lambda}_{\alpha}/G^{\lambda}}^{\vir}\xrightarrow{\Ind_{(\lambda,\alpha)}}\pi_*\underline{\BoQ}_{X/G}^{\vir}$ factors through ${\underline{\tau}}^{\leq k}\pi_*\underline{\BoQ}_{X/G}^{\vir}\rightarrow \pi_*\underline{\BoQ}_{X/G}^{\vir}$. This finishes the proof.
\end{proof}

\begin{lemma}
 The $\overline{W}_{(\lambda,\alpha)}$-action on $\HO^*(X^{\lambda}_{\alpha}/G^{\lambda},\BoQ^{\vir})$ respects the perverse filtration.
\end{lemma}
\begin{proof}
 We let $p\colon N_G(T)\rightarrow W\cong N_G(T)/T$ be the quotient map and $N_G^{(\lambda,\alpha)}(T)\coloneqq p^{-1}(W_{(\lambda,\alpha)})$ (see after Definition~\ref{definition:Weylgrouprep}). Then, $N_G^{(\lambda,\alpha)}(T)$ acts by automorphisms on the map $\pi_{(\lambda,\alpha)}\colon X^{\lambda}_{\alpha}/G^{\lambda}\rightarrow X^{\lambda}_{\alpha}\cms G^{\lambda}$. Since $\underline{\BoQ}_{X^{\lambda}_{\alpha}/G^{\lambda}}^{\vir}$ is $N_G^{(\lambda,\alpha)}(T)$-equivariant, the pushforward $(\pi_{(\lambda,\alpha)})_*\underline{\BoQ}_{X^{\lambda}_{\alpha}/G^{\lambda}}^{\vir}$ also inherits a $N_G^{(\lambda,\alpha)}(T)$-equivariant structure. Therefore, the perverse filtration on $(\pi_{(\lambda,\alpha)})_*\underline{\BoQ}_{X^{\lambda}_{\alpha}/G^{\lambda}}^{\vir}$ is $N_{G}^{(\lambda,\alpha)}(T)$-equivariant (by the general theory of equivariant perverse sheaves \cite{bernstein2006equivariant}). By taking derived global sections, we obtain that the $N_G^{(\lambda,\alpha)}(T)$-action on $\HO^*(X^{\lambda}_{\alpha}/G^{\lambda})$ respects the perverse filtration. This action factors through $W_{(\lambda,\alpha)}\cong N_G^{(\lambda,\alpha)}(T)/T$. This finishes the proof of the lemma, as the action of $W^{\lambda}\subset W_{(\lambda,\alpha)}$ on $\pi_{(\lambda,\alpha)}$ is trivial, and so the action descends to the quotient $\overline{W}_{(\lambda,\alpha)}$.
\end{proof}

\begin{proposition}
\label{proposition:cohintgraded}
 Let $V$ be a representation of a reductive group $G$. The cohomological integrality map of Theorem~\ref{theorem:abscohintreminder} \cite[Theorem 1.1]{hennecart2024cohomological} is a filtered isomorphism. That is, the associated graded morphism
 \[
  \bigoplus_{\tilde{\lambda}\in\SP_V/W}(\gr(\CP_{\lambda})\otimes\HO^*(\pt/G_{\lambda}))^{\varepsilon_{V,\lambda}}\rightarrow \CH_{V,G}
 \]
with respect to the perverse filtration is an isomorphism.
\end{proposition}
\begin{proof}
 By Proposition~\ref{proposition:inductionperversefiltration}, the induction morphisms preserve the perverse filtration. Therefore, for any $i\in\BoZ$, by filtering the isomorphism of Theorem \ref{theorem:abscohintreminder}, we obtain an isomorphism
 \[
  \bigoplus_{\tilde{\lambda}\in\SP_V/W}(\FP^{\leq i}(\CP_{\lambda}\otimes\HO^*(\pt/G_{\lambda})))^{\varepsilon_{V,\lambda}}\rightarrow\FP^{\leq i}\CH_{V,G}\,.
 \]
Now, we have
\begin{multline}
 A_i\coloneqq(\FP^{\leq i}(\CP_{\lambda}\otimes\HO^*(\pt/G_{\lambda})))^{\varepsilon_{V,\lambda}}/(\FP^{\leq i-1}(\CP_{\lambda}\otimes\HO^*(\pt/G_{\lambda})))^{\varepsilon_{V,\lambda}}
 \\
 \cong (\bigoplus_{j\leq i}(\FP^{\leq j}\CP_{\lambda}/\FP^{\leq j-1}\CP_{\lambda})\otimes \FP^{\leq i-j}\HO^*(\pt/G_{\lambda})/\FP^{\leq i-j-1}\HO^*(\pt/G_{\lambda}))^{\varepsilon_{V,\lambda}}
\end{multline}
The perverse filtration on $\HO^*(\pt/G_{\lambda})$ is the filtration by cohomological degrees, and so, by assembling pieces, we obtain the isomorphism of the proposition.
\end{proof}

The advantage of taking the associated graded with respect to the perverse filtration is that we can work in the Abelian category of cohomologically graded mixed Hodges modules $\MMHM^{\BoZ}(X\cms G)$ (\S\ref{subsection:cohgradedmhm}) rather than in the triangulated category of complexes of mixed Hodge modules. In the former, one can take kernels and cokernels of morphisms.

Let $(\lambda,\alpha)\in\SQ_{X,G}$. We let $\overline{\Ind}_{(\lambda,\alpha)}\coloneqq\gr(\Ind_{(\lambda,\alpha)})=\bigoplus_{i\in\BoZ}\CH^i(\Ind_{(\lambda,\alpha)})\colon  (\imath_{(\lambda,\alpha)})_*\underline{\CH}_{(\lambda,\alpha)}\rightarrow\underline{\CH}$ be the sheafified associated graded mutiplication and we use the same notation after taking derived global sections.

\section{Local neighbourhood theorem for smooth stacks}
\label{section:localneighbourhoodsmoothstacks}
\subsection{Neighbourhood theorem via \'etale slices}
\begin{theorem}
\label{theorem:localneighbourhood}
 Let $G$ be a reductive group and $X$ a smooth affine $G$-variety. Let $x\in G$ such that the orbit $G\cdot x$ is closed, and so its stabilizer $G_x\subset G$ is a reductive group. Let $\lambda\colon\BoG_{\rmm}\rightarrow G_x$ be a one-parameter subgroup. We let $N_x\subset X$ be a normal slice through $x$ to the $G$-orbit of $x$. It is a $G_x$-variety. There exists a $G_x$-invariant Zariski open neighbourhood $U\subset \Tan_xN_x$ of $0$ such that we have commutative diagrams (after possibly shrinking $N_x$ to a Zariski open neighbourhood of $x$ in $X$)
\[\begin{tikzcd}
	& {X^{\lambda}_{\alpha}/ G^{\lambda}} & {X^{\lambda\geq 0}_{\alpha}/G^{\lambda\geq 0}} & {X/G} \\
	{N_x^{\lambda}/G_x^{\lambda}} & {N_x^{\lambda\geq 0}/G_x^{\lambda\geq0}} & {N_x/G_x} & {} \\
	{} & {X^{\lambda}_{\alpha}\cms G^{\lambda}} & {} & {X\cms G} \\
	{N_x^{\lambda}\cms G_x^{\lambda}} && {N_x\cms G_x}
	\arrow[no head, from=1-2, to=2-2]
	\arrow["{q_{(\lambda,\alpha)}}"', from=1-3, to=1-2]
	\arrow["{p_{(\lambda,\alpha)}}", from=1-3, to=1-4]
	\arrow[from=1-4, to=3-4]
	\arrow[from=2-1, to=1-2]
	\arrow[from=2-1, to=4-1]
	\arrow[from=2-2, to=1-3]
	\arrow["{q_{x,\lambda}}"', from=2-2, to=2-1]
	\arrow["{p_{x,\lambda}}", from=2-2, to=2-3]
	\arrow[from=2-2, to=3-2]
	\arrow[from=2-3, to=1-4]
	\arrow[from=2-3, to=4-3]
	\arrow["{\imath_{(\lambda,\alpha)}}"{pos=0.6}, no head, from=3-2, to=3-3]
	\arrow[from=3-3, to=3-4]
	\arrow[from=4-1, to=3-2]
	\arrow["{\imath_{x,\lambda}}", from=4-1, to=4-3]
	\arrow[from=4-3, to=3-4]
\end{tikzcd}\]
\[\begin{tikzcd}
	& {U^{\lambda}\cms G^{\lambda}_x} & {U^{\lambda\geq 0}/G^{\lambda\geq 0}_x} & {U/G_x} \\
	{N_x^{\lambda}/G_x^{\lambda}} & {N_x^{\lambda\geq 0}/G_x^{\lambda\geq0}} & {N_x/G_x} & {} \\
	{} & {U^{\lambda}\cms G_x^{\lambda}} & {} & {U\cms G_x} \\
	{N_x^{\lambda}\cms G_x^{\lambda}} && {N_x\cms G_x}
	\arrow[no head, from=1-2, to=2-2]
	\arrow["{q_{\lambda}}"', from=1-3, to=1-2]
	\arrow["{p_{\lambda}}", from=1-3, to=1-4]
	\arrow[from=1-4, to=3-4]
	\arrow[from=2-1, to=1-2]
	\arrow[from=2-1, to=4-1]
	\arrow[from=2-2, to=1-3]
	\arrow["{q_{x,\lambda}}"', from=2-2, to=2-1]
	\arrow["{p_{x,\lambda}}", from=2-2, to=2-3]
	\arrow[from=2-2, to=3-2]
	\arrow[from=2-3, to=1-4]
	\arrow[from=2-3, to=4-3]
	\arrow["{\imath_{\lambda}}"{pos=0.6}, no head, from=3-2, to=3-3]
	\arrow[from=3-3, to=3-4]
	\arrow[from=4-1, to=3-2]
	\arrow["{\imath_{x,\lambda}}", from=4-1, to=4-3]
	\arrow[from=4-3, to=3-4]
\end{tikzcd}\]
 with \'etale diagonal arrows and Cartesian left and right faces of both parallelepipeds.
\end{theorem}
\begin{proof}
 This is a version of Luna \'etale slice theorem applied to $X/G$. More precisely, we let $\Tan_x(X)\cong \Tan_x(G\cdot x)\oplus \Tan_xN_x$ be a $G_x$-equivariant decomposition where $N_x\subset X$ be an \'etale slice at $x$. By Luna \'etale slice theorem, we have a diagram with Cartesian squares and \'etale horizontal maps
\[\begin{tikzcd}
	{U/G_x} & {N_x/G_x} & {X/G} \\
	{U\cms G_x} & {N_x\cms G_x} & {X\cms G}
	\arrow[from=1-1, to=2-1]
	\arrow[from=1-2, to=1-1]
	\arrow[from=1-2, to=1-3]
	\arrow[from=1-2, to=2-2]
	\arrow[from=1-3, to=2-3]
	\arrow[from=2-2, to=2-1]
	\arrow[from=2-2, to=2-3]
\end{tikzcd}\]
 This is the combination of the right faces of the parallelepipeds in the diagrams of the theorem.
 
 We let $\Fg\coloneqq \Fg_x\oplus\Fg_x^{\perp}$ be a $G_x$-stable decomposition of $\Fg$, where $\Fg_x^{\perp}=\Fg/\Fg_x$. We have $\Fg^{\lambda}\cong \Fg_x^{\lambda}\oplus (\Fg_x^{\perp})^{\lambda}$. As $G_x$-representations, we have $\Tan_x(G\cdot x)\cong\Fg_x^{\perp}$. As $G_x^{\lambda}$-representations, we have $\Tan_x(G^{\lambda}\cdot x)\cong (\Fg_x^{\perp})^{\lambda}$ (Lemma~\ref{lemma:fixedpointstorus}).
 
 We have $(\Tan_xX)^{\lambda}\cong (\Tan_x (G\cdot x))^{\lambda}\oplus (\Tan_xN_x)^{\lambda}\cong \Tan_x(G^{\lambda}\cdot x)\oplus (\Tan_xN_x)^{\lambda}$, proving that $(\Tan_xN_x)^{\lambda}$ is the normal bundle to $G^{\lambda}\cdot x$ at $x$. We prove similarly that $N_x^{\lambda}$ is a normal slice to $G^{\lambda}\cdot x$ at $x$. By the \'etale slice theorem, we obtain the left faces of the two parallelepipeds in Theorem \ref{theorem:localneighbourhood}.
 
 Next, we consider $X^{\lambda\geq 0}\cong (\Tan_x(G\cdot x))^{\lambda\geq 0}\oplus (\Tan_xN_x)^{\lambda\geq 0}$ and $\Fg^{\lambda\geq 0}\cong \Fg_x^{\lambda\geq 0}\oplus(\Fg_x^{\perp})^{\lambda\geq 0}$. Since $(\Tan_x(G\cdot x))^{\lambda\geq 0}\cong \Tan(G^{\lambda\geq 0}\cdot x)$, we obtain the \'etale morphisms $N_x^{\lambda\geq 0}/G_x^{\lambda\geq 0}\rightarrow U^{\lambda\geq 0}/G_x^{\lambda\geq 0}$ and $N_x^{\lambda\geq 0}/G_x^{\lambda\geq 0}\rightarrow X^{\lambda\geq 0}/G^{\lambda\geq 0}$.

 The commutativity of all squares in the diagrams of Theorem~\ref{theorem:localneighbourhood} is straighforward to check.
\end{proof}

\subsection{Fiber induction}
Let $(\lambda,\alpha)\in \SQ_{X,G}$ and $\overline{x}\in X^{\lambda}_{\alpha}\cms G^{\lambda}$. We let $\imath_{\overline{x}}\colon\{\overline{x}\}\rightarrow X^{\lambda}_{\alpha}\cms G^{\lambda}$ be the inclusion. The induction morphism $\Ind_{(\lambda,\alpha)}\colon (\imath_{(\lambda,\alpha)})_*(\pi_{(\lambda,\alpha)})_*\underline{\BoQ}_{X^{\lambda}_{\alpha}/G^{\lambda}}^{\vir}\rightarrow \pi_*\underline{\BoQ}_{X/G}^{\vir}$ induces a morphism
\[
 (\imath_{(\lambda,\alpha)}\circ\imath_{\overline{x}})^*\Ind_{(\lambda,\alpha)}\colon \HO^*(\pi_{(\lambda,\alpha)}^{-1}(\imath_{(\lambda,\alpha)}^{-1}(\imath_{(\lambda,\alpha)}(\overline{x}))),\BoQ^{\vir})\rightarrow\HO^*(\pi^{-1}(\imath_{(\lambda,\alpha)}(\overline{x})),\BoQ^{\vir})
\]
and since $\pi_{(\lambda,\alpha)}^{-1}(\overline{x})$ is a connected component of $\pi_{(\lambda,\alpha)}^{-1}(\imath_{(\lambda,\alpha)}^{-1}(\imath_{(\lambda,\alpha)}(\overline{x})))$ (by finiteness of $\imath_{(\lambda,\alpha)}$, Lemma~\ref{lemma:finitemap}), a morphism
\begin{equation}
\label{equation:inductionfiberx}
 \Ind_{(\lambda,\alpha),\overline{x}}\colon \HO^*(\pi_{(\lambda,\alpha)}^{-1}(\overline{x}),\BoQ^{\vir})\rightarrow\HO^*(\pi^{-1}(\imath_{(\lambda,\alpha)}(\overline{x})),\BoQ^{\vir})\,.
\end{equation}

\subsection{Compatibility of fiber inductions}
Let $x\in X$ be such that $\overline{x}=G\cdot x$ is closed and $T_x\coloneqq T\cap G_x$ is a maximal torus of $G_x$. We let $\lambda\in X_*(T_x)$ and $\alpha\in\pi_0(X^{\lambda})$ be the connected component containing $x$. There is an induction morphism
\[
\Ind_{\lambda}\colon (\imath_{\lambda})_*(\pi_{\lambda})_*\underline{\BoQ}_{(\Tan_xN_x)^{\lambda}/G^{\lambda}}^{\vir}\rightarrow\pi_*\underline{\BoQ}_{\Tan_xN_x/G_x}\,.
\]
By restriction over $\overline{0}\in \Tan_xN_x\cms G_x$ and $\BoC^*$-equivariance, we obtain a morphism
\begin{equation}
\label{equation:inductionzero}
 \Ind_{\lambda,\overline{0}}\colon \HO^*(\pi_{\lambda}^{-1}(\overline{0}),\BoQ^{\vir})\rightarrow\HO^*(\pi^{-1}(\overline{0}),\BoQ^{\vir})\,.
\end{equation}

\begin{proposition}
\label{proposition:compatibilityinduction}
 There is a commutative diagram in which vertical maps are isomorphisms:
\[\begin{tikzcd}
	{\HO^*(\pi_{\lambda}^{-1}(\overline{0}),\BoQ^{\vir})} & {\HO^*(\pi^{-1}(\overline{0}),\BoQ^{\vir})} \\
	{\HO^*(\pi_{(\lambda,\alpha)}^{-1}(\overline{x}),\BoQ^{\vir})} & {\HO^*(\pi^{-1}(\imath_{(\lambda,\alpha)}(\overline{x})),\BoQ^{\vir})}
	\arrow["{\Ind_{\lambda,\overline{0}}}", from=1-1, to=1-2]
	\arrow[from=1-1, to=2-1]
	\arrow[from=1-2, to=2-2]
	\arrow["{\Ind_{(\lambda,\alpha),\overline{x}}}", from=2-1, to=2-2]
\end{tikzcd}\]
\end{proposition}
\begin{proof}
 This follows from the local neighbourhood theorem (Theorem~\ref{theorem:localneighbourhood}), by base-change.
\end{proof}

\section{Sheafified cohomological integrality for affine varieties with reductive group action}

\subsection{The character}
\label{subsection:character}
For $(\lambda,\alpha)\in \SQ_{X,G}$, we let $W_{(\lambda,\alpha)}=\{w\in W\mid w\cdot \overline{(\lambda,\alpha)}=\overline{(\lambda,\alpha)}\}\subset W$ (\S\ref{subsection:Weylgroups}). This is a subgroup of $W$. We let $N_{X}(X^{\lambda}_{\alpha})\cong \Tan X_{|X^{\lambda}_{\alpha}}/\Tan X^{\lambda}_{\alpha}$ be the normal bundle of $X^{\lambda}_{\alpha}\subset X$. We let $T'\coloneqq (\ker(G^{\lambda}\rightarrow\Aut(X^{\lambda}_{\alpha})\cap Z(G^{\lambda}))^{\circ}$. We have $w\cdot T'=T'$ for any $w\in W_{(\lambda,\alpha)}$ (Lemma~\ref{lemma:relWeylpreferredtorus}). We let $N_X(X^{\lambda}_{\alpha})= \bigoplus_{\beta\in X_*(T')}N_X(X^{\lambda}_{\alpha})_{\beta}$ be the weight decomposition with respect to the $T'$-action. Then, by the Bia\l inicky-Birula decomposition, $X_{\alpha}^{\lambda> 0}\cong N_X(X^{\lambda}_{\alpha})^{\lambda>0}\coloneqq\bigoplus_{\substack{\beta\in X_*(T')\\\langle\lambda,\beta\rangle>0}}N_X(X^{\lambda}_{\alpha})_{\beta}$. Now, by symmetry of $X$, the weights of $N_X(X^{\lambda}_{\alpha})$ are symmetric, that is for any $\beta\in X_*(T')$, $\rank(N_X(X^{\lambda}_{\alpha})_{\beta})=\rank(N_X(X^{\lambda}_{\alpha})_{-\beta})$. Similarly, we let $\Fg^{\lambda\neq 0}=\bigoplus_{\substack{\beta\in X_*(T')\\\langle\lambda,\beta\rangle\neq 0}}\Fg_{\beta}$ be the decomposition of $\Fg^{\lambda\neq 0}$ in weight spaces under the $T'$-action.

We define $k_{(\lambda,\alpha)}\coloneqq \frac{\prod_{\substack{\beta\in X_*(T')\\\langle\lambda,\beta\rangle>0}}\beta^{\rank(N_X(X^{\lambda}_{\alpha})_{\beta})}}{\prod_{\substack{\beta\in X_*(T')\\\langle\lambda,\beta\rangle>0}}\beta^{\dim\Fg_{\beta}}}$.

\begin{lemma}
\label{lemma:character}
 There exists a character $\varepsilon_{X,(\lambda,\alpha)}\colon W_{(\lambda,\alpha)}\rightarrow\{\pm1\}$ such that for any $w\in W_{(\lambda,\alpha)}$, $w\cdot k_{(\lambda,\alpha)}=\varepsilon_{X,(\lambda,\alpha)}(w)k_{(\lambda,\alpha)}$.
\end{lemma}
\begin{proof}
 By symmetry of $X$, we have
 \[
  (-1)^{r_{(\lambda,\alpha)}}k_{(\lambda,\alpha)}^2= \frac{\prod_{\beta\in X_*(T')}\beta^{\rank(N_X(X^{\lambda}_{\alpha})_{\beta})}}{\prod_{\substack{\beta\in X_*(T')\\ \langle\lambda,\beta\rangle\neq 0}}\beta^{\dim\Fg_{\beta}}}
 \]
Moreover, since $W_{(\lambda,\alpha)}$ acts on the weights of $N_X(X^{\lambda}_{\alpha})$ and $\Fg$, we have $(-1)^{r_{(\lambda,\alpha)}}(w\cdot k_{(\lambda,\alpha)})^2=(-1)^{r_{(\lambda,\alpha)}}k_{(\lambda,\alpha)}^2$. It follows immediately that $w\cdot k_{(\lambda,\alpha)}=\varepsilon_{X,(\lambda,\alpha)}(w)k_{(\lambda,\alpha)}$ for some uniquely determined sign $\varepsilon_{X,(\lambda,\alpha)}(w)\in \{\pm1\}$. The unicity implies that the map $\varepsilon_{X,(\lambda,\alpha)}\colon W_{(\lambda,\alpha)}\rightarrow\{\pm1\}$ is a character.
\end{proof}

\begin{lemma}
 The character $\varepsilon_{X,(\lambda,\alpha)}$ is trivial on $W^{\lambda}$, and so induces a character $\varepsilon_{X,(\lambda,\alpha)}\colon \overline{W}_{(\lambda,\alpha)}\rightarrow\{\pm1\}$.
\end{lemma}
\begin{proof}
 This follows from the fact that $w\in W^{\lambda}$ preserves $\lambda$ and the pairing between characters and cocharacters is $W$-invariant: $\langle\lambda,\beta\rangle=\langle w\lambda,\beta\rangle=\langle \lambda,w^{-1}\beta\rangle$ for any $\lambda\in X_*(T)$, $\beta\in X^*(T)$ and $w\in W^{\lambda}$. Therefore, $\langle\lambda,w\cdot\beta\rangle>0$ iff $\langle\lambda,\beta\rangle>0$. This implies that $k_{(\lambda,\alpha)}$ is $w$-invariant.
\end{proof}

\subsection{Character and \'etale slices}
In this section, we show the compatibility of the character defined in \S\ref{subsection:character} with \'etale slices.

Let $x\in X$ be such that $G\cdot x$ is a closed $G$-orbit and $T_x=T\cap G_x$ is a maximal torus of $G_x$. We let $N_x$ be a normal slice through $x$ and $\lambda\in X_*(T_x)$ a cocharacter. We let $\alpha\in\pi_0(X^{\lambda})$ be the connected component of $X^{\lambda}$ containing $x$.

\begin{lemma}
 We have a commutative diagram
\[\begin{tikzcd}
	{\overline{W}_{x,\lambda}} & {} & {\overline{W}_{(\lambda,\alpha)}} \\
	& {\{\pm1\}}
	\arrow[from=1-1, to=1-3]
	\arrow["{\varepsilon_{x,\lambda}}"', from=1-1, to=2-2]
	\arrow["{\varepsilon_{X,(\lambda,\alpha)}}", from=1-3, to=2-2]
\end{tikzcd}\]
where the horizontal map is the inclusion of Lemma~\ref{lemma:injectmorphgroups} with $T'=G_{\lambda}^{\circ}$.
\end{lemma}
\begin{proof}
 We have decompositions $\Tan_xX\cong(\Tan_x(G\cdot x))\oplus \Tan_xN_x$ and $\Fg\cong\Fg_x\oplus\Fg_x^{\perp}$. Moreover, $\Tan_x(G\cdot x)\cong \Fg_x^{\perp}$ as $G_x$-representations. Therefore,
 \[
  (\Tan_xX)^{\lambda>0}\cong (\Fg_x^{\perp})^{\lambda>0}\oplus(\Tan_xN_x)^{\lambda>0}, \quad \Fg^{\lambda>0}\cong \Fg_x^{\lambda>0}\oplus (\Fg_{x}^{\perp})^{\lambda>0}\,.
 \]
 Consequently,
 \[
  k_{(\lambda,\alpha)}=\frac{\prod_{\substack{\beta\in X_*(T')\\\langle\lambda,\beta\rangle>0}}\beta^{\dim (T_xX)_{\beta}}}{\prod_{\substack{\beta\in X_*(T')\\\langle\lambda,\beta\rangle>0}}\beta^{\dim \Fg_{\beta}}}=\frac{\prod_{\substack{\beta\in X_*(T')\\\langle\lambda,\beta\rangle>0}}\beta^{\rank (T_xN_x)_{\beta}}}{\prod_{\substack{\beta\in X_*(T')\\\langle\lambda,\beta\rangle>0}}\beta^{\rank (\Fg_x)_{\beta}}}=k_{x,\lambda}\,.
 \]
If $w\in \overline{W}_{x,\lambda}$, $w$ acts independently on the weights of $\Fg_x$ and $\Fg_x^{\perp}$ on the one side, and $\Tan_xN_x$ and $\Fg_x^{\perp}$ on the other side. Therefore, the action of $w$ changes $k_{(\lambda,\alpha)}$ and $k_{x,\lambda}$ by the same sign. This concludes.

\end{proof}

\subsection{A group-theoretic lemma}

\begin{lemma}
\label{lemma:grouptheoretic}
 Let $W$ be a finite-group. We assume that $W$ acts transitively on a finite set $I$ and we let $V=\bigoplus_{i\in I}V_i$ be a finite dimensional representation of $W$ such that for any $v\in V_i$, $w\cdot v\in V_{w\cdot i}$. We let $\varepsilon\colon W\rightarrow\BoC^*$ be a character of $W$ and $i\in I$. If $W_i\subset W$ is the stabilizer of $i$ and $\varepsilon_i\colon W_i\rightarrow\BoC^*$ the character induced by $\varepsilon$, then $V^{\varepsilon}\cong V_i^{\varepsilon_i}$.
\end{lemma}
\begin{proof}
 The natural projection $V\rightarrow V_i$ induces a map $V^{\varepsilon}\rightarrow V_i^{\varepsilon_i}$. We now construct an inverse. Let $v_i\in V_i^{\varepsilon_i}$ and $j\in I$. We let $w\in W$ such that $w\cdot i=j$ (by transitivity). Then, $\varepsilon(w)^{-1}w\cdot v_i$ does not depend on $w$. Indeed, if $w'\cdot i=j$, then $w^{-1}w'\in W_i$ and so $w^{-1}w'\cdot v_i=\varepsilon(w)^{-1}\varepsilon(w')v$. Therefore, $\varepsilon(w')^{-1}w'\cdot v_i=\varepsilon(w)^{-1}w\cdot v_i$. We set $v_j\coloneqq \varepsilon(w)^{-1}w\cdot v_i$. Then, the inverse is $v_i\mapsto (v_j)_{j\in I}$.
\end{proof}

\subsection{Fixed loci}
\begin{lemma}
\label{lemma:fixedloci}
 Let $V$ be a representation of a reductive group and $\lambda\colon \BoG_{\rmm}\rightarrow G$ a one-parameter subgroup. We let $G\cdot x\in V\cms G$ and $G_x\subset G$ be the stabiliser of $x$. Then, $\overline{x}=G\cdot x$ belongs to the image of $\imath_{\lambda}\colon V^{\lambda}\cms G^{\lambda}\rightarrow V\cms G$ if and only if a $G$-conjugate of $\lambda$ factors through $G_x$. 
\end{lemma}
\begin{proof}
 The point $\overline{x}$ belongs to the image of $\imath_{\lambda}$ if and only if $G\cdot x\cap V^{\lambda}\neq\emptyset$. Therefore, $\lambda$ fixed $g\cdot x$ for some $g\in G$ and so $g^{-1}\lambda g$ fixed $x$.
\end{proof}

\subsection{The BPS (cohomologically graded) mixed Hodge module}
\label{subsection:bpsmhm}
In this section, we define cohomologically graded mixed Hodge modules over $X^{\lambda}_{\alpha}\cms G^{\lambda}$ for $(\lambda,\alpha)\in \SQ_{X,G}$ that satisfy a sheafified version of the cohomological integrality isomorphism \cite{hennecart2024cohomological}, see Theorem~\ref{theorem:sheafcohint}. We mimick at the sheaf level the definition of the vector space $\CP_{\lambda}$ in \cite{hennecart2024cohomological}, recalled in \S\ref{section:reminderabsolute}. Namely, for $(\lambda,\alpha)\in \SQ_{X,G}$, we see $X^{\lambda}_{\alpha}$ as a $\overline{G^{\lambda}}\coloneqq G^{\lambda}/G_{(\lambda,\alpha)}$-variety where $G_{(\lambda,\alpha)}=\ker(G^{\lambda}\rightarrow\Aut(X^{\lambda}_{\alpha}))\cap Z(G^{\lambda})$. For $(\mu,\beta)\in \SQ_{X^{\lambda}_{\alpha},G^{\lambda}}$, the induction formalism \S\ref{section:parabolicinduction} provides us with a map
\[
 \Ind_{(\mu,\beta),(\lambda,\alpha)}\colon (\imath_{(\lambda,\alpha),(\mu,\beta)})_*(\pi_{(\lambda,\alpha),(\mu,\beta)})_*\underline{\BoQ}_{X^{\lambda,\mu}_{\alpha,\beta}/\overline{G^{\lambda}}^{\mu}}^{\vir}\rightarrow (\pi_{(\lambda,\alpha)})_*\underline{\BoQ}_{X^{\lambda}_{\alpha}/\overline{G^{\lambda}}}^{\vir}\,.
\]
As in the end of \S\ref{section:perversefiltration}, we let $\overline{\Ind}_{(\mu,\beta),(\lambda,\alpha)}$ be the sheafified associated graded induction with respect to the perverse filtration.

The image $\underline{\CJ}_{(\lambda,\alpha)}$ of the induction map
\[
 \bigoplus_{\substack{\overline{(\mu,\beta)}\in\SP_{X^{\lambda}_{\alpha},\overline{G^{\lambda}}}\\\ (X^{\lambda}_{\alpha})^{\mu}\neq X^{\lambda}_{\alpha}}}(\imath_{(\lambda,\alpha),(\mu,\beta)})_*(\pi_{(\lambda,\alpha),(\mu,\beta)})_*\underline{\BoQ}_{X^{\lambda,\mu}_{\alpha,\beta}/\overline{G^{\lambda}}^{\mu}}^{\vir}\xrightarrow{\sum_{\overline{(\mu,\beta)}}\overline{\Ind}_{(\mu,\beta),(\lambda,\alpha)}}(\pi_{(\lambda,\alpha)})_*\underline{\BoQ}_{X^{\lambda}_{\alpha}/\overline{G^{\lambda}}}^{\vir}
\]
is a $\overline{W}_{(\lambda,\alpha)}$-equivariant cohomologically graded sub-mixed Hodge module of $(\pi_{(\lambda,\alpha)})_*\underline{\BoQ}_{X^{\lambda}_{\alpha}/\overline{G^{\lambda}}}^{\vir}$ (by semisimplicity). We let $\underline{\BPS}_{(\lambda,\alpha)}$ be a $\overline{W}_{(\lambda,\alpha)}$-equivariant direct sum complement of $\underline{\CJ}_{(\lambda,\alpha)}$.

We have compatibility of the BPS sheaf with \'etale slices. More precisely, let $x\in X$ be such that the orbit $\overline{x}\coloneqq G\cdot x$ is closed. Up to $G$-conjugation, we can assume that $T_x\coloneqq G\cap T$ is a maximal torus of $G_x$. We let $N_x$ be a normal slice through $G_x$. We let $\imath_{\overline{x}}\colon \{G\cdot x\}\rightarrow X\cms G$ and $\imath_{\overline{0}}\colon \{\overline{0}\}\rightarrow \Tan_xN_x\cms G_x$ be the natural inclusions.
\begin{proposition}
\label{proposition:compatibilityBPSetaleslices}
 Let $\lambda\in X_*(T_x)$ and $\alpha\in\pi_0(X^{\lambda})$ be the connected component containing $x$. We have a commutative diagram in which vertical arrows are isomorphisms:
\[\begin{tikzcd}
	{\imath_{\overline{x}}^*\underline{\BPS}_{(\lambda,\alpha)}} & {\imath_{\overline{x}}^*(\pi_{(\lambda,\alpha)})_*\underline{\BoQ}_{X^{\lambda}_{\alpha}/G^{\lambda}}^{\vir}} \\
	{\imath_{\overline{0}}^*\underline{\BPS}_{\Tan_xN_x,\lambda}} & {\imath_{\overline{0}}^*(\pi_{\lambda})_*\underline{\BoQ}_{\Tan_xN_x/G_x}^{\vir}}
	\arrow[from=1-1, to=1-2]
	\arrow[from=1-1, to=2-1]
	\arrow[from=1-2, to=2-2]
	\arrow[from=2-1, to=2-2]
\end{tikzcd}\]
\end{proposition}
\begin{proof}
 This follows from Theorem~\ref{theorem:localneighbourhood} and Lemma~\ref{lemma:fixedloci} by the definition of the cohomologically graded mixed Hodge modules $\underline{\BPS}_{(\lambda,\alpha)}$.
\end{proof}

With this definition, we are now able to prove Proposition~\ref{proposition:boundedness} regarding the boundedness of the cohomologically graded mixed Hodge modules $\underline{\BPS}_{(\lambda,\alpha)}$.

\begin{proof}[Proof of Proposition~\ref{proposition:boundedness}]
 Let $(\lambda,\alpha)\in\SQ_{X,G}$ We assume by contradiction that $\CH^i(\underline{\BPS}_{(\lambda,\alpha)})$ is nonzero for arbitrarily large or small $i\in\BoZ$. For any \'etale slice $N_x$ at some point $x\in X$ with stabilizer $G_x$, $\HO^*(\underline{\BPS}_{\Tan_xN_x,\lambda})=0$ for $i<\dim G_x-\dim N_x=\dim G-\dim V$ and for $i> \dim G+\dim V\geq \dim G_x+\dim N_x$ by \cite[Proposition~1.4]{hennecart2024cohomological}. We reach a contradiction by taking $i\in \BoZ$ with either $i\ll 0$ or $i\gg 0$ and $\overline{x}\in X^{\lambda}_{\alpha}\cms G^{\lambda}$ such that $\imath_{\overline{x}}^*\CH^i(\underline{\BPS}_{X,(\lambda,\alpha)})\neq 0$. Indeed, in this case, by Proposition~\ref{proposition:compatibilityBPSetaleslices}, we have $\HO^{i'}(\underline{\BPS}_{\Tan_xN_x,\lambda})\neq 0$ for $i'>\dim G+\dim V$ or $i'<\dim G-\dim V$.
\end{proof}

\subsection{Sheafified cohomological integrality for symmetric affine smooth $G$-varieties}
For any $(\lambda,\alpha)\in \SP_{X,G}$, we have by definition a morphism
\[
 \underline{\BPS}_{(\lambda,\alpha)}\rightarrow \underline{\CH}_{X^{\lambda}_{\alpha},G^{\lambda}}\,.
\]
Moreover, by \cite[\S3]{hennecart2024cohomological}, we have an action of $\HO^*(\pt/G_{(\lambda,\alpha)})$ on $\underline{\CH}_{X^{\lambda}_{\alpha},G^{\lambda}}$ which we use to define a morphism
\[
  \underline{\BPS}_{(\lambda,\alpha)}\otimes\HO^*(\pt/G_{(\lambda,\alpha)})\rightarrow \underline{\CH}_{X^{\lambda}_{\alpha},G^{\lambda}}\,.
\]
By composing these morphisms with the sheafified induction morphisms, we obtain the map
\[
  (\imath_{(\lambda,\alpha)})_*(\underline{\BPS}_{(\lambda,\alpha)}\otimes\HO^*(\pt/G_{(\lambda,\alpha)}))\rightarrow \underline{\CH}_{X,G}\,.
\]

\begin{theorem}
\label{theorem:sheafcohint}
 The induction induces an isomorphism of complexes of mixed Hodge modules
 \[
  \Ind_{X,G}\colon\bigoplus_{\tilde{(\lambda,\alpha)}\in\SP_{X,G}/W}((\imath_{(\lambda,\alpha)})_*\underline{\BPS}_{(\lambda,\alpha)}\otimes\HO^*(\pt/G_{(\lambda,\alpha)}))^{\varepsilon_{X,(\lambda,\alpha)}}\rightarrow\underline{\CH}_{X,G}\,.
 \]
\end{theorem}
\begin{proof}
 It suffices to prove the statement after taking the associated graded with respect to the perverse filtrations on both sides. We denote by $\overline{\Ind}_{X,G}$ the morphism we obtain. If the map $\overline{\Ind}_{X,G}$ is not an isomorphism, we let $\CF$ be a simple direct summand of either the kernel of the cokernel of $\overline{\Ind}_{X,G}$ (in the Abelian category $\MMHM^{\BoZ}(X\cms G)$). We let $\overline{x}\in\supp(\CF)\subset X\cms G$ be such that $\imath_{\overline{x}}^*\CF\neq 0$. We let $x\in X$ be a lift of $\overline{x}$ such that $T\cap G_x$ is a maximal torus of $G_x$. Then,
 \[
  \imath_{\overline{x}}^*\Ind_{X,G}\colon \bigoplus_{\tilde{(\lambda,\alpha)}\in\SP_{X,G}/W}(\imath_{\overline{x}}^*((\imath_{(\lambda,\alpha)})_*\underline{\BPS}_{X,(\lambda,\alpha)})\otimes\HO^*(\pt/G_{(\lambda,\alpha)}))^{\varepsilon_{X,(\lambda,\alpha)}}\rightarrow\imath_{\overline{x}}^*\underline{\CH}_{X,G}
 \]
 is not an isomorphism.
 
 By Lemma~\ref{lemma:bijection}, Proposition~\ref{proposition:compatibilityinduction}, Proposition~\ref{proposition:compatibilityBPSetaleslices} and Lemma~\ref{lemma:grouptheoretic}, $\imath_{\overline{x}}^*\Ind_{X,G}$ is identified with the map
 \begin{equation}
 \label{equation:cohintzero}
 \imath_{\overline{0}}^*\Ind_{\Tan_xN_x,G_x}\colon \bigoplus_{\tilde{\lambda}\in\SP_{\Tan_xN_x}/W_x}(\imath_{\overline{0}}^*((\imath_{\lambda})_*\underline{\BPS}_{\Tan_xN_{x},\lambda})\otimes\HO^*(\pt/G_{\lambda}))^{\varepsilon_{\Tan_xN_{x},\lambda}}\rightarrow\imath_{\overline{0}}^*\underline{\CH}_{\Tan_xN_x,G_x}
 \end{equation}
 which is therefore not an isomorphism. But, by $\BoC^*$-equivariance, \eqref{equation:cohintzero} is the cohomological integrality isomorphism \cite[Theorem 1.1]{hennecart2024cohomological}, see also Theorem~\ref{theorem:abscohintreminder}, for the representation $\Tan_xN_x$ of $G_x$. We reached a contradiction.
\end{proof}

\begin{proposition}
\label{proposition:characterisationBPS}
 The BPS sheaves $\underline{\BPS}_{X,(\lambda,\alpha)}$ are uniquely determined (up to isomorphism). That is, if $\underline{\CF}_{(\lambda,\alpha)}$ are $\overline{W}_{\lambda}$-equivariant cohomologically graded mixed Hodge modules on $X^{\lambda}_{\alpha}\cms G^{\lambda}$ for $(\lambda,\alpha)\in\SQ_{X,G}$ such that there exists a graded isomorphism
 \[
  \bigoplus_{\tilde{(\lambda,\alpha)}\in\SP_{X,G}/W}((\imath_{(\lambda,\alpha)})_*\underline{\CF}_{(\lambda,\alpha)}\otimes\HO^*(\pt/G_{\lambda}))^{\varepsilon_{X,(\lambda,\alpha)}}\rightarrow\underline{\CH}_{X,G}\,,
 \]
 and similarly for all pairs $(X^{\lambda}_{\alpha},G^{\lambda})$ then $\underline{\CF}_{(\lambda,\alpha)}\cong\underline{\BPS}_{X,(\lambda,\alpha)}$ for any $(\lambda,\alpha)\in\SQ_{X,G}$.

\end{proposition}
\begin{proof}
We proceed by induction on $\overline{(\lambda,\alpha)}\in\SP_{X,G}$. If $\overline{(\lambda,\alpha)}$ is minimal, then there is no $(\mu,\beta)\in \SQ_{X^{\lambda}_{\alpha},G^{\lambda}}$ such that $(X^{\lambda}_{\alpha})^{\mu}\subsetneq X^{\lambda}_{\alpha}$. Therefore,
\[
 \underline{\BPS}_{X,(\lambda,\alpha)}\cong \underline{\CP}_{\lambda}\cong (\pi_{(\lambda,\alpha)})_*\underline{\BoQ}_{X^{\lambda}_{\alpha}/(G^{\lambda}/G_{(\lambda,\alpha)})}\,.
\]

 Let now $(\lambda,\alpha)\in\SQ_{X,G}$ and we assume that $\underline{\CF}_{(\mu,\beta)}=\underline{\BPS}_{X,(\mu,\beta)}$ for any $(\mu,\beta)\in\SQ_{X,G}$ such that $(\mu,\beta)\prec(\lambda,\alpha)$ (for the order relation $\preceq$ defined in \S\ref{subsection:orderrelation}). By Theorem~\ref{theorem:sheafcohint}, the cokernel of the morphism of cohomologically graded mixed Hodge modules
 \[
  \bigoplus_{\substack{\tilde{(\mu,\beta)}\in\SP_{X^{\lambda}_{\alpha}}/W^{\lambda}\\(X^{\lambda}_{\alpha})^{\mu}\neq X^{\lambda}_{\alpha}}}\overline{\Ind}_{(\mu,\beta),(\lambda,\alpha)}\colon\bigoplus_{\substack{\tilde{(\mu,\beta)}\in\SP_{X^{\lambda}_{\alpha}}/W^{\lambda}\\(X^{\lambda}_{\alpha})^{\mu}\neq X^{\lambda}_{\alpha}}}(\imath_{(\lambda,\alpha),(\mu,\beta)})_*\underline{\BPS}_{X,(\mu,\beta)}\otimes\HO^*(\pt/G^{\lambda}_{(\mu,\beta)})\rightarrow \underline{\CH}_{X,(\lambda,\alpha)}
 \]
 is $\underline{\BPS}_{X,(\lambda,\alpha)}\otimes\HO^*(\pt/G_{(\lambda,\alpha)})$. By induction hypothesis, we also have a morphism
 \[
  \bigoplus_{\substack{\tilde{(\mu,\beta)}\in\SP_{X^{\lambda}_{\alpha}}/W^{\lambda}\\(X^{\lambda}_{\alpha})^{\mu}\neq X^{\lambda}_{\alpha}}}(\imath_{(\lambda,\alpha),(\mu,\beta)})_*\underline{\BPS}_{(\mu,\beta)}\otimes\HO^*(\pt/G^{\lambda}_{(\mu,\beta)})\rightarrow \underline{\CH}_{X,(\lambda,\alpha)}
 \]
whose cokernel is $\underline{\CP}_{(\lambda,\alpha)}\otimes\HO^*(\pt/G_{(\lambda,\alpha)})$. Since all objects appearing are semisimple, we have $\underline{\CP}_{(\lambda,\alpha)}\cong \underline{\BPS}_{X,(\lambda,\alpha)}$.
\end{proof}

\subsection{Self-duality}

\begin{lemma}
\label{lemma:selfdualitypf}
 For any $i\in\BoZ$, we have $\underline{\CH}^i(\pi_*\underline{\BoQ}_{X/G}^{\vir})\cong \BD\underline{\CH}^{i}(\pi_*\underline{\BoQ}_{X/G}^{\vir})\otimes\SL^i[2i]$.
\end{lemma}
\begin{proof}
We let $G'=G\times\BoC^*$ and let $\BoC^*$ act trivially on $X$. Since $\pi'_*\underline{\BoQ}_{X/G'}\cong\pi_*\underline{\BoQ}_{X/G}\otimes\HO^*(\pt/\BoC^*)$, it suffices to prove the statement for $X/G'$ by an easy induction on cohomological degrees. We use the fact that $\pi'$ can be approached by proper maps (Proposition~\ref{proposition:Gprimeapproachable}). For $n$ big enough,
\[
 \underline{\CH}^i(\pi'_*\underline{\BoQ}_{X/G'}^{\vir})\cong \underline{\CH}^{i-2d_n}((p_n)_*\IC(\overline{X_n}))\otimes\SL^{d_n}[2d_n]
\]
for some $d_n\in \frac{1}{2}\BoZ$ and projective algebraic variety $\overline{X_n}$. 
By the relative hard Lefschetz theorem (\cite[Theorem 3.1.16]{de2009decomposition} for perverse sheaves) applied to $p_n$, we have
\[
 \underline{\CH}^{i-2d_n}((p_n)_*\underline{\IC}(\overline{X_n}))\cong\underline{\CH}^{2d_n-i}((p_n)_*\underline{\IC}(\overline{X_n}))\otimes \SL^{i-2d_n}[2(i-2d_n)]
\]
By properness of $p_n$ and self-duality of $\underline{\IC}(\overline{X_n})$, we have
\[
 \underline{\CH}^{2d_n-i}((p_n)_*\underline{\IC}(\overline{X_n}))\cong \BD\underline{\CH}^{i-2d_n}((p_n)_*\underline{\IC}(\overline{X_n}))\,.
\]
By combining these two isomorphisms, we obtain
\[
 \underline{\CH}^{i-2d_n}((p_n)_*\underline{\IC}(\overline{X_n}))\cong \BD\underline{\CH}^{i-2d_n}((p_n)_*\underline{\IC}(\overline{X_n}))\otimes \SL^{i-2d_n}[2(i-2d_n)]\,.
\]
Tensoring both sides by $\SL^{d_n}[2d_n]$, we obtain the desired isomorphism.
\end{proof}

\begin{corollary}
\label{corollary:selfduality}
 For any $G$-equivariant function $f\colon X\rightarrow\BoC$, we have $\underline{\CH}^i(\pi_*\varphi_f^{p}\underline{\BoQ}_{X/G}^{\vir})\cong\BD\underline{\CH}^i(\pi_*\varphi_f^p\underline{\BoQ}^{\vir}_{X/G})\otimes\SL^i[2i]$.
\end{corollary}
\begin{proof}
 This follows from the fact that the vanishing cycle functor is perverse $t$-exact and commutes with proper pushforward. Therefore, it also commutes with maps that are approachable by proper maps.
\end{proof}

\begin{corollary}
\label{corollary:purebelowabove}
 The complex of monodromic mixed Hodge modules $\pi_*\varphi^p_f\underline{\BoQ}_{X/G}^{\vir}$ is pure if and only if it is pure below, if and only if it is pure above.
\end{corollary}
\begin{proof}
 This is a consequence of Corollary~\ref{corollary:selfduality}, since the Verdier duality changes the weights to their opposite.
\end{proof}

\subsection{A conjecture}
We make conjectures regarding the precise identification of the BPS complexes $\underline{\BPS}_{X,(\lambda,\alpha)}$.

\begin{definition}
 We let $\SQ_{X,G}^{\st}\subset \SQ_{X,G}$ be the subset of \emph{stable} pairs $(\lambda,\alpha)$, that is of $(\lambda,\alpha)\in \SQ_{X,G}$ such that there are closed orbits with finite stabilizer for the $G^{\lambda}/G_{(\lambda,\alpha)}$-action on $X^{\lambda}_{\alpha}$.
\end{definition}

\begin{proposition}
 A pair $(\lambda,\alpha)$ is stable if and only if $\dim X^{\lambda}_{\alpha}\cms G^{\lambda}=\dim X^{\lambda}_{\alpha}/(G^{\lambda}/G_{(\lambda,\alpha)})$.
\end{proposition}
\begin{proof}
 The equality of dimensions forces the union of closed $G^{\lambda}$-orbits inside $X^{\lambda}_{\alpha}$ to be a non-empty open subset and moreover, a general element in this union has finite stabilizer inside $G^{\lambda}/G_{(\lambda,\alpha)}$. This means precisely $(\lambda,\alpha)\in \SQ_{X,G}^{\st}$.
\end{proof}

\begin{conjecture}
\label{conjecture:BPSsheavesnofunction}
 We make the following conjectures\footnote{As this paper was being revised, a proof of Conjecture~\ref{conjecture:BPSsheavesnofunction} appeared in \cite{bu2025cohomology} under an orthogonality assumption for the action of $G$.} regarding the cohomologically graded mixed Hodge modules $\underline{\CP}_{(\lambda,\alpha)}\coloneqq \underline{\BPS}_{(\lambda,\alpha)}$:
 \begin{enumerate}
  \item $\underline{\CP}_{(\lambda,\alpha)}\neq 0$ if and only if $(\lambda,\alpha)\in \SQ_{X,G}^{\st}$.
  \item For $(\lambda,\alpha)\in \SQ_{X,G}$, $\underline{\CP}_{(\lambda,\alpha)}={\underline{\CH}}^{\dim G_{(\lambda,\alpha)}}(\pi_*\underline{\BoQ}_{X^{\lambda}_{\alpha}/G^{\lambda}}^{\vir})[-\dim G_{(\lambda,\alpha)}]$.
  \item There is a natural inclusion $\underline{\IC}(X^{\lambda}_{\alpha}\cms G^{\lambda})\otimes\SL^{\dim G_{(\lambda,\alpha)}/2}\hookrightarrow\underline{\CP}_{(\lambda,\alpha)}$, which is an isomorphism.
 \end{enumerate}
\end{conjecture}

\begin{lemma}
\label{lemma:finitegroupquotient}
 Let $G$ be a finite group acting on a smooth quasiprojective variety $X$ and $\pi\colon X\rightarrow X\cms G$ the good moduli space map. Then, $\pi_*\underline{\BoQ}_{X/G}\cong\underline{\BoQ}_{X\cms G}$.
\end{lemma}
\begin{proof}
 We let $H$ be the generic stabiliser, so that $G/H$ acts generically freely on $X$. The map $\pi$ factors as $X/G\xrightarrow{\pi'} X/(G/H)\xrightarrow{\pi''} X\cms G$. The map $\pi'$ is a $\pt/H$ gerbe, and so $\pi'_*\underline{\BoQ}_{X/G}\cong\underline{\BoQ}_{X/(G/H)}$. The map $\pi''$ is finite (and so small) and generically an isomorphism, and so $\pi''_*\underline{\BoQ}_{X/(G/H)}\cong\underline{\BoQ}_{X\cms G}$.
\end{proof}

\begin{lemma}
\label{lemma:finitestabilisers}
 Let $X$ be a smooth quasiprojective algebraic variety on which a group $G$ acts with finite stabilisers. Let $\pi\colon X/G\rightarrow X\cms G$ be a good moduli space map. Then $\pi_*\underline{\BoQ}_{X/G}\cong\underline{\BoQ}_{X\cms G}$.
\end{lemma}
\begin{proof}
 We have a natural adjunction map $\underline{\BoQ}_{X\cms G}\rightarrow \pi_*\underline{\BoQ}_{X/G}$. We can check that it is an isomorphism locally. Since $X/G$ is locally a finite group quotient \cite[Exercise 34]{behrend2004cohomology}, this follows from Lemma~\ref{lemma:finitegroupquotient}.
\end{proof}

\begin{corollary}
 Let $X$ be a smooth quasiprojective variety acted on by a reductive algebraic group $G$ and $H\subset G$ be a normal subgroup acting trivially on $X$ such that $G/H$ acts on $X$ with finite stabilisers. Then, $\pi_*\underline{\BoQ}_{V/G}\cong \underline{\BoQ}_{X\cms G}\otimes\HO^*(\pt/H)$.
\end{corollary}
\begin{proof}
 We can factor the good moduli space map as $X/G\xrightarrow{\pi'} X/(G/H)\xrightarrow{\pi''} X\cms G\cong X\cms (G/H)$. The morphism $\pi'$ is a $\pt/H$ gerbe and so $\pi'\underline{\BoQ}_{X/G}\cong \underline{\BoQ}_{X/(G/H)}\otimes\HO^*(\pt/H)$. Then, since $G/H$ acts on $X$ with finite stabilisers, we have $\pi''_*\underline{\BoQ}_{X/(G/H)}\cong \underline{\BoQ}_{X\cms G}$ by Lemma~\ref{lemma:finitestabilisers}. By combining these two isomorphisms, we obtain $\pi_*\underline{\BoQ}_{X/G}\cong \underline{\BoQ}_{X\cms G}\otimes\HO^*(\pt/H)$.
\end{proof}

As a step toward Conjecture~\ref{conjecture:BPSsheavesnofunction}, we prove the following.

\begin{proposition}
\label{proposition:ICsubobjectBPS}
 Let $(\lambda,\alpha)\in \SQ_{X,G}^{\st}$. Then, we have a canonical monomorphism $\underline{\IC}(X^{\lambda}_{\alpha}\cms G^{\lambda})\otimes\SL^{\dim G_{(\lambda,\alpha)}/2}\hookrightarrow\underline{\CP}_{(\lambda,\alpha)}$ of cohomologically graded mixed Hodge modules.
\end{proposition}
\begin{proof}
 By the approximation by proper maps, the complex of monodromic mixed Hodge modules $\pi_*\underline{\BoQ}_{X^{\lambda}_{\alpha}/G^{\lambda}}^{\vir}$ is semisimple (Proposition~\ref{proposition:semisimplicitypistar}). If $(\lambda,\alpha)\in \SQ_{X,G}^{\st}$, then the morphism $X^{\lambda}_{\alpha}/G^{\lambda}\rightarrow X^{\lambda}_{\alpha}\cms G^{\lambda}$ is generically a gerbe for a finite group extension of $G_{(\lambda,\alpha)}$. Therefore, the restriction of $(\pi_{(\lambda,\alpha)})_*\underline{\BoQ}_{X^{\lambda}_{\alpha}/G^{\lambda}}^{\vir}$ over a smooth open subset $U\subset X^{\lambda}_{\alpha}\cms G^{\lambda}$ is isomorphic to $\underline{\IC}(U)\otimes\HO^*(\pt/G_{(\lambda,\alpha)})$. By semisimplicity, $\underline{\IC}(X^{\lambda}_{\alpha}\cms G^{\lambda})\otimes\SL^{\dim G_{(\lambda,\alpha)}/2}$ is a direct summand of $(\pi_{(\lambda,\alpha)})_*\underline{\BoQ}_{X^{\lambda}_{\alpha}/G^{\lambda}}$.

 Moreover, the support of the image of any nontrivial induction morphism $\Ind_{(\mu,\beta),(\lambda,\alpha)}$ is contained in the strictly semistable locus inside $X^{\lambda}_{\alpha}\cms G^{\lambda}$. This means that $\underline{\IC}(X^{\lambda}_{\alpha}\cms G^{\lambda})\otimes\SL^{\dim G_{(\lambda,\alpha)}/2}$ does not appear in any of the nontrivial induction maps. Therefore, it is a direct summand of any direct sum complement of the image of the induction, and hence a direct summand of $\underline{\CP}_{(\lambda,\alpha)}$.
\end{proof}

\section{Sheafified cohomological integrality for critical affine $G$-varieties}

Let $X$ be a smooth affine $G$-variety for a reductive group $G$ and $f\colon X\rightarrow\BoC$ be a $G$-invariant function on $X$. Then, for any $(\lambda,\alpha)\in \SQ_{X,G}$, we let $f_{(\lambda,\alpha)}\colon X^{\lambda}_{\alpha}\rightarrow\BoC$ be the restriction of $f$, a $G^{\lambda}$-invariant function. We also denote $f_{(\lambda,\alpha)}\colon X^{\lambda}_{\alpha}\cms G^{\lambda}\rightarrow\BoC$ the induced regular function on the good moduli space. We define
\[
 \underline{\CH}_{X,f,(\lambda,\alpha)}\coloneqq (\pi_{(\lambda,\alpha)})_*\varphi^p_{X^{\lambda}_{\alpha},f_{(\lambda,\alpha)}}\underline{\BoQ}_{X^{\lambda}_{\alpha}/G^{\lambda}}^{\vir}\in\MMHM^{\BoZ}(X^{\lambda}_{\alpha}\cms G^{\lambda})\,.
\]

\subsection{Cohomological integrality for critical loci}

\begin{lemma}
\label{lemma:vanishingcyclecommutepf}
 Let $X$ be a smooth affine $G$-variety and $f\colon X\rightarrow \BoC$ a $G$-invariant function. Then, the pushforward by $\pi\colon X/G\rightarrow X\cms G$ commutes with the vanishing cycle functor.
\end{lemma}
\begin{proof}
 The same statement for $\pi'\colon X/(G\times \BoC^*)\rightarrow X\cms G$ where $\BoC^*$ acts trivially on $X$ is true since $\pi'$ is approachable by proper maps. The result follows from the fact that $\pi'_*\varphi_f\CF\cong \pi_*\varphi_f\CF\otimes\HO^*_{\BoC^*}$ for any constructible complex $\CF$ on $X/G$.
\end{proof}

\begin{theorem}
\label{theorem:cohintcriticallocus}
 Assume that $X$ is a smooth symmetric affine $G$-variety. Let $f\colon X\rightarrow \BoC$ be a $G$-invariant regular function. We have an isomorphism of cohomologically graded mixed Hodge modules
 \[
  \bigoplus_{\tilde{(\lambda,\alpha)}\in\tilde{\SP}_{X,G}}((\imath_{(\lambda,\alpha)})_*\varphi^p_{X^{\lambda}_{\alpha},f_{(\lambda,\alpha)}}\underline{\CP}_{(\lambda,\alpha)}\otimes\HO^*(\pt/G_{(\lambda,\alpha)}))^{\varepsilon_{X,(\lambda,\alpha)}}\rightarrow\underline{\CH}_{X,f}\,.
 \]
\end{theorem}
\begin{proof}
 This follows from Theorem~\ref{theorem:sheafcohint} by applying the vanishing cycle functor $\varphi_f^p$, using the fact that $\imath_{(\lambda,\alpha)}$ is a finite map (Lemma~\ref{lemma:finitemap}) and that $\pi_*$ commutes with the vanishing cycle functor by Lemma~\ref{lemma:vanishingcyclecommutepf}.
\end{proof}

We let $\underline{\BPS}_{X^{\lambda}_{\alpha},f_{(\lambda,\alpha)}}\coloneqq\varphi_{X^{\lambda}_{\alpha},f_{(\lambda,\alpha)}}^p\underline{\CP}_{(\lambda,\alpha)}$. Since the vanishing cycle functor is perverse exact, $\underline{\BPS}_{X^{\lambda}_{\alpha},f_{(\lambda,\alpha)}}$ is a cohomologically graded complex of monodromic mixed Hodge modules.

\section{Weak moment maps and local neighbourhood theorem for derived stacks with self-dual cotangent bundle}

\subsection{Weak moment map}
Let $X$ a smooth complex algebraic $G$-variety. We say that $X$ is weakly symplectic if there exists a $G$-equivariant isomorphism $\psi\colon \Tan^*X\rightarrow\Tan X$. We let $a\colon X\times G\rightarrow \Tan X$, $(x,\xi)\mapsto (x,X_{\xi}(x))$ be the infinitesimal action of $\Fg$ on $X$, where $X_{\xi}$ is the vector field on $X$ associated to $\xi$. We recall the notion of \emph{weak moment maps} following Halpern-Leistner \cite{halpern2020derived}.

\begin{definition}
\label{definition:weakmomentmap}
 A weak moment map for the $G$-action on the weak symplectic variety $X$ is a $G$-equivariant map $\mu\colon X\rightarrow\Fg^*$ such that there exists a $G$-equivariant isomorphism $\phi\colon X\times\Fg\rightarrow X\times\Fg$ such that the following diagram:
\begin{equation}\label{equation:weakmoment}\begin{tikzcd}
	{X\times\Fg} & {\Tan^*X} \\
	{X\times\Fg} & {\Tan X}
	\arrow["{\mathrm{d}\mu}", from=1-1, to=1-2]
	\arrow["\phi"', from=1-1, to=2-1]
	\arrow["\psi", from=1-2, to=2-2]
	\arrow["a", from=2-1, to=2-2]
\end{tikzcd}\end{equation}
commutes over $\mu^{-1}(0)$, where $\mathrm{d}\mu\colon (x,\xi)\mapsto \mathrm{d}\langle \mu(-),\xi\rangle(x)$.
\end{definition}

\subsection{Local neighbourhood theorem}

\begin{theorem}[{Neighbourhood theorem, \cite[Theorem 4.2.3]{halpern2020derived}}]
\label{theorem:HLneighbourhood}
 Let $\FX$ be a derived algebraic stack such that $\mathbb{L}_{\FX}\cong \mathbb{L}_{\FX}^{\vee}$ and let $\FX^{\cl}\rightarrow Y$ be a good moduli space. Let $x\in \FX(k)$ be a closed point  and $G=\Aut(\FX)(x)$ its stabilizer. This is a linearly reductive group. Then, there is a smooth affine $G$-scheme $X\coloneqq\Spec(R)$ with a $G$-fixed point $\tilde{x}\in X$ and a weak moment map $\mu\colon X\rightarrow\Fg^*$ and a map
 \[
  \pi_x\colon \FX'\coloneqq \mu^{-1}(0)/G\rightarrow\FX
 \]
taking $\tilde{x}$ to $x$ and where $\mu^{-1}(0)$ denotes the derived zero locus, which is strongly \'etale, that is we have a Cartesian diagram
\[\begin{tikzcd}
	{\FX'^{\cl}} & {\FX^{\cl}} \\
	{\mu^{-1}(0)^{\cl}\cms G} & Y
	\arrow[from=1-1, to=1-2]
	\arrow[from=1-1, to=2-1]
	\arrow["\lrcorner"{anchor=center, pos=0.125}, draw=none, from=1-1, to=2-2]
	\arrow[from=1-2, to=2-2]
	\arrow[from=2-1, to=2-2]
\end{tikzcd}\]
where $\mu^{-1}(0)^{\cl}$ is the classical truncation of the derived scheme $\mu^{-1}(0)$.
\end{theorem}

\subsection{Critical locus}
Let $X$ be a weak Hamiltonian $G$-variety and $f\colon \tilde{X}=X\times\Fg\rightarrow \BoC$ be the regular function obtained by contracting the weak moment map.

\begin{lemma}
\label{lemma:criticallocus}
 We assume that $\mu$ is an actual moment map, in particular that $\phi=\id_{X\times \Fg}$ in the diagram \eqref{equation:weakmoment}. We have $\crit(f)=\{(x,\xi)\in X\times\Fg\mid \mu(x)=0, \xi\in \Fg_x\}$ where $\Fg_x=\Lie(G_x)$ is the Lie algebra of the stabilizer of $x\in X$.
\end{lemma}
\begin{proof}
 We have $f(x,\xi)=\langle\mu(x),\xi\rangle$. Therefore, the vanishing of the partial differential with respect to $\xi$ gives the equation $\mu(x)=0$. In addition, the vanishing of the partial derivative with respect to $x$ is equivalent to the vanishing of $a(x,\xi)$ by the commutative diagram of Definition~\ref{definition:weakmomentmap}, that is $\xi\in\Fg_x$.
\end{proof}

We also need a version of Lemma~\ref{lemma:criticallocus} when $\mu$ is a weak moment map and not necessarily an actual moment map.
\begin{lemma}
\label{lemma:criticallocusweak}
Let $\mu\colon X\rightarrow\Fg^*$ be a weak moment map for a $G$-action on a weakly symplectic $G$-variety $X$. We have $\crit(f\circ \phi^{-1})=\{(x,\xi)\in X\times\Fg\mid \mu(x)=0,\xi\in\Fg_x\}$ where $\Fg_x=\Lie(G_x)$ is the Lie algebra of the stabilizer of $x\in X$.
\end{lemma}
\begin{proof}
Since $\phi$ is an isomorphism, $\crit(f\circ \phi^{-1})=\phi(\crit (f))$. Moreover, the differential of $f$ with respect to $\xi$ gives $\crit(f)\subset\mu^{-1}(0)\times\Fg$ (as for Lemma~\ref{lemma:criticallocus}). Since the diagram \eqref{equation:weakmoment} commutes over $\mu^{-1}(0)$, we have $\phi(x,\xi)=(x,\tilde{\phi}(x,\xi))$ for any $(x,\xi)\in\mu^{-1}(0)\times\Fg$. In particular, $\phi(\mu^{-1}(0)\times\Fg)=\mu^{-1}(0)\times\Fg$. Therefore, $\crit(f\circ\phi^{-1})\subset \mu^{-1}(0)\times\Fg$. We now examinate the derivative of $f\circ \phi$ with respect to $x$ for $x\in\mu^{-1}(0)$. We have
\[
\begin{aligned}
 \mathrm{d}_x(f\circ \phi)(x,\xi)=0&\iff \mathrm{d}_x\mu(\tilde{\phi}(x,\xi))=0\\
					    &\iff\psi^{-1}(a(x,\xi))=0\\
					    &\iff a(x,\xi)=0\,.
\end{aligned}
\]
Therefore, the vanishing of $\mathrm{d}_x(f\circ\phi)$ is the condition $\xi\in\Fg_x$.

\end{proof}

\section{Dimensional reduction}
\label{section:dimensionalreduction}
In this section, we recall the dimensional reduction isomorphisms \cite[Appendix A]{davison2017critical}, which relates a vanishing cycle sheaf on some space to the dualising sheaf of a different space. More generally, we recall the \emph{deformed} dimensional reduction theorem \cite{davison2022deformed}.

\begin{theorem}[Deformed dimensional reduction]
\label{theorem:deformeddimensionalreduction}
We let
\[
\begin{tikzcd}
 F&E& X\\
 \arrow["\pi_1",from =1-1, to=1-2]
 \arrow["\pi_2",from =1-2, to=1-3]
 \arrow["\pi",from = 1-1, to = 1-3, bend left=30]
\end{tikzcd}
\]
be a pair of a vector $F$ bundle over $E$ and a vector bundle $E$ over $X$. We let $s$ be a section of $F^*\rightarrow E$ (the dual of $F\rightarrow E$) and $g_0\colon E\rightarrow \BoA^1$ be a regular function on $E$. Contraction with $s$ gives a regular function $g_s$ on $F$. We consider the function $g=\pi_1^*g_0+g_s$ on $F$. We assume that there is a $\BoG_{\rmm}$-action on $F\rightarrow X$ with nonnegative weights and an action on $\BoA^1$ with positive weight such that $g$ is $\BoG_{\rmm}$-equivariant. We let $\overline{Z}\coloneqq \pi_1^{-1}(s^{-1}(0))\xrightarrow{\imath}F$. Then, the dimensional reduction morphism of functors
\[
 \pi_!\varphi_g\pi^*\rightarrow \pi_!\imath_*\varphi_{\pi_1^*g_0}\imath^*\pi^*
\]
is an isomorphism.
\end{theorem}

\begin{corollary}[Dimensional reduction]
\label{corollary:dimensionalreduction}
 If $E\rightarrow X$ is the trivial vector bundle, we recover the dimensional reduction theorem. More precisely, let $\pi\colon F\rightarrow X$ be a vector bundle, $s\colon X\rightarrow F^*$ be a section and $g\colon F\rightarrow \BoA^1$ be the associated regular function. We let $\overline{Z}\coloneqq \pi^{-1}(s^{-1}(0))$. Then, the dimensional reduction morphism
 \[
  \pi_!\varphi_g\pi^*\rightarrow\pi_!\imath_*\imath^*\pi^*\colon \CD(\MHM(X))\rightarrow\CD(\MMHM(X))
 \]
is an isomorphism.
\end{corollary}

\begin{corollary}
\label{corollary:dimredclassique}
In the setup of Corollary~\ref{corollary:dimensionalreduction}, we let $Z\coloneqq s^{-1}(0)$. We have an isomorphism
 \[
  \BD\underline{\BoQ}_{Z}^{\vir}\rightarrow \pi_*\varphi_g\underline{\BoQ}_F^{\vir}.
 \]
 where $\underline{\BoQ}_Z^{\vir}\coloneqq\underline{\BoQ}_Z\otimes\SL^{-(\dim X-\rank F)/2}$ and $\underline{\BoQ}_{F}^{\vir}=\underline{\BoQ}_F\otimes\SL^{-(\dim X+\rank F)/2}$.
\end{corollary}
\begin{proof}
 We apply Corollary~\ref{corollary:dimensionalreduction} to $\underline{\BoQ}_{X}$ and take the Verdier dual of the isomorphism obtained.
\end{proof}

\begin{lemma}
 Let $X$ be a weakly symplectic smooth affine algebraic variety. We let $\mu,\phi$ be as in Definition~\ref{definition:weakmomentmap}. We denote by $\pi\colon X\times \Fg\rightarrow X$ the projection. Then, $\pi_*\varphi_{f}\BoQ_{X\times \Fg}^{\vir}\cong \BD\BoQ_{\mu^{-1}(0)}^{\vir}\cong \pi_*\varphi_{f\circ \phi^{-1}}\BoQ_{X\times\Fg}^{\vir}$. 
\end{lemma}
\begin{proof}
 This follows from the fact that the critical loci of $f$ and $f\circ \phi^{-1}$ are contained in $\mu^{-1}(0)\times\Fg$ (Lemma~\ref{lemma:criticallocusweak}), that $\phi$ is an isomorphism (and so $\phi_*\varphi_{f}\BoQ_{X\times\Fg}^{\vir}=\varphi_{f\circ \phi^{-1}}\BoQ_{X\times\Fg}^{\vir}$), and the commutative diagram
 \[
\begin{tikzcd}
	{\mu^{-1}(0)\times\Fg} & {} & {\mu^{-1}(0)\times\Fg} \\
	& {\mu^{-1}(0)}
	\arrow["\phi", from=1-1, to=1-3]
	\arrow["\pi"', from=1-1, to=2-2]
	\arrow["\pi", from=1-3, to=2-2]
\end{tikzcd}\,.
 \]
 Indeed, we have $\pi_*\varphi_f\BoQ_{\tilde{X}}^{\vir}\cong\BD\BoQ_{\mu^{-1}(0)}^{\vir}$ by dimensional reduction.
\end{proof}

\section{Purity of the BPS sheaf for weak Hamiltonian reductions}

\subsection{A purity lemma}
For future reference, we state a purity lemma due to Davison.
\begin{lemma}[{\cite[Lemma 4.7]{davison2020bps}}]
\label{lemma:boundedbelow}
 Let $\FX$ be a finite type stack and $p\colon \FX\rightarrow\FY$ a morphism of stacks. Then, $p_*\BD\underline{\BoQ}_{\FX}$ is pure below and $p_!\underline{\BoQ}_{\FX}$ is pure above.
\end{lemma}

\subsection{Support Lemma for the BPS sheaf}
\subsubsection{Regular locus in Levi subgroups}
Let $G$ be a reductive group, $\Fg$ its Lie algebra. We recall the notion of regularity for elements of a Levi subalgebra $\Fl\subset\Fg$, the Lie algebra of the Levi subgroup $L\subset G$.

For $x\in\Fg$, we define $H_G(x)\coloneqq C_G(x^{\ssimp})=G^{x^{\ssimp}}$, the centraliser of the semisimple part of $x$.

If $X$ is a $G$-variety and $x\in\Fg$, we let $H_X(x)\coloneqq\{v\in X\mid x^{\ssimp}\cdot v=v\}=X^{x^{\ssimp}}$.

\begin{definition}
 An element $x\in \Fl$ is called \emph{$G$-regular} if $H_G(x^{\ssimp})\subset L$. We let $\Fl^{G-\reg}\subset\Fl$ be the subset of regular elements.
\end{definition}

We let $\Fg_{(L)}\coloneqq\{x\in\Fg\mid H_G(x) \text{ is conjugate to $L$}\}\subset \Fg$. We define the order relation $(L)\leq(M)$ if a conjugate of $L$ is a subgroup of $M$. We have an associated subvariety $\Fg_{\leq (L)}\subset \Fg$.

\begin{proposition}[{\cite[Proposition 2.11]{gunningham2018generalized}}]
\label{proposition:etaleGunningham}
 The morphism of stacks $d_{\Fl}\colon\Fl^{G-\reg}/L\rightarrow\Fg/G$ is \'etale with image $\Fg_{\leq(L)}$.
\end{proposition}

We will need a stronger notion of regularity. 

\begin{definition}
 Let $\lambda\in \Hom_{\BoZ}(\BoG_{\rmm},G)$ be a one-parameter subgroup and $X$ a $G$-variety. An element $x\in\Fg^{\lambda}$ is called $G,X$-regular if $H_G(x^{\ssimp})\subset G^{\lambda}$ and $H_X(x^{\ssimp})\subset X^{\lambda}$. We denote by $\Fg^{\lambda,G,X-\reg}$ the set of $G,X$-regular elements.
\end{definition}

\begin{lemma}
\label{lemma:superregopen}
 $\Fg^{\lambda,G,X-\reg}\subset\Fg^{\lambda,G-\reg}$ is a nonempty open subset.
\end{lemma}
\begin{proof}
 The subset $\Fg^{\lambda,G,X-\reg}$ is clearly nonempty, as it contains $\frac{\mathrm{d}}{\mathrm{d}t}\lambda(t)_{|t=1}\in\Fg$. We prove openness. We let $V\coloneqq\Tan X_{|X^{\lambda}}$ and $V=\Tan X^{\lambda}\oplus V'$ be a direct sum decomposition of $V$ as vector bundles over $X^{\lambda}$. Let $x\in\Fg^{\lambda}$. The infinitesimal action of $\Fg$ on $X$ gives a morphism of vector bundles $x\colon V\rightarrow V$. Then, $H_X(x)\subset X^{\lambda}$ if and only iff $x$ induces an invertible endomorphism of $V'$. This is an open condition on $x$.
\end{proof}

\begin{corollary}
 The morphism $\Fg^{\lambda,G,X-\reg}/G^{\lambda}\rightarrow\Fg/G$ is \'etale.
\end{corollary}
\begin{proof}
 This follows immediately from Proposition~\ref{proposition:etaleGunningham} and Lemma~\ref{lemma:superregopen}.
\end{proof}

\begin{lemma}
\label{lemma:cocharacterforelements}
 For any element $\xi\in \Fg$ there exists a one-parameter subgroup $\lambda\in\Hom(\BoG_{\rmm},G)$ such that $\xi\in \Fg^{\lambda,G,X-\reg}$.
\end{lemma}
\begin{proof}
 We let $\xi^{\ssimp}$ be the semisimple part of $\xi$. Up to $G$-conjugation, we can assume that $\xi^{\ssimp}\in\Ft$. We let $T'\subset Z(G^{\xi^{\ssimp}})$ be the maximal connected torus acting trivially on $X^{\xi^{\ssimp}}=\{x\in X\mid \xi^{\ssimp}\cdot x=x\}$.
 
 We write, for $\alpha\in\pi_0(X^{\xi^{\ssimp}}_{\alpha})$,
 \[
  \Tan X_{|X^{\xi^{\ssimp}}_{\alpha}}=\bigoplus_{\beta\in X_*(T')} (\Tan X_{|X^{\xi^{\ssimp}}_{\alpha}})_{\beta}\,.
 \]
 
 We define $\CW_0(\Tan X_{|X^{\xi^{\ssimp}}})=\{\beta\in X_*(T')\mid (\Tan X_{|X^{\xi^{\ssimp}}_{\alpha}})_{\beta}\neq 0 \text{ for some $\alpha\in\pi_0(X^{\xi^{\ssimp}})$}\}$. We define similarly $\CW_0(\Fg)\coloneqq \{\beta\in X_*(T')\mid \Fg_{\beta}\neq 0\}$.

 We let
 \[
  \CW_1\coloneqq\{\alpha\in\CW_0(\Tan X_{|X^{\xi^{\ssimp}}})\cup\CW_0(\Fg)\mid \langle \alpha,\xi^{\ssimp}\rangle=0\},
 \]
\[
 \CW_2\coloneqq\{\alpha\in\CW_0(\Tan X_{|X^{\xi^{\ssimp}}})\cup\CW_0(\Fg)\mid \langle \alpha,\xi^{\ssimp}\rangle\neq0\}.
\]
We consider the subset of the rational vector space $\Ft'_{\BoQ}=X_*(T')_{\BoQ}$:
\[
 Z\coloneqq \cap_{\alpha\in\CW_1}\alpha^{\perp}\setminus \cup_{\alpha\in\CW_2}\alpha^{\perp}.
\]
This subset is nonempty, as its base-change to $\BoC$ contains $\xi^{\ssimp}$. We pick $y\in Z$. For some integer $N\geq 1$, $Ny$ gives a cocharacter $\lambda\in X_*(T')$, and so a cocharacter $\lambda\in X_*(T)$, which satisfies the property of the Lemma.
\end{proof}

We let $X$ be a weak smooth affine symplectic variety with a $G$-action and a weak moment map $\mu\colon X\rightarrow\Fg^*$. We let $\tilde{f}\coloneqq f\circ\phi^{-1}\colon X\times\Fg\rightarrow\BoC$. Note that when $\mu$ is an actual moment map, $\tilde{f}=f$. Since $\phi$ is an isomorphism, the BPS sheaves $\underline{\BPS}_{\tilde{X},f}$ and $\underline{\BPS}_{\tilde{X},\tilde{f}}$ encode essentially the same data, but for $\tilde{f}$ we will make use of Lemma~\ref{lemma:criticallocusweak}.

Let $\lambda\in X_*(T)$ be a cocharacter. We define $\tilde{X}^{\reg}_{\lambda}\coloneqq X\times\Fg^{\lambda,G,X-\reg}$, so that we have a Cartesian square
\[\begin{tikzcd}
	{\tilde{X}_{\lambda}^{\reg}/G^{\lambda}} & {\tilde{X}/G} \\
	{\Fg^{\lambda,G,X\reg}/G^{\lambda}} & {\Fg/G}
	\arrow["{d_{\lambda}}", from=1-1, to=1-2]
	\arrow[from=1-1, to=2-1]
	\arrow[from=1-2, to=2-2]
	\arrow[from=2-1, to=2-2]
	\arrow["\lrcorner"{anchor=center, pos=0.125}, draw=none, from=1-1, to=2-2]
\end{tikzcd}\]
and by Proposition~\ref{proposition:etaleGunningham} and base-change, $d_{\lambda}$ is \'etale. By abuse, we also denote by $d_{\lambda}$ the morphism $\tilde{X}^{\reg}_{\lambda}\rightarrow\tilde{X}$. The morphism $\overline{d}_{\lambda}\colon \tilde{X}^{\reg}_{\lambda}\cms G^{\lambda}\rightarrow \tilde{X}\cms G$ induced on the good moduli spaces is also \'etale.

For $\alpha\in\pi_0(X^{\lambda})$, we let $\tilde{X}^{\lambda,\reg}_{\alpha}\coloneqq X^{\lambda}_{\alpha}\times\Fg^{\lambda,G,X-\reg}$ and $\tilde{X}^{\lambda\geq 0,\reg}_{\alpha}\coloneqq X^{\lambda\geq 0}_{\alpha}\times\Fg^{\lambda,G,X-\reg}$.

We have natural $G^{\lambda}$-equivariant maps $s_{(\lambda,\alpha)}\colon\tilde{X}^{\lambda\geq 0,\reg}_{\alpha}\rightarrow \tilde{X}_{\lambda}^{\reg}$ (closed immersion) and $q_{(\lambda,\alpha)}\colon \tilde{X}^{\lambda\geq 0,\reg}_{\alpha}\rightarrow\tilde{X}^{\lambda,\reg}_{\alpha}$. This is a vector bundle. We have a (not everywhere commutative) diagram of maps
\[\begin{tikzcd}
	& {\tilde{X}_{\lambda}^{\reg}} \\
	{} & {\tilde{X}^{\lambda\geq 0,\reg}_{\alpha}} & {\tilde{X}} & \BoC \\
	& {\tilde{X}^{\lambda,\reg}_{\alpha}}
	\arrow["{s_{(\lambda,\alpha)}}", from=2-2, to=1-2]
	\arrow["{d_{\lambda}}", from=1-2, to=2-3]
	\arrow[from=2-2, to=2-3]
	\arrow["{q_{(\lambda,\alpha)}}"', from=2-2, to=3-2]
	\arrow["\tilde{f}", from=2-3, to=2-4]
	\arrow["{\ell_{(\lambda,\alpha)}}", from=3-2, to=2-3]
	\arrow["{\tilde{f}_{(\lambda,\alpha)}}"', from=3-2, to=2-4]
	\arrow["u_{(\lambda,\alpha)}",bend left=60, from=3-2, to=1-2]
\end{tikzcd}\]
where $\ell_{(\lambda,\alpha)}$ is induced is the $(G^{\lambda},G)$-equivariant map $\tilde{X}^{\lambda,\reg}_{\alpha}\rightarrow\tilde{X}$ and $u_{(\lambda,\alpha)}$ is the closed inclusion induced by $X^{\lambda}_{\alpha}\subset X$.

\begin{lemma}
\label{lemma:criticalloci}
 We have an isomorphism of critical loci
 \[
  \crit(\tilde{f}\circ d_{\lambda})\cong \bigsqcup_{\alpha\in\pi_0(X^{\lambda})} u_{(\lambda,\alpha)}\crit(\tilde{f}_{(\lambda,\alpha)})\,.
 \]
\end{lemma}
\begin{proof}
 We have
 \[
  \crit(\tilde{f}\circ d_{\lambda})\cong \{(v,\xi)\in X\times \Fg^{\lambda,G,X-\reg}\mid \mu(v)=0, \xi\cdot v=0\}
 \]
 by Lemma~\ref{lemma:criticallocusweak} and the fact that after taking stacky quotients by the appropriate groups, $d_{\lambda}$ is \'etale. Lemma~\ref{lemma:criticallocusweak} also gives
\[
 \crit(f_{(\lambda,\alpha)})\cong\{(v,\xi)\in X^{\lambda}_{\alpha}\times\Fg^{\lambda,G,X-\reg}\mid \mu_{(\lambda,\alpha)}(v)=0, \xi\cdot v=0\}
\]
where $\mu_{(\lambda,\alpha)}\colon X^{\lambda}_{\alpha}\rightarrow(\Fg^{\lambda})^*$ is the induced weak moment map.

The statement follows then from the definition of $\Fg^{\lambda,G,X-\reg}$, since $\xi\cdot v=0$ implies $v\in X^{\lambda}$.
\end{proof}

\begin{lemma}[Factorisation property]
\label{lemma:factorizationproperty}
We have isomorphisms 
\[\varphi^p_{\tilde{f}\circ d_{\lambda}}\underline{\BoQ}_{\tilde{X}_{\lambda}^{\reg}/G^{\lambda}}^{\vir}\cong d_{\lambda}^*\varphi^p_{\tilde{f}}\underline{\BoQ}_{\tilde{X}/G}^{\vir}\cong \bigoplus_{\alpha\in\pi_0(X^{\lambda})}(u_{(\lambda,\alpha)})_*\varphi^p_{\tilde{f}_{(\lambda,\alpha)}}\underline{\BoQ}_{\tilde{X}^{\lambda,\reg}_{\alpha}/G^{\lambda}}^{\vir}\,.
\]

\end{lemma}
\begin{proof}
 The first isomorphism comes from the fact that $d_{\lambda}$ is \'etale (after taking quotients), and hence commutes with vanishing cycles (up to the appropriate shifts). The second isomorphism is more subtle. The proof is inspired by the proof of \cite[Lemma 4.5]{davison2016integrality}. We work equivariantly. We let $\BoG_{\rmm}$ act on $\tilde{X}^{\lambda\geq 0,\reg}_{\alpha}$ via $\lambda$. By $\BoG_{\rmm}$-equivariance, we have $\tilde{f}\circ d_{\lambda}\circ s_{(\lambda,\alpha)}=\tilde{f}_{(\lambda,\alpha)}\circ q_{(\lambda,\alpha)}$. By smooth pullback of vanishing cycle sheaves, we have an isomorphism
 \[
  \varphi^p_{\tilde{f}_{(\lambda,\alpha)}\circ q_{(\lambda,\alpha)}}\underline{\BoQ}_{\tilde{X}_{\alpha}^{\lambda\geq 0,\reg}}^{\vir}\cong (q_{(\lambda,\alpha)}^*\varphi^p_{\tilde{f}_{(\lambda,\alpha)}}\underline{\BoQ}_{\tilde{X}^{\lambda,\reg}_{\alpha}}^{\vir})\otimes\SL^{-\rank(q_{(\lambda,\alpha)})/2}\,.
 \]

 Therefore,
 \begin{equation}
 \label{equation:fact1}
  \varphi_{\tilde{f}_{(\lambda,\alpha)}}^p\underline{\BoQ}_{\tilde{X}^{\lambda,\reg}_{\alpha}}\cong (q_{(\lambda,\alpha)})_*\varphi_{\tilde{f}_{(\lambda,\alpha)}\circ q_{\lambda,\alpha}}^p\underline{\BoQ}_{\tilde{X}^{\lambda\geq 0,\reg}_{\alpha}}\otimes\SL^{\rank (q_{(\lambda,\alpha)})/2}\,.
 \end{equation}
 Then, we can show exactly as in \cite[Proof of Lemma 4.5]{davison2016integrality} that 
 \begin{equation}
 \label{equation:fact2}
  u_{(\lambda,\alpha)}^*\varphi_{\tilde{f}\circ d_{\lambda}}^p(\underline{\BoQ}^{\vir}_{\tilde{X}^{\reg}_{\lambda}}\rightarrow (s_{(\lambda,\alpha)})_*\underline{\BoQ}^{\vir}_{\tilde{X}^{\lambda\geq 0,\reg}_{\alpha}})
 \end{equation}
 is an isomorphism. 
 
 Namely, by Lemma~\ref{lemma:criticalloci}, a connected component of $\crit(\tilde{f}\circ d_{\lambda})$ is contained in the image of $u_{(\lambda,\alpha)}$. Therefore, by working analytically locally, we can choose coordinates $x_i,y_i,z_i$ around $x\in (\Tan \tilde{X}^{\reg}_{\lambda})_{|\tilde{X}^{\lambda}_{\alpha}}$ such that in a neighbourhood of $x\in \tilde{X}^{\reg}_{\lambda}$, $\tilde{f}$ can be written $\tilde{f}_{(\lambda,\alpha)}+\sum_{i}y_iz_i$. The result then follows by dimensional reduction \S\ref{section:dimensionalreduction}.
 
 The combination of \eqref{equation:fact1} and \eqref{equation:fact2} gives the Lemma, since $s_{(\lambda,\alpha)}$ is a closed immersion, and so in particular proper, and so commutes via vanishing cycle functors.
\end{proof}

Recall that $\Fg_0\subset\Fg$ is the Lie algebra of $G_0\coloneqq Z(G)\cap \ker(G\rightarrow\Aut(X))$.
\begin{lemma}[Support lemma]
\label{lemma:supportlemma}
 The support of the BPS sheaf $\underline{\BPS}_{\tilde{X},\tilde{f}}$ is contained inside $\mu^{-1}(0)\cms G\times\Fg_0\subset \tilde{X}\cms G$.
\end{lemma}
\begin{proof}
The proof uses the sheafified cohomological integrality isomorphism for critical representations of reductive groups (Theorem~\ref{theorem:critcohint}). We let $\overline{(v,\xi)}\in \tilde{X}\cms G$ be a point in $\supp(\underline{\BPS}_{\tilde{X},\tilde{f}})$ where $(v,\xi)\in\tilde{X}$ has closed $G$-orbit. We let $G_{(v,\xi)}\subset G$ be the stabilizer of $(v,\xi)$, a reductive group. By Lemma~\ref{lemma:criticallocus}, $\xi\in\Lie(G_{v,\xi})$. We let $\lambda\in X_*(T)$ be such that $\xi\in\Fg^{\lambda,X,G-\reg}$ (Lemma~\ref{lemma:cocharacterforelements}). We will prove that $\overline{\lambda}$ is trivial, that is $X^{\lambda}=X$ and $\Fg^{\lambda}=\Fg$, which implies that $\xi^{\ssimp}\in\Fg_0$. Assuming this for a moment, if $\xi^{\nil}$ is the nilpotent part of $\xi$, we have $\xi^{\nil}\in \Lie(G_{(v,\xi)})$. There exists a one-parameter subgroup $\mu\colon \BoG_{\rmm}\rightarrow G_{(v,\xi)}$ such that $\lim_{t\rightarrow 0}\mu(t)\xi^{\nil}\mu(t)^{-1}=0$. Then, $\lim_{t\rightarrow 0}\mu(t)\cdot (v,\xi)=(v,\xi^{\ssimp})$. Therefore, since the $G$-orbit of $(v,\xi)$ is closed, $\xi=\xi^{\ssimp}$. This concludes.

It remains to show that $\overline{\lambda}$ is trivial, that is $X^{\lambda}=X$. Assume it is not trivial. Recall that, by the very definition of $\underline{\CP}_{(\lambda,\alpha)}$, $\CP_{0}$ is a direct sum complement of the image of
\begin{equation}
\label{equation:cohcompl}
 \bigoplus_{\substack{(\mu,\alpha)\in \SQ_{X,G}\\ \overline{(\mu,\alpha)}\neq \overline{0}}}((\imath_{(\mu,\alpha)})_*\varphi_{\tilde{f}_{(\mu,\alpha)}}^p\underline{\CP}_{(\mu,\alpha)}\otimes\HO^*(\pt/G_{(\mu,\alpha)}))^{\varepsilon_{\tilde{X},(\mu,\alpha)}}\rightarrow \underline{\CH}_{\tilde{X},\tilde{f}}.
\end{equation}

By Lemma~\ref{lemma:factorizationproperty} and Theorem~\ref{theorem:critcohint} applied to the smooth critical $G^{\lambda}$-variety $(\tilde{X}^{\lambda}_{\alpha}, \tilde{f}_{(\lambda,\alpha)})$, where $\alpha\in \pi_0(\tilde{X}^{\lambda})$ is the connected component of $\tilde{X}^{\lambda}$ containing $(v,\xi)$, applying $\overline{d}_{\lambda}^*$ to \eqref{equation:cohcompl} gives an isomorphism.

Therefore, $\overline{d}_{\lambda}^*\underline{\BPS}_{\tilde{X},\tilde{f}}=0$ and provides us with a contradiction, since $\overline{(v,\xi)}$ is in the image of $\overline{d}_{\lambda}$ and in the support of $\underline{\BPS}_{\tilde{X},\tilde{f}}$.
\end{proof}

We are now able to deduce a support lemma for the BPS sheaf $\underline{\BPS}_{\tilde{X},f}$ for the function $f$.
\begin{lemma}
\label{lemma:sectionsupport}
 The BPS sheaf $\underline{\BPS}_{\tilde{X},f}$ is supported on $\mu^{-1}(0)\cms G\times\Fg_0$.
\end{lemma}
\begin{proof}
 We still denote by $\phi\colon (X\times \Fg)\cms G\rightarrow (X\times\Fg)\cms G$ the function induced by $\phi$. We have $\underline{\BPS}_{\tilde{X},\tilde{f}}=\phi_*\underline{\BPS}_{\tilde{X},f}$. By $G$-equivariance of $\phi$, we have $\phi(X\times Z(\Fg))=X\times Z(\Fg)$. Therefore, the support of $\underline{\BPS}_{\tilde{X},f}$ is contained in $\mu^{-1}(0)\cms G\times Z(\Fg)$. Moreover, $\underline{\BPS}_{\tilde{X},f}$ is $\Fg_0$-equivariant, as the function $f$ is $\Fg_0$-invariant. Since for any $x\in\mu^{-1}(0)\cms G$, the support of $(\underline{\BPS}_{\tilde{X},\tilde{f}})_{x\times\Fg}$ is at most $\dim \Fg_0$ (by Lemma~\ref{lemma:supportlemma}), then if $(\underline{\BPS}_{\tilde{X},\tilde{f}})_x\neq 0$, its support is exactly $\{x\}\times\Fg_0$. This implies that the support of $(\underline{\BPS}_{\tilde{X},\tilde{f}})_x$ contains $\{x\}\times \{0\}$. The support of $\underline{\BPS}_{\tilde{X},f}$ therefore contains $\phi^{-1}(x,0)=(x,0)$ and so, by $\Fg_0$-invariance, is equal to $\{x\}\times\Fg_0$. Therefore, by putting all fibers together, the support of $\underline{\BPS}_{\tilde{X},f}$ is contained in $\mu^{-1}(0)\cms G\times\Fg_0$.
\end{proof}

We let $\imath\colon (\mu^{-1}(0)\cms G)\times\Fg_{0}\rightarrow \tilde{X}\cms G$ be the natural inclusion.
\begin{lemma}
\label{lemma:dimredBPSsheaf}
 We assume that $\mu$ is a weak moment map. The cohomologically graded monodromic mixed Hodge module $\underline{\BPS}_{\tilde{X},f}$ can be written
 \[
  \underline{\BPS}_{\tilde{X},f}\cong \underline{\BPS}^{\red}_{\tilde{X},f}\boxtimes\underline{\BoQ}_{\Fg_{0}}
 \]
for some cohomologically graded mixed Hodge module $\underline{\BPS}_{\tilde{X},f}^{\red}\in\MMHM^{\BoZ}(\mu^{-1}(0)\cms G)$.
\end{lemma}
\begin{proof}
 The map $\pi\colon (\tilde{X}\coloneqq X\times \Fg)/G\rightarrow (\tilde{X}\cms G)$ is equivariant for the action of the additive group $\Fg_{0}$ on $X\times\Fg$ by translation on the $\Fg$-factor. Moreover, the function $f\colon\tilde{X}\rightarrow \BoC$ is invariant for this action. This proves that $\underline{\BPS}_{\tilde{X},f}$ is $\Fg_0$-equivariant. Using the support property of Lemma~\ref{lemma:supportlemma}, this gives the desired decomposition $\underline{\BPS}_{\tilde{X},f}\cong\underline{\BPS}_{\tilde{X},f}^{\red}\boxtimes\underline{\BoQ}_{\Fg_0}$.
\end{proof}

\subsection{Purity of the BPS sheaf}

\begin{proposition}
\label{proposition:dimredcohint}
We have
 \[
  \pi_*\BD\underline{\BoQ}_{\mu^{-1}(0)/G}^{\vir}\cong \bigoplus_{\tilde{(\lambda,\alpha)}\in\SP_{X,G}/W}((\imath_{(\lambda,\alpha)})_*\underline{\BPS}_{\tilde{X}^{\lambda}_{\alpha},f_{(\lambda,\alpha)}}^{\red}\otimes \HO^*(\pt/G_{(\lambda,\alpha)}))^{\varepsilon_{X,(\lambda,\alpha)}}.
 \]
\end{proposition}
\begin{proof}
  This follows by dimensional reduction Corollary~\ref{corollary:dimredclassique} and Lemma~\ref{lemma:dimredBPSsheaf}, using the critical cohomological integrality isomorphism Theorem~\ref{theorem:critcohint}.
\end{proof}

\begin{theorem}
\label{theorem:purityBPSsheaf}
 The cohomologically graded mixed Hodge module $\underline{\BPS}_{\tilde{X},f}$ is pure and therefore semisimple. 
\end{theorem}
\begin{proof}
 Since $\underline{\BPS}_{\tilde{X},f}$ is a direct summand of $\pi_*\varphi_f^p\underline{\BoQ}_{\tilde{X}/G}^{\vir}$, it suffices to prove that $\pi_*\varphi_f^p\underline{\BoQ}_{\tilde{X}/G}^{\vir}$ is pure. By Proposition~\ref{corollary:purebelowabove}, it suffices to prove that $\pi_*\varphi_f^p\underline{\BoQ}_{\tilde{X}/G}^{\vir}$ has weights bounded below by $0$. By critical cohomological integrality (Theorem~\ref{theorem:critcohint}), it suffices to show that $\underline{\BPS}_{\tilde{X},f}$ has weights bounded below by $0$. Since $\underline{\BPS}_{\tilde{X},f}\cong\underline{\BPS}_{\tilde{X},f}^{\red}\boxtimes\underline{\BoQ}_{\Fg_0}$ (Lemma~\ref{lemma:dimredBPSsheaf}), we can equivalently prove that $\underline{\BPS}_{\tilde{X},f}^{\red}$ has weights bounded below by $0$. Since $\underline{\BPS}_{\tilde{X},f}^{\red}$ is a direct summand of $\pi_*\BD\underline{\BoQ}_{\mu^{-1}(0)/G}^{\vir}$ (Proposition~\ref{proposition:dimredcohint}), it suffices to show that $\pi_*\BD\underline{\BoQ}_{\mu^{-1}(0)/G}^{\vir}$ has weights bounded below by $0$. By \cite[Lemma 4.7]{davison2020bps}, see also Lemma~\ref{lemma:boundedbelow}, $\pi_!\underline{\BoQ}_{\mu^{-1}(0)/G}^{\vir}$ has weights bounded above by $0$. By taking the Verdier dual, $\pi_*\BD\underline{\BoQ}_{\mu^{-1}(0)/G}^{\vir}$ has weights bounded below by $0$. This concludes.
\end{proof}

\begin{corollary}
\label{corollary:sheafsymplredpure}
 For any weakly Hamiltonian smooth affine $G$-variety $X$ with weak moment map $\mu\colon X\rightarrow\Fg^*$, the pushforward $\pi_*\BD\underline{\BoQ}_{\mu^{-1}(0)/G}\in\CD^+(\MMHM(X\cms G))$ is a pure weight $0$ complex of mixed Hodge modules.
\end{corollary}
\begin{proof}
By Proposition~\ref{proposition:dimredcohint}, the purity of $\pi_*\BD\underline{\BoQ}_{\mu^{-1}(0)/G}$ follows from that of $\underline{\BPS}_{\tilde{X}^{\lambda}_{\alpha},f_{(\lambda,\alpha)}}$ for any $(\lambda,\alpha)\in\SQ_{X,G}$ (Theorem~\ref{theorem:purityBPSsheaf}) and the purity of the mixed Hodge structure on $\HO^*(\pt/G_{(\lambda,\alpha)})$. We use the fact that $\imath_{(\lambda,\alpha)}$ is finite (Lemma~\ref{lemma:finitemap}) and so proper: it preserves purity.
\end{proof}

\begin{corollary}
 For any representation $W$ of $G$, the Borel--Moore homology $\HO^{\rmBM}_*(\mu^{-1}(0)/G,\BoQ)$ of the stacky Hamiltonian reduction of $\Tan^*W$ has pure mixed Hodge structure.
\end{corollary}
\begin{proof}
 By Corollary~\ref{corollary:sheafsymplredpure}, we know that $\pi_*\BD\underline{\BoQ}_{\mu^{-1}(0)/G}$ is a pure weight $0$ mixed Hodge module. We let $p\colon \mu^{-1}(0)\cms G\rightarrow\pt$. We let $\{0\}\rightarrow \mu^{-1}(0)\cms G$ be the inclusion of $\{0\}$. Since the map $\pi \colon\mu^{-1}(0)/G\rightarrow \mu^{-1}(0)\cms G$ is $\BoC^*$-equivariant for the contracting action of $\BoC^*$, then we have an isomorphism $\HO^{\rmBM}_*(\mu^{-1}(0)/G,\BoQ)=p_*\pi_*\BD\underline{\BoQ}_{\mu^{-1}(0)/G}\cong \imath^*\pi_*\BD\underline{\BoQ}_{\mu^{-1}(0)/G}$ \cite[Lemma 6.10]{davison2020bps}. Since $p_*$ preserves lower bounds of weights and $\imath^*$ preserves upper bounds of weight, we deduce that $\HO^{\rmBM}_*(\mu^{-1}(0)/G,\BoQ)$ has pure weight zero mixed Hodge structure.
\end{proof}

\begin{corollary}
 Let $X$ be a weakly symplectic weakly Hamiltonian $G$-variety with weak moment map $\mu$ (Definition~\ref{definition:weakmomentmap}). We let $\tilde{X}\coloneqq X\times\Fg$ and $f\colon\tilde{X}\rightarrow\BoC$, $(x,\xi)\mapsto\langle\mu(x),\xi\rangle$. Then, $\varphi_f^p\underline{\IC}(\tilde{X}\cms G)\in\MMHM(\tilde{X}\cms G)$ is a pure monodromic mixed Hodge module.
\end{corollary}
\begin{proof}
 We first prove the statement when $\tilde{X}$ has non-empty stable locus, that is the generic stabilizer for the $G/G_0$-action on $\tilde{X}$ is finite. Then, by Proposition \ref{proposition:ICsubobjectBPS}, $\underline{\IC}(\tilde{X}\cms G)$ is up to a Tate twist a direct summand of $\underline{\BPS}_{\tilde{X}}$. Therefore, $\varphi_f^p\underline{\IC}(\tilde{X}\cms G)$ is up to a Tate twist a direct summand of $\varphi_f^p\underline{\BPS}_{\tilde{X}}$. Since the latter is pure (Theorem~\ref{theorem:purityBPSsheaf}), then $\varphi_f^p\underline{\IC}(\tilde{X}\cms G)$ is also pure.
 
 For the general case, we let $\overline{L}$ be the conjugacy class of the neutral component of the generic stabilizer of points of $\tilde{X}\cms G$. Then, by \cite[Theorem~6.16]{popov1994invariant}, the natural map $p\colon\tilde{X}^L\cms (N_G(L)/L)\rightarrow \tilde{X}\cms G$ is finite and surjective. Moreover, closed $N_G(L)/L$-orbits in $\tilde{X}^L$ have finite stabilizer and so the stable locus is non-empty. By the first part of the proof, $\varphi_f^p\underline{\IC}(\tilde{X}^L\cms (N_G(L)/L))$ is pure. By finiteness of $p$, $p_*\underline{\IC}(\tilde{X}^L\cms (N_G(L)/L))$ contains $\underline{\IC}(\tilde{X}\cms G)$ as a direct summand. Since $p$ is proper, $p_*\underline{\IC}(\tilde{X}^L\cms (N_G(L)/L))=\varphi_f^pp_*\underline{\IC}(\tilde{X}^L\cms (N_G(L)/L))$ is pure. This concludes as the latter contains $\varphi_f^p\underline{\IC}(\tilde{X}\cms G)$ as a direct summand.
\end{proof}

\section{A purity conjecture of Halpern-Leistner}

In \cite[Conjecture 4.4]{halpern2015theta}, Halpern-Leistner made a conjecture regarding the purity of the Borel--Moore homology of $0$-shifted symplectic stacks having a proper good moduli space. This conjecture is motivated in particular by the study of the stacks of semistable sheaves on K3 surfaces. In this section, we deduce this conjecture from Theorems \ref{theorem:HLneighbourhood} and Corollary~\ref{corollary:sheafsymplredpure}.

\begin{theorem}
\label{theorem:HLconjecture}
Let $\FM$ be a $1$-Artin derived stack with proper good moduli space $\pi\colon\FM\rightarrow\CM$ having affine diagonal and such that $\mathbb{L}_{\FM}\cong\mathbb{L}_{\FM}^{\vee}$. Then, the mixed Hodge structure on $\HO^{\rmBM}_*(\FM)$ is pure.
\end{theorem}
\begin{proof}
We let $\pi'\colon\CM\rightarrow  \pt$ be the unique map. Since $\CM$ is proper, if $\pi_*\BD\underline{\BoQ}_{\FM}$ is pure, then $\pi'_*\pi_*\BD\underline{\BoQ}_{\FM}=\HO^{\rmBM}_{*}(\FM,\BoQ)$ is also pure. It suffices therefore to prove that $\pi_*\BD\underline{\BoQ}_{\FM}\in\CD^+(\MMHM(\CM))$ is pure. Since the purity of a complex of mixed Hodge modules can be checked \'etale locally, the purity of $\pi_*\BD\underline{\BoQ}_{\FM}$ follows by base-change from the Cartesian diagram of Theorem~\ref{theorem:HLneighbourhood} for all $x\in\CM$ and the purity of $(\pi_x)_*\BD\underline{\BoQ}_{\FM_x}$ which is Corollary~\ref{corollary:sheafsymplredpure}.
\end{proof}

\begin{corollary}
\label{corollary:purityzeross}
 Let $\FM$ be a $0$-shifted symplectic stack with proper good moduli space. Then, $\HO^{\rmBM}(\FM,\BoQ)$ carries a pure mixed Hodge structure.
\end{corollary}
\begin{proof}
 The $0$-shifted symplectic form gives an isomorphism of the cotangent complex with its dual. The result follows immediately from Theorem~\ref{theorem:HLconjecture}.
\end{proof}

\begin{corollary}
\label{corollary:purityHiggs}
 Let $G$ be a reductive algebraic group and $\FM_G^{\Dol}(C)$ the moduli stack of semistable $G$-Higgs bundles on the smooth projective curve $C$. Then, $\HO^{\rmBM}_*(\FM_{G}^{\Dol}(C),\BoQ)$ has pure mixed Hodge structure.
\end{corollary}
\begin{proof}
 We let $\CM_G^{\Dol}(C)$ be the moduli space. The good moduli space is $\pi\colon\FM_G^{\Dol}(C)\rightarrow\CM_G^{\Dol}(C)$. The derived enhancement of $\FM_G^{\Dol}(C)$ is a $0$-shifted symplectic stack \cite{pantev2013shifted}, and therefore, $\pi_*\BD\underline{\BoQ}_{\FM_G^{\Dol}(C)}^{\vir}\in\CD(\MHM(\CM_G^{\Dol}(C)))$ is a pure complex of mixed Hodge modules (as in the proof of Theorem~\ref{theorem:HLconjecture}). We let $B_G$ be the Hitchin base. We have a proper morphism $h\colon\CM_G^{\Dol}(C)\rightarrow B_G$ (the Hitchin map). By properness, $h_*\pi_*\BD\underline{\BoQ}_{\FM_G^{\Dol}(C)}^{\vir}\in\CD(\MHM(B_G))$ is pure. Moreover, we have a $\BoC^*$ action on $\FM_G^{\Dol}(C)$, $\CM_G^{\Dol}(C)$ rescaling the Higgs field, and a $\BoC^*$-action on $B_G$ such that $\pi$ and $h$ are $\BoC^*$-equivariant. Therefore, $h_*\pi_*\BD\underline{\BoQ}_{\FM_G^{\Dol}(C)}^{\vir}$ is $\BoC^*$-equivariant. By \cite[Lemma~6.10]{davison2020bps}, the fiber over $0\in B_G$ of $h_*\pi_*\BD\underline{\BoQ}_{\FM_G^{\Dol}(C)}^{\vir}$, which (up to a Tate twist) coincides with $\HO^{\rmBM}_*(\FM_G^{\Dol}(C))$, is pure.
\end{proof}

We also recover one of the main results of \cite{davison2021purity}.

\begin{corollary}
 Let $\FM$ be the stack of objects in a $2$-Calabi--Yau category. We assume that $\FM$ has a proper good moduli space. Then, $\HO^{\rmBM}_*(\FM,\BoQ)$ carries a pure mixed Hodge structure.
\end{corollary}
\begin{proof}
 The moduli stack of objects in a $2$-Calabi--Yau category is $0$-shifted symplectic. The statement follows from Corollary~\ref{corollary:purityzeross}.
\end{proof}

\appendix

\section{Mixed Hodge modules}
\label{appendix:MHM}
We refer to \cite{davison2020cohomological} for the formalism of monodromic mixed Hodge module we are using. We take a rather pedestrian approach. The theory of mixed Hodge modules has recently been extended to stacks in \cite{tubach2024mixed}.

\subsection{The negatively bounded derived category}
Let $X$ be an algebraic variety, or more generally a finite type separated complex scheme. We define the negatively bounded derived category of mixed Hodge modules as in \cite[Definition 3.3]{davison2016integrality}. We have truncation functors $\underline{\tau}^{\leq k}\colon\CD^{\rmb}(\MMHM(X))\rightarrow\CD^{\rmb}(\MMHM(X))$. For $k\leq l$, we have canonical isomorphisms $\underline{\tau}^{k}\rightarrow\underline{\tau}^{\leq k}\underline{\tau}^{\leq l}$. We let $\CD^{+}(\MMHM(X))\coloneqq \varprojlim_{k\in\BoZ}\limits \CD^{\rmb}(\MMHM(X))$. Therefore, an object $\underline{\CF}\in\CD^{+}(\MMHM(X))$ is a sequence $(\underline{\CF}_k)_{k\in\BoZ}$ of bounded complexes of mixed Hodge modules with isomorphisms $\underline{\CF}_k\cong \underline{\tau}^{\leq k}\underline{\CF}_l$ for any $k\leq l$, satisfying natural compatibilities.

\subsection{Cohomologically graded mixed Hodge modules}
\label{subsection:cohgradedmhm}
A complex of monodromic mixed Hodge modules $\underline{\CF}\in\CD^+(\MMHM(X))$ is called a \emph{cohomologically graded mixed Hodge module} if $\underline{\CF}$ is isomorphic to its total cohomology, that is $\underline{\CF}\cong\bigoplus_{i\in\BoZ}\underline{\CH}^i(\underline{\CF})[-i]$. We let $\MMHM^{\BoZ}(X)\coloneqq \prod_{i\in\BoZ}\MMHM(X)[-i]$ be the Abelian category of cohomologically graded mixed Hodge modules on $X$. If $\underline{\CF},\underline{\CG}\in\MMHM^{\BoZ}(X)$ and $f\colon \underline{\CF}\rightarrow\underline{\CG}$ is a morphism in $\CD^+(\MMHM(X))$, the cohomology $\underline{\CH}(f)=\prod_{i\in\BoZ}\underline{\CH}^i(f)\colon\underline{\CH}(\underline{\CF})\cong \underline{\CH}(\underline{\CG})$ is a morphism of cohomologically graded monodromic mixed Hodge modules. In the Abelian category $\MMHM^{\BoZ}(X)$, one can consider kernels, cokernels, images, coimages, making it convenient for our purposes.

\subsection{Pushforward from a quotient stack to a scheme}
Let $X$ be a $G$-variety, $Y$ a finite type complex scheme and $\pi\colon X/G\rightarrow Y$ a morphism from the quotient stack $X/G$. In this section, we recall how the pushforward by $p$ of a constructible complex or complex of monodromic mixed Hodge modules can be defined\footnote{While this paper was in preparation, the preprint \cite{tubach2024mixed} treating mixed Hodge modules on stacks with the six-functor formalism appeared.}. We let $\hdots\subset V_n\subset V_{n+1}\subset\hdots$ be representations of $G$ such that for any $n\in \BoN$, there is an open subvariety $U_n\subset V_n$ on which $G$ acts freely and $U_n\rightarrow U_n/G$ is a principal $G$-bundle. We assume that $U_n\subset U_{n+1}$. We assume that $\lim_{n\rightarrow+\infty}\limits\codim_{V_n}(V_n\setminus U_n)=+\infty$. We let $X_n\coloneqq X\times V_n$, $\overline{X_n}\coloneqq X_n/G$, $X'_n\coloneqq X\times U_n$, and $\overline{X'_n}\coloneqq (X\times U_n)/G$. We have maps
\[
 \overline{X'_n}\xrightarrow{\imath_n} \overline{X_n}\xrightarrow{p_n} X/G\xrightarrow{\pi} Y.
\]

We have natural morphisms $\overline{X'_n}\xrightarrow{f_n} \overline{X'_{n+1}}$, and $\overline{X_n}\rightarrow\overline{X_{n+1}}$ making the diagram
\[\begin{tikzcd}
	{\overline{X'_n}} & {\overline{X_n}} \\
	{\overline{X'_{n+1}}} & {\overline{X_{n+1}}} & {X/G}
	\arrow["{\imath_{n}}", from=1-1, to=1-2]
	\arrow["{f_{n}}"', from=1-1, to=2-1]
	\arrow[from=1-2, to=2-2]
	\arrow["{p_n}", from=1-2, to=2-3]
	\arrow["{\imath_{n+1}}", from=2-1, to=2-2]
	\arrow["{p_{n+1}}", from=2-2, to=2-3]
\end{tikzcd}\]
commute.

This gives an (adjunction) morphism
\[
 (p_{n+1}\circ\imath_{n+1})^*\underline{\CF}\rightarrow (f_n)_*(p_n\circ\imath_n)^*\underline{\CF}
\]
and by pushing forward by $\pi\circ p_{n+1}\circ \imath_n$, a morphism
\[
  (\pi\circ p_{n+1}\circ \imath_{n+1})_*(p_{n+1}\circ\imath_{n+1})^*\underline{\CF}\rightarrow (\pi\circ p_{n}\circ \imath_n)_*(p_n\circ\imath_n)^*\underline{\CF}
\]
and, for $i\in\BoZ$, a morphism
\begin{equation}
\label{equation:truncatiso}
 \tau^{\leq i}(\pi\circ p_{n+1}\circ \imath_{n+1})_*(p_{n+1}\circ\imath_{n+1})^*\underline{\CF}\rightarrow \tau^{\leq i}(\pi\circ p_{n}\circ \imath_n)_*(p_n\circ\imath_n)^*\underline{\CF}
\end{equation}

The following proposition is classical, using the fact that the codimension of $V_n\setminus U_n$ inside $V_n$ tends to infinity as $n\rightarrow+\infty$.

\begin{proposition}
Let $\underline{\CF}\in\CD^+(\MMHM(X/G))$. Then, for $i\in\BoZ$, $\underline{\tau}^{\leq i}(\pi\circ p_n\circ \imath_n)_*(p_n\circ\imath_n)^*\CF$ stabilizes as $n\rightarrow+\infty$, that is, for $n\gg0$, the morphism \eqref{equation:truncatiso} is an isomorphism.
\end{proposition}

We may therefore define $\pi_*\underline{\CF}\in\CD^+(\MMHM(Y))$ by the object $(\underline{\CF}_k)_{k\in\BoZ}$, where $\underline{\CF}_k\coloneqq\tau^{\leq k}(\pi\circ p_n\circ \imath_n)_*(p_n\circ\imath_n)^*\underline{\CF}$ for $n\gg0$ and the transition map are adjunction maps for the truncation functors.

\section{Examples}
\label{appendix:examples}
In this section, we provide a few examples illustrating the cohomological integrality isomorphism of Theorem~\ref{theorem:sheafcohint}. One goal is to give support to Conjecture~\ref{conjecture:BPSsheavesnofunction} and in particular to show that, at least on some examples, the vector spaces $\CP_{\lambda}$ associated to cocharacters $\lambda\in X_*(T)$ of a maximal torus $T$ of $G$ defined in \cite{hennecart2024cohomological} from a representation $V$ of a reductive group $G$ (see also \S\ref{section:reminderabsolute}) can be identified with $\IH^*(V^{\lambda}\cms G^{\lambda})$ as graded vector spaces (intersection cohomology is given the perverse shift \S\ref{subsection:conventions}).

\subsection{$\BoC^*\curvearrowright\Tan^*\BoC$}
We let $\BoC^*$ act on $\BoC$ naturally (with weight $1$) and consider the induced action of $V\coloneqq\Tan^*\BoC$. Then, $V\cms \BoC^*\cong \Spec(\BoC[xy])\cong\BoA^1$. Therefore, $\IH^*(V\cms G)=\BoQ[1]$ is a one-dimensional rational vector space placed in cohomological degree $-1$.

In this situation, by \cite{hennecart2024cohomological}, we have $\CP_{\lambda_0}=\BoQ[t]_{\deg \leq 0}\subset \BoQ[t]$ and with the virtual shifts, $\HO^*(V/G)\cong \BoQ[1]\oplus\BoQ[-1]\oplus\BoQ[3]\oplus\hdots$. The first piece of this decomposition corresponds to $\CP_{\lambda_0}$ and $\IH^*(V\cms G)$.

\subsection{$\BoC^*\curvearrowright\Tan^*(\BoC^2)$}
We let $\BoC^*$ act on $\BoC^2$ with weight $1$ and consider the induced action on $V\coloneqq\Tan^*\BoC^2$. Then, $V\cms \BoC^*\cong\Spec(\BoC[x,y,z,t]/\langle xy=zt\rangle)$. Then,
\[
 \IH^*(V\cms \BoC^*)\cong\BoQ[3]\oplus\BoQ[1].
\]
The calculation of this intersection cohomology group is explained in \cite{williamson2010decomposition}, given that the blow-up at the origin of $V\cms \BoC^*$ is a resolution of $V\cms \BoC^*$ with exceptional divisor $\BoP^1\times\BoP^1$. Alternatively, one can blow-down one copy of $\BoP^1$ and obtain a small resolution of $V\cms G$, in which the fiber over the singular point is $\BoP^1$. This gives an alternative computation of $\IH(V\cms \BoC^*)$. These algebraic varieties appear when studying the Atiyah flop.

The intersection cohomology $\IH^*(V\cms \BoC^*)$ coincides with $\CP_{\lambda_0}=\BoQ\oplus\BoQ\cdot t$ where by taking virtual shifts into account, $\BoQ$ sits in cohomological degree $-3$ and $\BoC\cdot t$ in cohomological degree $-1$.

\subsection{$\BoC^*\curvearrowright \BoC^g$}
\label{subsection:Cstarweight1}
We let $\BoC^*$ act on $\BoC^g$ with weight $1$. We consider the induced symmetric representation $V=\Tan^*\BoC^g$ of $\BoC^*$. By choosing a non-trivial stability parameter, we obtain a map
\[
 p\colon V\cms_{\chi}\BoC^*\coloneqq(\BoC^g\times (\BoC^g\setminus\{0\}))/\BoC^*\rightarrow V\cms \BoC^*.
\]
We have
\begin{enumerate}
 \item $\dim V\cms\BoC^*=2g-1$
 \item $p$ is an isomorphism over the complement of $0\in V\cms \BoC^*$
 \item The fiber over $0\in V\cms \BoC^*$ is $\BoP^{g-1}$.
\end{enumerate}
Therefore, the inequality
\[
 g-1<\frac{1}{2}(2g-1)
\]
implies that $p$ is a small map. Consequently, $\IH(V\cms\BoC^*)\cong \HO^*(V\cms_{\chi}\BoC^*)$. Moreover, the projection on the second factor gives a vector bundle $V\cms_{\chi}\BoC^*\rightarrow\BoP^{g-1}$. Therefore, $\dim \IH(V\cms \BoC^*)=g$.

\subsection{$\BoC^*\curvearrowright V$}
\label{label:VCstar}
One can generalise \S\ref{subsection:Cstarweight1} to any symmetric representation (or even weakly symmetric) of $\BoC^*$. Let $\BoC^*$ act on $V$. We let $V\cong \bigoplus_{k\in\BoZ}V_k$ be the weight decomposition. We assume that $V$ is symmetric, and so $\dim V_{k}=\dim V_{-k}$. The strong version of the sheafified cohomological integrality predicts that $\dim\IH(V\cms \BoC^*)=\sum_{k>0}\dim V_k$.

In this case, we have an identification of the intersection cohomology with the cohomology a a weighted projective space: $\IH(V\cms \BoC^*)\cong \HO^*(\BoP(\bigoplus_{k>0}V_k))$. Moreover, by \cite[Theorem 1]{kawasaki1973cohomology}, it is known that the cohomology of a weighted projective space coincides with the cohomology of the projective space of the same dimension, and hence $\dim\IH^*(V\cms \BoC^*)=\sum_{k>0}\dim V_k$.

\subsection{Strong sheafified cohomological integrality for symmetric representations of $\BoC^*$}
Let $V$ be a symmetric representation of $\BoC^*$. We let $V\cong\bigoplus_{k\in \BoZ}V_k$ be its weight space decomposition. The calculation presented in \S\ref{label:VCstar} gives the stronger following result. This is what we call \emph{strong cohomological integrality isomorphism}. The adjective \emph{strong} refers to the fact that the abstractly defined cohomologically graded mixed Hodge modules $\underline{\CP}_{(\lambda,\alpha)}$ (\S\ref{subsection:bpsmhm}) are identified with the intersection cohomology of the good moduli spaces (Conjecture~\ref{conjecture:BPSsheavesnofunction}).
\begin{theorem}
 We have an isomorphism of constructible complexes
 \[
  \IC(V\cms\BoC^*)\oplus \left(\IC(V^{\BoC^*})\otimes\HO^*(\pt/\BoC^*)\right)\xrightarrow{\sim}\pi_*\BoQ_{V/\BoC^*}^{\vir}.
 \]
\end{theorem}
\begin{corollary}
\label{corollary:cohintCstar}
 We have an isomorphism
 \[
  \IH(V\cms \BoC^*)\oplus\left(\IH(V^{\BoC^*})\otimes\HO^*(\pt/\BoC^*)\right)\rightarrow\HO^*(V/\BoC^*).
 \]
\end{corollary}
We have
\[
 \IH^*(V\cms\BoC^*)\cong \bigoplus_{j=0}^{\frac{\dim V-\dim V^{\BoC^*}}{2}-1}\BoQ[\dim V-1-2j],
\]
\[
 \IH(V^{\BoC^*})\otimes\HO^*(\pt/\BoC^*)\cong \bigoplus_{j\geq 0}\BoQ[\dim V^{\BoC^*}-2j].
\]
and
\[
 \HO^*(V/\BoC^*)\cong \bigoplus_{j\geq 0}\BoQ[\dim V-1-2j].
\]
Therefore, Corollary~\ref{corollary:cohintCstar} recovers the absolute integrality isomorphism of \cite{hennecart2024cohomological}.

\subsection{$\GL_2(\BoC)\curvearrowright (\Tan^*\BoC^2)^2$}

We consider the natural action of $\GL_2(\BoC)$ on $\BoC^2$. It induces an action of $\GL_2(\BoC)$ on $\Tan^*\BoC^2\cong\BoC^2\oplus (\BoC^2)^{\vee}$. We consider the diagonal action of $\GL_2(\BoC)$ on $(\Tan^*\BoC^2)^2$. It may be identified with the action on $\Tan^*\Mat_2(\BoC)$ induced by the multiplication of the left on $\Mat_2(\BoC)$. In more precise terms, it is given by
\[
 g\cdot (M,N)=(gM,Ng^{-1})\,,
\]
for $g\in\GL_2(\BoC)$, $M,N\in\Mat_2(\BoC)$, when identifying $\Mat_2(\BoC)^{\vee}$ and $\Mat_2(\BoC)$ by using the trace pairing.

We let $a,b,c,d,e,f,g,h$ be the coordinates on $\Mat_2(\BoC)^{\times 2}$, so that
\[
 M=\begin{pmatrix}
    a&b\\c&d
   \end{pmatrix}\,,
\quad
 N=\begin{pmatrix}
    e&f\\g&h
   \end{pmatrix}\,.
\]
We show that $\BoC[\Mat_2(\BoC)^{\times 2}]^{\GL_2(\BoC)}=\BoC[ea+fc,eb+fd,ga+hc,gb+hd]$ and moreover that the generators $ea+fc,eb+fd,ga+hc,gb+hd$ and $cf+dh$ are algebraically independent.

We consider the $\GL_2(\BoC)$-stable affine open subvariety $U\coloneqq\GL_2(\BoC)\times\Mat_2(\BoC)\subset \Mat_2(\BoC)^{\times 2}$. The morphism $U\rightarrow\Mat_2(\BoC)$, $(M,N)\mapsto NM$ induces an isomorphism $U/\GL_2(\BoC)\cong\Mat_2(\BoC)$. Therefore, the coordinate functions of $NM$ are exactly the $\GL_2(\BoC)$-invariant functions on $U$. Since these coordinates are defined on the whole of $\Mat_2(\BoC)$, the coordinate functions of $NM$ generate the ring of $\GL_2(\BoC)$-invariant functions on $\Mat_2(\BoC)$. We calculate:
\[
 \begin{pmatrix}
  e&f\\g&h
 \end{pmatrix}
 \begin{pmatrix}
  a&b\\c&d
 \end{pmatrix}
=
\begin{pmatrix}
ea+fc&eb+fd\\ga+hc&gb+hd 
\end{pmatrix}\,.
\]
This proves that $\BoC[\Mat_2(\BoC)^{\times 2}]^{\GL_2(\BoC)}=\BoC[ea+fc,eb+fd,ga+hc,gb+hd]$. We prove that the generators are algebraically independent. We let $\phi$ be a non-zero polynomial in four variables such that $\phi(ea+fc,eb+fd,ga+hc,gb+hd)=0$. We assume that $\phi$ is of lowest possible degree. Then, by setting $b=h=e=0$, we obtain $\phi(fc,fd,ga,0)=0$. Since $fc,fd,ga$ are clearly linearly independent, this means that $\phi=t\phi_2(x,y,z,t)$ and $\phi_2(ea+fc,eb+fd,ga+hc,gb+hd)=0$. This contradicts the minimality of $\deg \phi$.

We obtain
\[
 \Mat_2(\BoC)^{\times 2}\cms \GL_2(\BoC)\cong \BoA^4\,.
\]
In particular, $\IH^*(\Mat_2(\BoC)^{\times 2}\cms \GL_2(\BoC))\cong\BoQ$. If we take into account the perverse shift, $\IH^*(\Mat_2(\BoC)^{\times 2}\cms \GL_2(\BoC))\cong\BoQ$ is one-dimensional and concentrated in degree $-4$.

On the other hand, the calculation of \cite[\S6.4]{hennecart2024cohomological} for $g=2$ describes $\CP_{\lambda_0}$ as a one-dimensional space in degree $-4$ (when taking perverse shifts into account). Therefore, the strong integrality conjecture is confirmed in this case.

\subsection{$\rmSL_2(\BoC)\curvearrowright\Mat_{2\times n}(\BoC)$}
\label{subsection:examplemat2n}
I would like to thank Sebastian Schlegel Mejia and Dimitri Wyss for suggesting this example and providing ideas for the computation.

We let $\rmSL_2(\BoC)$ act on $\Mat_{2\times n}(\BoC)$ by multiplication on the left. This provides an example for which Conjecture~\ref{conjecture:BPSsheavesnofunction} is true but which is not under the scope of \cite{bu2025cohomology} since $\Mat_{2\times n}(\BoC)$ is not always an orthogonal representation of $\rmSL_2(\BoC)$, as shown by the following lemma.

\begin{lemma}
\label{lemma:mat2nnotorthogonal}
 The representation $\Mat_{2\times n}(\BoC)$ is an orthogonal representation of $\rmSL_2(\BoC)$ if and only if $n$ is even.
\end{lemma}
\begin{proof}
 We denote by $x_{i,j}$, $1\leq i\leq 2, 1\leq j\leq n$ the coordinates on $\Mat_{2\times n}(\BoC)$. If $n$ is even, then the quadratic form
 \[
  (x_{i,j})\mapsto \sum_{l=1}^{\frac{n}{2}}x_{1,2l-1}x_{2,2l}-x_{1,2l}x_{2,2l-1}
 \]
is a non-degenerate $\rmSL_2(\BoC)$-invariant quadratic form on $\Mat_{2\times n}(\BoC)$. It is the sum of the determinants of the $\frac{n}{2}$ $2\times 2$ matrices that form a $2\times n$ matrix.

If $n$ is odd, $\Mat_{2\times n}(\BoC)\cong \Mat_{2\times n-1}(\BoC)\oplus \BoC^2$ as $\rmSL_2(\BoC)$-representations. Since $\Mat_{2\times n-1}(\BoC)$ is orthogonal (as $n-1$ is even), and $\BoC^2$ is not orthogonal (by a direct check -- it has no non-zero $\rmSL_2(\BoC)$-invariant functions), then $\Mat_{2\times n}(\BoC)$ is not orthogonal by \cite[Lemma~4.1.5]{bu2025cohomology}.
\end{proof}

\subsubsection{Cohomological integrality isomorphism}
We let $n\geq 2$ and $V=\Mat_{2\times n}(\BoC)$. We let $T\subset\rmSL_2(\BoC)$ be the maximal torus of diagonal matrices. We have $\SP_V=\{\overline{\lambda}_0,\overline{\lambda}_1\}$ where
\[
 \begin{matrix}
  \lambda_0&\colon &\BoC^*&\rightarrow &T\\
  &&t&\mapsto&\begin{pmatrix}
               1&0\\0&1
              \end{pmatrix}
 \end{matrix}
\]
and
\[
 \begin{matrix}
  \lambda_1&\colon &\BoC^*&\rightarrow &T\\
  &&t&\mapsto&\begin{pmatrix}
               t&0\\0&t^{-1}
              \end{pmatrix}\,.
 \end{matrix}
\]
Since the action of $\rmSL_2(\BoC)$ on $V$ has no kernel, $G_0=1$ and by definition, $\CP_0$ is a direct sum complement of the induction morphism
\[
\begin{matrix}
 \Ind_{\lambda_1,\lambda_0}&\colon&\BoQ[x]&\rightarrow&\BoQ[x^2]\\
                           &&f&\mapsto& \frac{1}{2}\left(x^{n-1}f(x)+(-x)^{n-1}f(-x)\right)\,.
\end{matrix}
\]

\begin{proposition}
\label{proposition:Plambda}
We have $\CP_0=\BoQ[x^2]_{\deg<2(n-1)}$. In particular, $\dim \CP_0=\lfloor\frac{n}{2}\rfloor$.
\end{proposition}
\begin{proof}
 The image is $\Ind_{\lambda_1,\lambda_0}$ is $\BoQ[x^2]_{\geq2(n-1)}$. It is clear that $\BoQ[x^2]_{\deg<2(n-1)}$ is a direct sum complement. The dimension of this vector space is $\lfloor\frac{n-1}{2}\rfloor+1=\lfloor\frac{n}{2}\rfloor$.
\end{proof}

\subsubsection{Cohomology of the Grassmannian}
We let $\Gr(2,n)$ be the Grassmannian of $2$-planes inside $\BoC^n$.

\begin{proposition}
\label{proposition:cohomologygrassmannian}
 For $i\leq 2(n-2)$, we have
 \[
  \dim\HO^i(\Gr(2,n))=\left\{\begin{aligned}
                              &\lfloor\frac{i}{4}\rfloor+1\quad\text{ if $i$ is even}\\
                              &0\quad\text{ if $i$ is odd.}
                             \end{aligned}
\right.
 \]
\end{proposition}
\begin{proof}
 It is known that Grassmannians have no odd cohomology, and the dimension of $\HO^{2j}(\Gr(2,n))$ is the number of partitions of $j$ that fit inside a $2\times (n-2)$ rectangle. For $j\leq n-2$, the number of such partitions is $\lfloor\frac{j}{2}\rfloor+1$. If $i=2j$, we get $\dim\HO^i(\Gr(2,n))=\lfloor\frac{i}{4}\rfloor+1$.
\end{proof}

\subsubsection{Primitive cohomology of $\Gr(2,n)$}
\begin{proposition}
\label{proposition:primitivecohomologyGr2n}
 The primitive cohomology of $\Gr(2,n)$ is given by the following:
\[
 P^j(\Gr(2,n))\cong\left\{\begin{aligned}&\BoQ\quad\text{ if $j\equiv 0\pmod 4$ and $j\leq 2(n-2)$} \\
                    &0\quad\text{ otherwise.}
                   \end{aligned}\right.
\]

\end{proposition}
\begin{proof}
 We just have to compute the dimension of the graded pieces of the primitive cohomology. If we let $p_i=\dim P^{2i}(\Gr(2,n))$, we have
 \[
  b_{2i}=\sum_{j\leq i}p_j
 \]
for $i\leq 2(n-2)$. Therefore, $p_i=b_{2i}-b_{2(i-1)}$. By the formula of Proposition~\ref{proposition:cohomologygrassmannian}, we obtain $p_i=1$ if $i\equiv 0\pmod{4}$ and $p_i=0$ otherwise. This concludes the proof of the Proposition.
\end{proof}

\subsubsection{Intersection cohomology of $\Mat_{2\times n}(\BoC)\cms \rmSL_2(\BoC)$}

\begin{lemma}
\label{lemma:IHGITquotient}
 The GIT quotient $\Mat_{2\times n}(\BoC)\cms \rmSL_2(\BoC)$ is the affine cone over the Grassmannian $\Gr(2,n)$. In particular, the intersection cohomology $\IH^*(\Mat_{2\times n}(\BoC)\cms\rmSL_2(\BoC))$ coincides with the primitive cohomology of $\Gr(2,n)$.
\end{lemma}
\begin{proof}
 We let $\GL_2(\BoC)$ act on $\Mat_{2\times n}(\BoC)$ by left multiplication. We have $\Mat_{2\times n}(\BoC)\cms_{\det}\GL_2(\BoC)\cong \Gr(2,n)$. Therefore, $\BoC[\Mat_{2\times n}(\BoC)]^{\rmSL_2(\BoC)}$ is the homogeneous coordinate ring of $\Gr(2,n)$. Therefore, $\Mat_{2\times n}(\BoC)\cms\rmSL_2(\BoC)$ is the affine cone over $\Gr(2,n)$.
 
 The last statement of the Lemma follows from \cite[Example~2.2.1]{de2009decomposition}.
\end{proof}

\begin{corollary}
 There is an isomorphism of $\BoZ$-graded vector spaces $\IH^*(\Mat_{2\times n}\cms \rmSL_2(\BoC))\cong \CP_0$.
\end{corollary}
\begin{proof}
 It suffices to compare the dimensions of graded pieces of $\CP_0^j$ and $\IH^j(\Mat_{2\times n}\cms \rmSL_2(\BoC))$ This follows from the comparison of Proposition~\ref{proposition:primitivecohomologyGr2n}, Lemma~\ref{lemma:IHGITquotient} and Proposition~\ref{proposition:Plambda}. Indeed, we just have to compare $\CP_{0}$ with the primitive cohomology of $\Gr(2,n)$. We have $\dim\CP_0^j=1$ if $j\equiv 0\pmod{4}$ and $j<2(n-1)$. This is the same condition as $j\cong 0\pmod{4}$ and $j\leq 2(n-2)$ since $j$ has to be even. This finishes the proof.
\end{proof}

\printbibliography
\end{document}